\documentclass[11pt]{amsart} 

\usepackage[margin=1in]{geometry} 
\usepackage[utf8]{inputenc} 
\usepackage{quiver}
\usepackage{amsmath} 
\usepackage{amsthm} 
\usepackage{tcolorbox}
\usepackage{xcolor}
\usepackage{tensor}
\usepackage{inconsolata} 
\usepackage[shortlabels]{enumitem} 
\setlist[enumerate]{
     itemsep  = 0.2cm,
       label  = {\upshape (\arabic*)},
         ref  = \arabic*,
    leftmargin  = *
} 
\usepackage{xcolor} 
\usepackage{tikz-cd} 
\tikzcdset{arrow style=tikz, diagrams={>={Computer Modern Rightarrow[width=5pt,length=3pt]}}} 
\usepackage[colorlinks,allcolors=blue]{hyperref} 
\usepackage[nameinlink]{cleveref} 
\usepackage[final]{microtype}
\usepackage{todonotes} 
\usepackage{mathtools}

\usepackage{CJKutf8} 

\numberwithin{equation}{section} 
\newtheorem{theorem}{Theorem}[section]
\newtheorem{definition}[theorem]{Definition}
\newtheorem{proposition}[theorem]{Proposition}
\newtheorem{lemma}[theorem]{Lemma}

\newtheorem{remark}[theorem]{Remark}

\newcommand{\C}{\mathbb{C}}

\newcommand{\R}{\mathbb{R}}
\newcommand{\Z}{\mathbb{Z}}
\newcommand{\N}{\mathbb{N}}

\newcommand{\Sb}{\mathbb{S}}
\newcommand{\Hb}{\mathbb{H}}

\newcommand{\Af}{\mathbf{A}}
\newcommand{\Bf}{\mathbf{B}}
\newcommand{\Cf}{\mathbf{C}}

\newcommand{\Ef}{\mathbf{E}}

\newcommand{\Jf}{\mathbf{J}}

\newcommand{\Pf}{\mathbf{P}}
\newcommand{\Qf}{\mathbf{Q}}
\newcommand{\Rf}{\mathbf{R}}

\newcommand{\Zf}{\mathbf{Z}}

\newcommand{\bh}{\widehat{b}}
\newcommand{\bl}{\overline{b}}

\newcommand{\Acal}{\mathcal{A}}
\newcommand{\Bcal}{\mathcal{B}}

\newcommand{\Lcal}{\mathcal{L}}

\newcommand{\Ncal}{\mathcal{N}}
\newcommand{\Pcal}{\mathcal{P}}
\newcommand{\Qcal}{\mathcal{Q}}
\newcommand{\Rcal}{\mathcal{R}}

\newcommand{\pdv}{\partial}
\DeclareMathOperator{\supp}{supp}
\DeclareUnicodeCharacter{030C}{\v}
\newcommand{\wH}{\widehat{H}}
\newcommand{\wB}{\widehat{B}}
\newcommand{\ep}{\varepsilon}
\newcommand{\cg}{\mathring{g}}
\newcommand{\ck}{\mathring{k}}
\newcommand{\cQ}{\mathring{Q}}
\newcommand{\bsigma}{\overline{\sigma}}
\newcommand{\bpi}{\overline{\pi}}
\newcommand{\tsigma}{\widetilde{\sigma}}
\newcommand{\tpi}{\widetilde{\pi}}

\newcommand*\dd{\mathop{}\!\mathrm{d}}

\DeclareMathOperator{\Ric}{Ric}

\DeclareMathOperator{\tr}{tr}
\DeclareMathOperator{\dv}{div}

\DeclareMathOperator{\Hess}{Hess}
\DeclareMathOperator{\SO}{SO}
\DeclareMathOperator{\so}{\mathfrak{so}}

\definecolor{Lavender}{HTML}{E6E6FA}
\definecolor{Melon}{HTML}{FFF4E8}
\definecolor{Neutrogena}{HTML}{FFE8EC}
\definecolor{Cream}{HTML}{F9F0E6}

\usepackage[style=ieee,citestyle=numeric,backend=biber,sorting=nyt,url=false]{biblatex}

\AtEveryBibitem{\clearfield{issn}}	
\AtEveryCitekey{\clearfield{issn}}
\AtBeginBibliography{\small}		

\DeclareFieldFormat{postnote}{#1}

\DeclareFieldFormat[article,inbook,incollection,inproceedings,patent,thesis,unpublished]{title}{\textit{#1\isdot}} 

\DeclareFieldFormat{journaltitle}{#1}  

\DeclareFieldFormat[book,inbook,incollection,inproceedings]{series}{#1} 

\defbibheading{references}[\refname]{%
\section*{#1}%
\addcontentsline{toc}{section}{References} %
\markboth{#1}{#1}
}

\title{Initial Data Gluing for Asymptotically Hyperbolic Spacetimes}
\author{Aren Martinian}
\date{\today}

\begin{document} 

\begin{abstract}
We give a new proof of Corvino-Schoen-type gluing, also known as gluing up to linear obstructions, for asymptotically hyperbolic spacetimes satisfying appropriate smallness conditions in arbitrary dimension.
The results retrieve and improve on previous known smallness requirements on initial data to conduct such gluing, including the work of Chrusci\'el and Delay. 
Our technique is highly versatile and well-suited for gluing results in highly anisotropic geometric settings.
\end{abstract}

\maketitle 

\tableofcontents

\section{Introduction}\label{introduction}
In classical general relativity, one of the most important problems is to understand the flexibility and limitations of solution to the Einstein Constraint equations
\begin{equation}\label{eqn:constraints_intro}
\begin{cases}
R(g) + (\tr_g K)^2 - |K|^2_g = 2\Lambda \\
\nabla^i_g K_{ij} - \nabla_j \tr_g K = 0.
\end{cases}
\end{equation}
More precisely, a natural question to ask is in what way such a solution can be deformed on a compact or localized region while preserving the geometry elsewhere.
One approach to this question is on \emph{initial data gluing}, which prescribes two initial data sets solving \eqref{eqn:constraints_intro} in disjoint regions and asks whether one can prescribe a global solution which interpolates between the two.
One such approach is the celebrated Corvino-Schoen-type gluing, pioneered in \cite{corvino2000scalar} and later refined in \cite{CorvinoSchoen2006}. 
Schematically, the approach reduces solving \eqref{eqn:constraints_intro} to a linear problem of the form $Pu=f$, where $P$ is the linearized Einstein constraint operator containing the linearized scalar curvature and momentum constraint equations. 
If $P$ is surjective, or equivalently if $P^*$ has empty cokernel, one can construct a right-inverse directly.
However, in general this is not the case, and the approach instead constructs a right-inverse up to a finite dimensional cokernel. 

Since this seminal work, several gluing results have appeared in the literature. For the asymptotically flat case, i.e. $\Lambda=0$, we direct the interested reader to the excellent review articles \cite{Carlotto2021,Cortier2012}.
Some select work includes gluing an asymptotically flat initial data set to the exterior of initial data for Kerr as in \cite{ChruscielCorvinoIsenberg2011,ChruscielDelay2003}, gluing an interior region to an exterior region without changing each prescribed piece as in \cite{BieriChrusciel2017,Chrusciel2017}, and more recently the microlocal techniques of \cite{Hintz2023a,Hintz2023b,Hintz2024}.
In the hyperbolic setting $\Lambda = -d(d-1)/2$, the first main application of Corvino-Schoen-type gluing can be found in \cite{chrusciel2007gluing}.
In this paper, they consider an annulus of fixed radial size 
\[A_{\delta,4\delta} = \{\sqrt{1-8\delta} \leq R \leq \sqrt{1-2\delta}\}\] 
with coordinate $R$ on the Poincar\'e Ball. 
Instead of inverting the operator directly, they study a fourth-order operator of the form $P^*(\rho P u) =f$, where $\rho$ is a suitable weight-function decaying at infinity, and obtain estimates in weighted Sobolev spaces.
Quantitatively, for dimension $d$ they obtain a solution operator bound for the Hamiltonian constraint equations of $O(\delta^{-4})$, which when running a fixed point argument corresponds to requiring a smallness bound on the order of $\min\{\delta^8, \delta^{(8+d)/2}\}$ to close their estimates. 
Since then, \cite{Cortier2010} demonstrated the importance of understanding the finite dimensional obstruction space using charges, and \cite{ChruscielDelay2018} describe a Carlotto-Schoen type gluing approach as in \cite{CarlottoSchoen2016} that is localized to a particular angular sector.

More recently, the authors of \cite{mao2023,MaoTao2022} developed an approach in the asymptotically flat case that allows for the construction of a Bogovskii-type operator as a direct right inverse up to cokernel with desired regularity and support properties. 
Their general theory and framework were presented with examples for a much broader class of equations in \cite{isett2025integral} with more forthcoming. 
This powerful method was observed by \cite{ChruscielCogoNuetzi2025}, and they present a short note explaining that the Bogovskii-type operator construction is not unique to the asymotptically flat setting, but also holds in the asymptotically hyperbolic one, providing qualitative statements on initial data gluing for asymptotically hyperbolic spacetimes using the Bogovskii-type solution operator construction.

Inspired by the previous results, we apply the support preserving Bogovskii-type operator approach outlined in \cite{mao2023} and \cite{isett2025integral} to the problem of initial data gluing for asymptotically hyperbolic initial data $(g_{\mathrm{in}},k_{\mathrm{in}})$ and $(g_{\mathrm{out}},k_{\mathrm{out}})$ separated by a thin annular region
\[A_{2/\delta} = \{\log(2/\delta) \leq r \leq \log(8/\delta)\},\]
which is the hyperboloidal analogue of \cite{chrusciel2007gluing}.
Our work demonstrates that this approach provides more than an alternative proof of initial data gluing; it yields a stronger analytic framework than the existing literature.
Quantitatively, we provide sharper smallness assumptions on initial data by two orders of magnitude, obtaining a solution operator bound of $O(\delta^{-2})$ for the Hamiltonian constraint and $O(\delta^{-4})$ for the Momentum constraint. 
Consequently, we require a bound of $\delta^6 \ep$ for the Hamiltonian constraint equation and $\delta^{7}\ep$ for the Momentum constraint equation, where $\ep = \ep_c$.
This appears to be essentially optimal without additional regularity assumptions on the nonlinear corrections.
Moreover, the approach works in highly anisotropic settings, constructing solution operators at two different scales. 
For a more detailed statement of the main result, see Section~\ref{sec:firstgluing}.
The ability for this approach to handle anisotropic directions suggests that other types of gluing, such as Czimek-Rodnianski-type gluing of \cite{CzimekRodnianski2022} and localized bump constructions, are permissible with related suitable smallness assumptions.

\subsection{Acknowledgements}
I thank my advisor, Sung-Jin Oh, for introducing me to this problem and taking on a consultant role throughout this project.
His guidance, ideas, contributions, and words of encouragement for this project have been essential.
I would also like to thank the authors of \cite{CRMECA_2025__353_G1_29_0} and \cite{chrusciel2007gluing} for providing an excellent framework and a launching point of this project, and to Ning Tang and Zhongkai Tao for helpful discussions.
This research was supported NSF through DMS-1945615, DMS2452760, the Miller Research Professorship, and through NSF GRFP grant DGE-2146752.
\vspace{.2cm}

\textbf{Conflict of Interest}\;\;
The authors declares that they have no conflict of interest.
\vspace{.2cm}

\textbf{Data Availability and AI Statement}\;\;
No datasets were created or analyzed during the present work, so data sharing is not relevant to this article.
All ideas presented are either my own or from helpful discussions with my advisor. All of the writing, the figure, and the arguments in this paper are my own. Claude Opus $5.5$ was used at the very end of this project exclusively for proofreading in three places: the frequency based estimate Lemma~\ref{lem:fio_estimate}, the $\delta$-power numerology in Steps 4 and 5, and remark on comparison with \cite{chrusciel2007gluing}.

\section{Preliminaries}\label{sec:Preliminaries}
In this section, we make several definitions, some of which are analogous to those in \cite{chrusciel2007gluing}, and recall some important propositions of \cite{isett2025integral}.

Throughout our many different choices of coordinates, our hyperbolic base space will be the $d$-dimensional manifold $M = (r_0, \infty) \times N$, for some compact manifold $N$.
For our purposes $N$ can be taken to be the sphere $\Sb^{d-1}$, and in practice our sets will be generalized annuli with this setting in mind.
This manifold represents the spatial directions, and time will be considered separately.

\subsection{Hyperbolic Metrics}
\begin{definition}[\cite{chrusciel2007gluing}]
The \textbf{Kottler Metric} $b$ is defined as 
\begin{equation}
b \coloneqq \frac{dr^2}{r^2+\kappa} + r^2 \bh
\end{equation}
where $\kappa \in \{0, \pm 1\}$ and $\bh$ is a metric on $N$ that satisfies $\Ric(\bh) = \kappa(d-2)\bh$.
\end{definition}
Note that in the hyperbolic setting, the sectional curvature $\kappa$ is $-1$, and $\bh$ is the round metric on the unit sphere $\Sb^{d-1} = N$. 

For our purposes, it will be convenient to rewrite this metric in coordinates reflecting the compactification $\overline{M}$ for a suitable boundary-defining function $\rho$ which vanishes on $\partial M$ with $d\rho$ nonvanishing there.

In this coordinate system, our metric reads 
\begin{equation}
b \coloneqq \rho^{-2} \left( d\rho^2 + \frac{(4-k\rho^2)^2}{16}\bh \right)
\end{equation}
We can then define the conformal compactification 
\[\bl \coloneqq \rho^2 b.\]

\begin{definition}[\cite{chrusciel2007gluing}]
The metric $b$ is part of a larger class of metrics called \textbf{generalized Kottler metrics} defined as
\begin{equation}
b_m = \rho^{-2}\left( \left[ 1 - 2m\rho^n \left(1-k \left(\frac{\rho}{2} \right)^2 \right)^{2-n} \left(1+k \left(\frac{\rho}{2} \right)^2 \right)^{-2} \right]^{-1} d\rho^2 + \left(1 - k \left(\frac{\rho}{2}\right)^2 \right)^2 \bh \right)
\end{equation}

\end{definition}

We will assume that we are in the asymptotically ``time-symmetric case," namely that the second fundamental form $k$ tends to $0$ at infinity. 

Note that we will work with the \textbf{standard hyperbolic metric}, which on the Poincar\'e ball reads
\begin{equation}
b = \omega^{-2} \delta
\end{equation}
where $\delta$ is the Euclidean metric on $\R^d$ and the conformal factor $\omega(x) = \frac{1}{2}(1-|x|^2_{\delta})$.
\begin{remark}
As remarked in \cite{chrusciel2007gluing}, for generalized Kottler metrics $b_m$ we can only hope for a one-dimensional cokernel of the linearized scalar curvature with parameter given by the mass $m$.
In the standard hyperbolic metric setting, the complete $SO(d,1)$ symmetry gives us a full cokernel of dimension $(d+1)(d+2)/2$. 
This is related to the notion of complete integrability, discussed in \cite{isett2025integral};
in particular, the lack of a full dimensional cokernel means one could run precisely the same gluing argument as we will outline in Section~\ref{sec:firstgluing}, but the cost is that the solution operator bounds one obtains will be ineffective.
This means that by repeating the proof of Theorem~\ref{thm:gluing_linear} for $b_m$ but using Theorem 5.21 in \cite{isett2025integral} instead, one obtains a linear operator $\Rcal$ such that
\[P(S - \Rcal)f = f - w_A \langle Z^A, f\rangle\]
for $Z^A$ the cokernel element corresponding to the mass and $w_A \in C_c^{\infty}(A_{2/\delta})$ and an identical gluing result holds, with no quantitative information on the necessary smallness assumptions.
\end{remark}

Since we are in the hyperbolic setting, the Hamiltonian and Momentum constraint equations read 

\begin{equation}\label{eqn:constraints_einstein}
\begin{cases}
R(g) + (\tr_g k)^2 - |k|^2_g = 2\Lambda \\
\mathrm{div}(k-(\tr_g(k)g)) = 0
\end{cases}
\end{equation}
Here we take the cosmological constant $\Lambda = \frac{-d(d-1)}{2}$, so that if $k=0$ we have that $\Ric(g) = -(d-1)$ and $R(g) = -d(d-1)$.

From \cite{chrusciel2007gluing}, we read off the linearized scalar curvature and its adjoint as 
\begin{equation}\label{eqn:linearized_unrefined}
P_R h \coloneqq DR\vert_b [h] = -\Delta_b(\tr_b h) + \dv_b(\dv_b h) + (d-1)\tr_b h
\end{equation}
and
\begin{equation}
P_R^* f \coloneqq DR^*\vert_b [f] = -\Delta_b fb + \nabla_b \nabla_b f + (d-1) f b
\end{equation}

It will be important to introduce Schematic notation to reduce clutter and not have to consider terms explicitly:
\begin{definition}[Definition 2.1 in \cite{mao2023}]\label{defn:schematic}
We let $M_{\sigma}^{(\ast)}(u_1, \dots, u_k)$ and $N_{\sigma}^{(\ast)}(u_1, \dots, u_k)$ denote error terms given by linear combinations of contractions of $\{u_i\}$ with a smooth tensor field $A({\sigma})$ depending only on ${\sigma}$ to produce a scalar.
Here, the superscript refers to an index used to identify the terms.
Moreover for some fixed $M > 0$, we assume that if $|{\sigma}| \leq M$, then for all $n \in \N$ we have $|\partial^{(n)}_{\sigma} A({\sigma})| \lesssim_{M, n} 1$.
\end{definition}

Using Definition \ref{defn:schematic} and the substitions
\[\sigma_{ij} = g_{ij} - b_{ij} - b_{ij}\tr_b(g-b); \qquad \pi_{ij} = k_{ij} - b_{ij} \tr_b k,\] 
one can show that the equations reduce to
\begin{equation}\label{eqn:constraints}
\begin{aligned}
\begin{cases}
\nabla_i^{(b)} \nabla_j^{(b)} \sigma^{ij} - \tr_b \sigma = M^{(0)}_{\sigma}(\pi,\pi) + M^{(1)}_{\sigma}(\nabla^{(b)} \sigma, \nabla^{(b)} \sigma) + M^{(2)}_{\sigma}(\sigma, (\nabla^{(b)})^2 \sigma) \\
\nabla_i^{(b)} \pi^{ij} = N^{(1)j}_{\sigma}(\sigma, \nabla^{(b)} \pi) + N^{(0)j}_{\sigma}(\nabla^{(b)} \sigma, \pi)
\end{cases}
\end{aligned}
\end{equation}
where the superscript refers to the ambient metric; we will drop this for convenience.
This computation is carried out explicitly in Appendix~\ref{sec:derive_sigma_pi}.

\subsection{Hyperbolic Annuli}

Throughout the next section, it will be more convenient to consider the metric on the hyperboloid
\[b = dr^2 + \sinh^2(r)g_{\Sb^{d-1}}, \]
where $g_{\Sb^{d-1}}$ is the round metric on the $d-1$-dimensional unit sphere.
For the Poincar\'e Ball, the annular region we will be interested in will be the analog of the unit annulus in \cite{mao2023}, namely $\{\sqrt{1-8\delta} \leq \rho \leq \sqrt{1-2\delta}\}$.
For the coordinates on the hyperboloid, this transforms to 
\[A_{2/\delta} \coloneqq \left\{\log\left(\frac{2}{\delta}\right) \leq r \leq \log \left(\frac{8}{\delta} \right) \right\}.\]

It is also useful to have a thickening of this annulus to define cutoff functions which are identically $1$ on the interior annulus; we will denote this thickened annulus by
\[\widetilde{A}_{2/\delta} = \left\{ \log\left(\frac{1}{\delta}\right) \leq r \leq \log \left(\frac{16}{\delta} \right) \right\}.\]
Here, the scaling is of the order $\log(4)$ and $\log(16)$, which are roughly on the same order as $(1,2)$ and $(1/2, 4)$.
This is fixed, and independent of choice of $\delta$.

In addition to the thin annulus described above, we will consider a euclidean approximation and its rescaling:

\begin{definition}\label{defn:eucl_ann}
\begin{enumerate}
    \item The \textbf{euclidean annular approximation} $A$ of $A_{2/\delta}$ is defined as 
    \[\left\{x^1, \dots, x^d: \frac{1}{\delta} \leq ||x||_2 = \sqrt{x^ix_i} \leq \frac{1}{\delta} + 1 \right\}.\]
    \item The \textbf{unit cylinder rescaling} of $A$, denoted by $C$, is given by $z \in \Sb^{d-1} \times [0,1]$ equipped with the coordinate transformation
    \begin{equation}\label{eqn:stereo_proj}
    \begin{aligned}
    \Phi &\colon A_- \to C \\
    \Phi(x',x_d) &= \left(\frac{x'}{||x||_2 - x_d}, ||x||_2 - \frac{1}{\delta} \right) \coloneqq (z',z_d).
    \end{aligned}
    \end{equation}
\end{enumerate}
Here, $A_-$ denotes the bottom hemisphere, $x',z'$ refer to the angular components, while $x_d,z_d$ refer to the radial component.

Likewise, we will denote by $\widetilde{A}$ and $\widetilde{C}$ the $1$-thickenings of $A$ and $C$, respectively.
\end{definition}

\begin{remark}
Since $\sinh(\log(1/\delta)) \approx 1/\delta$, it is reasonable to assume that the thin euclidean annulus is a valid approximation of the hyperbolic annulus.
However, as we shall see, this thin annulus does not completely capture the full geometry of the setting.
Nevertheless, it will be convenient to consider the Euclidean approximation, as well as the unit scale regime which is obtained via stereographic projection of a hemisphere onto a plane tangent to the bottom of the sphere, followed by an angular rescaling. 
\end{remark}

\subsection{Sobolev Norms and Anisotropy}
Due to the complicated nature of the analysis in the hyperbolic setting and the strength of the conditions that would be need to placed on $(\sigma, \pi)$ for our results, it is more convenient to estimate quantities in the spaces $A$ and $C$.
In particular, we work with a standard, or isotropic, definition of $H^s$ on the unit cylinder $C$, and transport this to an anisotropic $\wH^s$ defined on $A$.
This is reflective of the actual geometry and standard Sobolev space on $A_{2/\delta}$, where each additional large-scale derivative costs a factor of $1/\delta$.

\begin{definition}[Anisotropic Sobolev Norm]\label{def:sobolev}
Let $u \colon A \to \R$ be a function. 
We define the \textbf{Anisotropic Sobolev Norm} $\wH^k$ on $A$ as
\begin{equation}
    ||u||_{\wH^k}^2 = \delta^{(d-1)} \sum_{\substack{\alpha' \in \N^{d-1}, \alpha_d \in \N \\ |\alpha'| + \alpha_d \leq k}} ||\delta^{-|\alpha'|} \pdv_{x'}^{\alpha'} \pdv_{x_d}^{\alpha_d} u ||^2_{L^2(A)}. 
\end{equation}
Similarly, for $T^{a_1, \dots, a_k}{}_{b_1, \dots, b_l}$ a $(k,l)$-tensor, we define the anisotropic Sobolev norm of $T$ as 
\begin{equation}\label{eqn:tensor_norm}
||T||_{\wH^k} = \sum_{(a_i,b_j)=(x_m',x_d),i,j,m} ||\delta^{(\#a_i = x_m') - (\#b_j = x_m')} T^{a_1, \dots, a_k}{}_{b_1, \dots, b_l}||_{\wH^k} 
\end{equation}
where $i \in \{1, \dots, k\}, j \in \{1, \dots, l\}$, and  $m \in \{1,\dots, d-1\}$.
In other words, the factor of $\delta$ for tensors is precisely the number of upper indices of the tensor which are angular components minus the number of lower indices which are angular components.
\end{definition}
 
This definition is completely consistent with the definition of $\wH^k$, as taking an angular derivative costs a factor of $\frac{1}{\delta}$, which removes one angular upper index of the tensor. 
As an example,
\[||\nabla_a \sigma^{ax_j'}||_{\wH^k} = \left|\left| \frac{1}{\delta} \pdv_{x_i'} \delta^2 \sigma^{x_i'x_j'} \right|\right|_{\wH^k} + ||\pdv_{x_d}\delta \sigma^{x_d x_j'}||_{\wH^k}, \]
and indeed both of the terms have a factor of $\delta$ corresponding to one raised angular index.

With this norm, a straightforward computation demonstrates the following:

\begin{proposition}\label{prop:equiv_of_norms}
Let $T$ be a $(k,l)$-tensor as above. 
Then we have 
\[||T||_{\wH^k(A)} = ||T||_{H^k(C)}.\]
\end{proposition}

Let us observe first that the factor of $\delta^{d-1}$ in front of the $\wH^k$ norm is due to the volume factor. 
Hence, we are treating $C$ as truly euclidean, so the $L^2$ norm is standard as well on $C$.
This will be extremely crucial when transferring bounds on the solution operator $S$.

In addition, for $s \in \R$, one can define the anisotropic Sobolev and Besov norms $\wH^s$ and $\wB^s_{p,r}$ as 

\begin{equation}\label{eqn:aniso_sobolev_s}
||u||_{\wH^s(A)}^2 = \delta^{d-1} \int (1+ \delta^{-2}|\xi'|^2 + |\xi_d|^2)^s |\hat{u}(\xi',\xi_d)|\dd\xi' \dd\xi_d
\end{equation}
and 
\begin{equation}\label{eqn:aniso_besov_s}
||u||_{\wB^s_{p,r}(A)}^2 = ||P_{0,0} u||_{L^p}^r + \sum_{\substack{j,k \geq 0 \\ (j,k) \neq (0,0)}} (1+2^{2j}+2^{2k})^{sr/2} ||P_{j,k}||^r_{L^p},
\end{equation}
where
\[\widehat{P_{j,k} u} = \phi\left(\frac{|\xi'|}{2^j \delta}\right)\phi\left(\frac{|\xi_d|}{2^j }\right). \]
Such spaces will be useful to state our results in full generality, and to run the necessary estimates.
One can check that Proposition \ref{prop:equiv_of_norms} extends to $\wH^s$ and $\wB^s_{p,r}$ as well.

\subsection{Charges and Conserved Quantities}
In order to have the flexibility to perform an iteration argument, we must first discuss the charges in the hyperbolic setting.

The generators of the cokernel of the constraint equations are given by the \emph{Poincar\'e} group, where we observe that we have a full cokernel in the standard hyperbolic case.
Explicitly, the kernel of the adjoint Hamiltonian constraint operator $P^*_R f = \nabla_i \nabla_j f - b_{ij}f$ is given by the restriction to $\Hb^d$ of the ambient coordinates $\{y^0, \dots, y^d\}$, while the kernel of the adjoint Momentum constraint operator $P^*_M Y_{\alpha}$ are given by:
\begin{itemize}
    \item \emph{boost generators} $y^0 \delta_{ji} + y^i \delta_{j0}$, and 
    \item \emph{rotation generators} $y^i\delta_{jk} - y^j \delta_{ik}$.
\end{itemize}
Here, indices run from $0$ to $d$, while $i,j,k \in \{1, \dots, d\}$, and the expressions are the components of the corresponding Killing vectors of Minkowski spacetime once restricted to the hyperboloid.
One verifies that there are exactly $(d+1)(d+2)/2$ such elements across both constraint equations.

A routine exercise in integration by parts and properties of the connection allows us to obtain the following lemma:

\begin{lemma}\label{lem:charges}
Fix $0 < r_0 < r_1$.
Then we have
\begin{equation}\label{eqn:hamiltoncharge}
\begin{aligned}
\int_{B_{r_1} \setminus B_{r_0}} \left( \nabla_i \nabla_j \sigma^{ij} \begin{pmatrix}
    y^0 \\
    \vdots \\
    y^d \\
\end{pmatrix} - b_{ij} \sigma^{ij} \begin{pmatrix}
    y^0 \\
    \vdots \\
    y^d \\
\end{pmatrix} \right) \sqrt{b} \dd y &= \int_{\pdv B_{r_1}} \left( \begin{pmatrix}
    y^0 \\
    \vdots \\
    y^d \\
\end{pmatrix} \nabla_j \sigma^{ij} - \sigma^{ij} \nabla_j \begin{pmatrix}
    y^0 \\
    \vdots \\
    y^d \\
\end{pmatrix} \right) \cdot \nu_i \dd S \\
&- \int_{\pdv B_{r_0}} \left( \begin{pmatrix}
    y^0 \\
    \vdots \\
    y^d \\
\end{pmatrix} \nabla_j \sigma^{ij} - \sigma^{ij} \nabla_j \begin{pmatrix}
    y_0 \\
    \vdots \\
    y^d \\
\end{pmatrix} \right) \cdot \nu_i \dd S, 
\end{aligned}
\end{equation}
for the Hamiltonian constraint, and 

\begin{equation}\label{eqn:momentumcharge}
\begin{aligned}
\int_{B_{r_1} \setminus B_{r_0}} \left( \nabla_i \pi^{ij} \begin{pmatrix}
    y^0 \delta_{jl} + y^l \delta_{j0} \\
    y^i \delta_{jk} - y^k \delta_{ji}
\end{pmatrix} \right) \sqrt{b} \dd y &= \int_{\pdv B_{r_1}} \left( \pi^{ij} \begin{pmatrix}
    y^0 \delta_{jl} + y^l \delta_{j0} \\
    y^i \delta_{jk} - y^k \delta_{ji}
\end{pmatrix} \right) \cdot \nu_i \dd S \\
&- \int_{\pdv B_{r_0}} \left(  \pi^{ij} \begin{pmatrix}
    y^0 \delta_{j}^l + y^l \delta_{j0} \\
    y^k \delta_{jl} - y^l \delta_{jk}
\end{pmatrix} \right) \cdot \nu_i \dd S, 
\end{aligned}
\end{equation}
in the Momentum constraint.

\end{lemma}

From this lemma, it is convenient to define the following charges, analogously to the flat case:
\begin{equation}\label{eqn:energy_charge}
\Ef[(\sigma, \pi); \pdv B_r] = \int_{\pdv B_{r}} \left( y^0 \nabla_j \sigma^{ij} - \sigma^{ij} \nabla_j y^0 \right) \cdot \nu_i \dd S 
\end{equation}

\begin{equation}\label{eqn:com_charge}
\Cf^j[(\sigma, \pi); \pdv B_r] = \int_{\pdv B_{r}} \left( y^j \nabla_j \sigma^{ij} - \sigma^{ij} \nabla_j y^j \right) \cdot \nu_i \dd S   
\end{equation}

\begin{equation}\label{eqn:mom_charge}
\Pf^l[(\sigma, \pi); \pdv B_r] = \int_{\pdv B_{r_1}} \left( \pi^{ij} \left(
    y^0 \delta_{j}^l + y^l \delta_{j0} \\
 \right) \right) \cdot \nu_i \dd S    
\end{equation}

\begin{equation}\label{eqn:angmom_charge}
\Jf^{kl}[(\sigma, \pi); \pdv B_r] = \int_{\pdv B_{r_1}} \left( \pi^{ij} \left(
    y^k \delta_{j}^l - y^l \delta_{j}^k \\
 \right) \right) \cdot \nu_i \dd S       
\end{equation}

All of these together will be a vector in $\R^{(d+1)(d+2)/2}$ given by $\Qf[(\sigma,\pi); \pdv B_r]$.

In addition, we want to define averaged charges as in \cite{mao2023}, but we need to be careful in this construction.
In the flat case, the natural scaling is simply dilation, but in the hyperbolic regime the natural scaling is translational symmetries. 

To this end, we define $\eta \in C_c^{\infty}((0,\infty))$ such that $\supp(\eta) \in (0, \log(4))$, $\int \eta(r')\dd r' = 1$ and $\eta_r(r') = \eta(r'+r)$.
Then we see that for $r =  \log(2/\delta)$, integrating against this smooth function integrates over the entire annulus $A_{2/\delta}$.
We can then define the averaged charge

\begin{equation}\label{eqn:avg_charge}
\Qf[(\sigma,\pi); A_{r}] = \int \eta_r(r') \Qf[(\sigma,\pi); \pdv B_{r'}]\dd r'.
\end{equation}

This allows us to state the following averaged version of Lemma \ref{lem:charges}:

\begin{lemma}\label{lem:avg_charges}
Fix $0<r_0 < r_1$, let $\eta_r$ be as above. Then for $\chi_{r_0,r_1}(r) = \int_{-\infty}^r (\eta_{r_0}(r') - \eta_{r_1}(r'))\dd r'$, we have the following:
\begin{equation}\label{eqn:avghamcharge}
\begin{aligned}
\int \chi_{r_0,r_1}(r) \left( \nabla_i \nabla_j \sigma^{ij} \begin{pmatrix}
    y^0 \\
    \vdots \\
    y^d \\
\end{pmatrix} - b_{ij} \sigma^{ij} \begin{pmatrix}
    y^0 \\
    \vdots \\
    y^d \\
\end{pmatrix} \right) \sqrt{b} \dd x &= 
\begin{pmatrix}
\Ef \\
\Cf^1 \\
\vdots \\
\Cf^d \\
\end{pmatrix}
[(\sigma,\pi); A_{r_1}] -
\begin{pmatrix}
\Ef \\
\Cf^1 \\
\vdots \\
\Cf^d \\
\end{pmatrix}
[(\sigma,\pi); A_{r_0}],
\end{aligned}
\end{equation}
for the Hamiltonian constraint, and 
\begin{equation}\label{eqn:avgmomentumcharge}
\begin{aligned}
\int \chi_{r_0,r_1}(r) \left( \nabla_i \pi^{ij} \begin{pmatrix}
    y^0 \delta_{jl} + y^l \delta_{j0} \\
    y^i \delta_{jk} - y^k \delta_{ji}
\end{pmatrix} \right) \sqrt{b} \dd x &=  
\begin{pmatrix}
\Pf^1 \\
\vdots \\
\Pf^d \\
\vdots \\
\Jf^{kl} \\
\vdots
\end{pmatrix}
[(\sigma, \pi); A_{r_1}] - 
\begin{pmatrix}
\Pf^1 \\
\vdots \\
\Pf^d \\
\vdots \\
\Jf^{kl} \\
\vdots
\end{pmatrix}
[(\sigma, \pi); A_{r_0}].
\end{aligned}
\end{equation}
for the Momentum constraint.
\end{lemma}

We make the observation that if we want a charge that is identically $1$ on $A_{2/\delta}$ and supported on $\widetilde{A}_{2/\delta}$, we can take $\chi_{\log(1/4\delta), \log(4/\delta)}(r)$.
Moreover, the precise analogous statements hold on $A$, on $C$, and their thickenings, where the thickenings are taken to be simply padding out by $1$ above and below. 
We will denote the corresponding averaging function in the unit picture by 
\begin{equation}\label{eqn:avg_C}
\chi^C_{-1,1} = \int_{-\infty}^r (\eta_{-1}^C(r') - \eta_{1}^C(r'))\dd r'.    
\end{equation}

Given that we will need to run a fixed point argument in the space of charges, it will make sense to consider a special collection of elements $Q \in \R^{(d+1)(d+2)/2}$.
Indeed, we will need to impose some conditions to run such an argument, hence we need to make the following definition as in \cite{mao2023}:

\begin{definition}[Definition 1.2 in \cite{mao2023}]\label{def:Q-admissible}
Consider a bounded open subset $\Qcal \subset \R^{(d+1)(d+2)/2}$ of charges and $s \in \R$ fixed. 
We say that a $(d+1)(d+2)/2$-parameter family $\{(\sigma_Q, \pi_Q)\}_{Q \in \Qcal}$ is a $\Qcal$-admissible family of annular initial data sets on $\widetilde{A}_r$ with Sobolev regularity $s$ if the following four conditions hold:

\begin{enumerate}
    \item For every $Q \in \Qcal$, $(\sigma_Q, \pi_Q) \in (H^1 \cap C^0) \times L^2(\widetilde{A}_r)$ and solves Equation~\eqref{eqn:constraints_einstein} in $\widetilde{A}_r$.
    \item For every $Q \in \Qcal$, $(\sigma_Q, \pi_Q) \in H^s \times H^{s-1} (\widetilde{A}_r)$.
    \item The map $\Qcal \to H^s \times H^{s-1}(\widetilde{A}_r)$ defined by $Q \mapsto (\sigma_Q, \pi_Q)$ is Lipschitz.
    \item For every $Q \in \Qcal$, one can recover $Q$ from $(\sigma_Q, \pi_Q)$ by integrating to obtain the average charge over the annulus $Q = \Qf[(\sigma_Q, \pi_Q); A_r]$.
\end{enumerate}
\end{definition}

\subsection{Graded Augmented Systems and Recovery on Curves}\label{subsec:grded_sys}
We end the preliminaries by explaining an important tool that will allow us to verify the recovery on curves condition of \cite{isett2025integral} and obtain a bound on the solution operator $S$, which is referred to as a \emph{graded augmented system.}

\begin{definition}[Definition 1.4 in \cite{isett2025integral}, Graded augmented system] \label{def:aug}
Let $P$ be an $r_{0} \times s_{0}$-matrix-valued differential operator on an open subset $U \subseteq \R^{d}$, with $m_{K}$ denoting the order of $(P^{\ast} \varphi)_{K}$ for each $K \in \{1, \ldots, s_{0}\}$. Given $m_{K}' \in \Z_{\geq 0}$ for each $K \in \{1, \ldots, s_{0}\}$ and an $\C^{r_{0}}$-valued function $(\varphi_{J})_{J \in \{1, \ldots, r_{0}}\}$ on $U$, consider $(\Phi_{\Af} = \Phi_{\Af}(y))_{\Af \in \Acal}$ (called \textbf{augmented variables}) satisfying the following properties:
\begin{enumerate}[label=($\Phi$-\arabic*), series=augcondition]
\item \label{hyp:aug1} {$(\Phi_{\Af})$ is an augmentation of $(\varphi_{J})$.} The index set $\Acal$ contains $\{1, \ldots, r_{0}\}$; moreover, $\Phi_{J} = \varphi_{J}$ for $J \in \{1, \ldots, r_{0}\} \subseteq \Acal$.
\item \label{hyp:aug2} {$(\varphi_{J}) \mapsto (\Phi_{\Af})$ is a differential operator.} There exist functions $c[\Phi_{\Af}]^{(\alpha, J)}$ on $U$, where $\alpha$ is a multi-index and $J \in \{1, \ldots, r_{0}\}$, such that
\begin{equation*}
\Phi_{\Af}(y) = c[\Phi_{\Af}]^{(\alpha, J)}(y) \pdv^{\alpha} \varphi_{J}(y),
\end{equation*}
for all $\Af \in \Acal$ and $y \in U$.
\item \label{hyp:aug3} {$\Phi$ and $P^{*} \varphi$ solve a first-order PDE (augmented PDE system).} There exist matrix-valued $1$-forms $\tensor{(\Bf_{i})}{_{\Af}^{\Af'}}$ and $\tensor{(\Cf_{i})}{_{\Af}^{(\gamma, K)}}$ on $U$, where $\gamma$ is a multi-index and $K \in \{1, \ldots, s_{0}\}$, such that
\begin{equation} \label{eq:aug-pde}
	\pdv_{i} \Phi_{\Af}(y) = \tensor{(\Bf_{i})}{_{\Af}^{\Af'}}(y) \Phi_{\Af'}(y) + \tensor{(\Cf_{i})}{_{\Af}^{(\gamma, K)}}(y) \pdv^{\gamma} (\Pcal^{\ast} \varphi)_{K}(y),
\end{equation}
for all $\Af \in \Acal$, $i \in \{1, \ldots, d\}$ and $y \in U$. Furthermore,
\begin{equation*}
\tensor{(\Cf_{i})}{_{\Af}^{(\gamma, K)}} = 0 \quad \hbox{ if } {|\gamma|} > m_{K}'.
\end{equation*}
\item \label{hyp:aug4} {\textbf{Graded structure.}} Define the \textbf{degree} $d_{\Af}$ of $\Phi_{\Af}$ by
\begin{equation*}
d_{\Af} \coloneqq - \max \{{|\alpha|} : c[\Phi_{\Af}]^{(\alpha, J)} \neq 0 \hbox{ for some } J.\}
\end{equation*}
In particular, $\varphi_{J}$ has degree $0$, i.e., $d_{J} = 0$ for $J \in \{1, \ldots, r_{0}\}$. We assume that:
\begin{itemize}
\item {\text{Graded structure for the augmented system.}}
\begin{align*}
\tensor{(\Bf_{i})}{_{\Af}^{\Af'}} &= 0 \qquad \hbox{ if } d_{\Af} >  d_{\Af'} + 1, \\
\tensor{(\Cf_{i})}{_{\Af}^{(\gamma, K)}} &= 0 \qquad \hbox{ if } d_{\Af} > -m_{K} - {|\gamma|} +1.
\end{align*}
\end{itemize}
\end{enumerate}

We write $N_{0} \coloneqq \max\{|{d_{\Af}}|\}_{\Af \in \Acal} + 1$ for the maximal degree\footnote{Here, $+1$ accounts for the derivative $\pdv_{i}$ on the LHS of \eqref{eq:aug-pde}.} that occurs in \eqref{eq:aug-pde}.

We call a collection $(\Acal, (\Phi_{\Af})_{\Af \in \Acal}, \tensor{(\Bf_{i})}{_{\Af}^{\Af'}}, \tensor{(\Cf_{i})}{_{\Af}^{(\gamma, K)}})$ satisfying \ref{hyp:aug1}--\ref{hyp:aug4} a (\textbf{graded}) \textbf{augmented system} for $P$.

Moreover, such a collection is called \textbf{completely integrable} if $\dim \ker P^* = \# \Acal$.
\end{definition}

A crucial fact about these augmented systems is the following: 
\begin{proposition}[Proposition 1.6 in \cite{isett2025integral}]\label{prop:aug_sys_RC}
Suppose that $P$ has a graded augmented system $(\Phi_{\Af})_{\Af \in \Acal}$, and that $\Bf_i$ and $\Cf_i$ and their derivatives are uniformly bounded on $U$.
Then $P$ satisfies the recovery on curves condition for any admissible family of curves $x(y,y_1,s)$, in the sense that one can linearly recover the value of $\phi$ at $y$ in terms of its jet at $y_1$ and the jet of $P^* \phi$ along $x(y,y_1,\cdot)$.
More precisely, one can apply Theorem \ref{thm:sol_op_5.15}.
\end{proposition}

\section{Gluing Up to Linear Obstructions}\label{sec:firstgluing}

In this section, we to prove a theorem analogous to Theorem 1.3 in \cite{mao2023} in the hyperbolic setting, namely gluing up to linear obstructions on a thin annulus.

\begin{theorem}[Gluing up to Linear Obstructions on the Thin Hyperbolic Annulus]\label{thm:gluing_linear}
Fix $s > \max\{\frac{d}{2},2\}$, and $\delta \ll 1$. Let $r$ denote the radial component of tensors $\cg,\ck, b$, and $\check{A},\check{B}$ denote the non-unit-normalized angular components.
Let $A$ be the euclidean annular approximation of $A_{2/\delta}$ and $\widetilde{A}$ be the euclidean annular approximation of $\widetilde{A}_{2/\delta}$.
Then we have $\ep_c = \ep_c(s) > 0$ and $M_c = M_c(s) > 0$ such that the following is true.
Let $(\cg,\ck)$ be a solution to Equation~\eqref{eqn:constraints_einstein} in $\wH^s \times \wH^{s-1}(\widetilde{A})$ such that in components, we have the hierarchy
\begin{equation}\label{eqn:metric_bounds}
\begin{cases}
||\cg^{\check{A}\check{B}}-b^{\check{A}\check{B}}||_{\wH^s(\widetilde{A})} \leq \delta^8\ep \\
||\cg^{r\check{A}}-b^{r\check{A}}||_{\wH^s(\widetilde{A})} \leq \delta^6\ep \\
||\cg^{rr}-b^{rr}||_{\wH^s(\widetilde{A})} \leq \delta^6\ep \\
\end{cases} 
\qquad
\begin{cases}
||\ck^{\check{A}\check{B}}||_{\wH^{s-1}(\widetilde{A})} \leq \delta^9\ep \\
||\ck^{r\check{A}}||_{\wH^{s-1}(\widetilde{A})} \leq \delta^7\ep \\
||\ck^{rr}||_{\wH^{s-1}(\widetilde{A})} \leq \delta^7\ep \\
\end{cases}
\end{equation}
Let $\cQ = \Qf[(\cg,\ck); A_{2/\delta}] \in \R^{(d+1)(d+2)/2},$ and let $\Qcal \subset \R^{(d+1)(d+2)/2}$ be a bounded open set such that 
\[\cQ \in \Qcal; \qquad B_{\delta^9 M_c \ep^2 (\cQ)} \subseteq \Qcal.\]
Let $\{(g_Q, k_Q)\}_{Q \in \Qcal}$ be $\Qcal$-admissible on $\widetilde{A}_{2/\delta}$ such that for all $Q, Q' \in \Qcal$, the bounds in Equation~\eqref{eqn:metric_bounds} hold for $(g_Q, k_Q)$, and 
\begin{equation}\label{eqn:lip_assump_Q}
||(g_Q - g_{Q'}, k_Q - k_{Q'})||_{\wH^s \times \wH^{s-1}(\widetilde{A})} \leq K|Q-Q'|.
\end{equation}

Then for $\ep < \ep_c$, $K \ep < \ep_c$, there is $(g,k) \in \wH^s \times \wH^{s-1}(\widetilde{A})$ and $Q \in \Qcal$ such that
\[\begin{cases}
||g^{\check{A}\check{B}}-b^{\check{A}\check{B}}||_{\wH^s \times \wH^{s-1}(\widetilde{A})} \leq \delta^6\ep \\
||g^{r\check{A}}-b^{r\check{A}}||_{\wH^s \times \wH^{s-1}(\widetilde{A})} \leq \delta^4\ep \\
||g^{rr}-b^{rr}||_{\wH^s \times \wH^{s-1}(\widetilde{A})} \leq \delta^4\ep. \\
\end{cases}
\qquad
\begin{cases}
||k^{\check{A}\check{B}}||_{\wH^{s-1}(\widetilde{A})} \leq \delta^{5}\ep \\
||k^{r\check{A}}||_{\wH^{s-1}(\widetilde{A})} \leq \delta^{5}\ep \\
||k^{rr}||_{\wH^{s-1}(\widetilde{A})} \leq \delta^{3}\ep \\
\end{cases} \]
In addition, $|Q - Q'|_H + \delta|Q-Q'|_M < \delta^6 M_c \ep^2$, and moreover $(g,k)$ agrees with $(\cg,\ck)$ and $(g_Q, k_Q)$ on either of the two boundaries in $\pdv \widetilde{A}$.
\end{theorem}

In addition to the theorem above, we have the following two results:

\begin{proposition}\label{prop:properties_linear}
Fix conventions as in Theorem~\ref{thm:gluing_linear}. Then
\begin{enumerate}
    \item (Lipschitz Continuity) The map $(\cg,\ck) \mapsto (g,k,Q)$ is Lipschitz as a map from the subset of $\wH^s \times \wH^{s-1}$ such that Equation~\eqref{eqn:metric_bounds} holds into $\wH^s \times \wH^{s-1} \times \R^{(d+1)(d+2)/2}$.
    \item (Persistence of Regularity). Fix $m \in \N $, $B = \frac{1}{\delta^6}$.
    Then assuming in addition that $\ep = \delta^{2}\ep_c$, for 
    \[(\cg,\ck) \in \wH^{s+m} \times \wH^{s+m-1} (\widetilde{A}) \quad \text{and} \quad \{(g_Q,k_Q)\}  \subseteq \wH^{s+m} \times \wH^{s+m-1}\] 
    we have $(g,k) \in \wH^{s+m} \times \wH^{s+m-1}(\widetilde{A})$, and 
    \[||(g-b)||_{ \wH^{s+m} \times \wH^{s+m-1}(\widetilde{A})} \leq B \cdot C \left(||(\cg,\ck)||_{\wH^{s+m} \times \wH^{s+m-1}(\widetilde{A})} + ||(g_Q-b,k_Q)||_{\wH^{s+m} \times \wH^{s+m-1}(\widetilde{A})}\right). \]
\end{enumerate}
\end{proposition}

Note that $\ep = \delta^2 \ep_c$ is unnecessary in the persistence of regularity statement; one can trade this dependency for a worse bound $B$, by working in the same $\delta$-dependent ball and building a rescaled norm from that.

Before explaining the proof, we make the following remark:

\begin{remark}
The bound on both the metric and the second fundamental form $||g||_{\wH^s}$, $||k||_{\wH^s}$ are independent on the Sobolev regularity $s$, and consequently the dimension. 
This is also true for the persistence of regularity argument, since the factor of  $1/\delta^6$ appears independent of any differentiation.
It is worth noting that existing results, such as \cite{chrusciel2007gluing}, had factors of $\delta$ depending on the dimension.
By Sobolev embedding, any dependence on regularity is equivalent to the dependence on the dimension by replacing $s > d/2$.
Converting our norm to their norm, we find that the overall power for the Hamiltonian constraint solution operator is $\delta^{-3}$, which is a genuine improvement of one factor of $\delta$ over their result; however, it is not optimal, as they explain in \cite[Remark 3.3]{chrusciel2007gluing} that $\delta^{-2}$ would be the sharp bound.
Converting their bound to our setting would require an estimate of $\delta^8 \epsilon$, at least in dimensions $n=5,6,7$, due to a different choice of geometric norms.
This is consistent with the fact that we are inverting the operator $P$ directly, rather than $P^*P$.
\end{remark}

Let us now adopt a procedure with which we will prove Theorem \ref{thm:gluing_linear}.
We will largely follow the proof strategy outlined in \cite{mao2023}, with some notable differences.

\begin{enumerate}
\item Write down the Einstein Constraint equation and the linearized operator, and find an appropriate reformulation in terms of a trace-adjusted difference $(\sigma_{ij}, \pi_{ij})$ sufficiently close to our asymptotically hyperbolic metric $(b, 0)$.
\item Determine the adjoint of the linearized operator, the cokernel, and hence the conserved quantities.
\item Verify, using \cite{isett2025integral}, that a Bogovskii-type operator $S$ exists so that $S$ is a right inverse of our linearized constraint $P$ up to a finite dimensional cokernel. 
First, we write down a completely integrable graded augmented system in the sense of Definition~\ref{def:aug} to verify the Recovery on Curves condition for an admissible family of curves $x(y,y_1,s)$ in the sense of Proposition~\ref{prop:aug_sys_RC}. 
To proceed, we use three solution operators:
\begin{enumerate}[label=\alph*)]
    \item For the first solution operator, for $y \in A_{2/\delta}$ we take paths to $y_1$ supported in a small neighborhood of the center of the annulus with the same angular component as $y$.
    We integrate the augmented system over these paths of geometric length $O(1)$ in both the radial and angular directions.
    Applying Theorem~\ref{thm:smth_ker_5.11} and Lemma~\ref{lem:b_eta_5.17} of \cite{isett2025integral} with a smoothing function $\eta_1$ around these points yields a solution operator $S^1 = (S^1_{f},S^1_Y)$ so that
    \[P S^1 h = h - B_{\eta_1} h \quad \text{where} \hspace{0.1cm} B_{\eta_1} \colon C_c^{\infty}(A_{2/\delta}) \to C_c^{\infty}(A_{2/\delta}) \hspace{0.1cm} \text{is smoothing}.\]
    This allows us to obtain the desired radial estimates with optimal regularity.
    \item For the second operator solving $Pu = B_{\eta_1} h$, we first take a radial path followed by some fixed angular path in some direction of geometric length $O(1/\delta)$; we then use a rotation to move an arbitrary point $y_1$ onto the great circle of this coordinate. 
    Moreover, it can be arranged that 
    \[PS_{\eta_2}h = B_{\eta_1}h - B_{\eta_2 }B_{\eta_1}h\]
    where $B_{\eta_2}$ is supported in patch of size $O(1/\delta)$ in $A_{2/\delta}.$
    \item Finally, the third operator solves $Pu = B_{\eta_2 }B_{\eta_1}h$, where now the source is confined to two unit-sized balls; a modification of the stereographic projection can be chosen to contain both.
    Another application of the same results for the second operator reduces the problem to a finite dimensional cokernel.
\end{enumerate}
\item Begin with a Naive approximation and construct an equation for the correction. Apply modified Moser and product estimates on $H^s(C)$ to run Picard Iteration such that for each $Q$, $(\sigma_Q, \pi_Q)$ solve the equations up to the cokernel.
\item We conclude by running a second fixed point argument on the charges so that the cokernel term vanishes.
\end{enumerate}

\begin{proof}[Proof of Theorem \ref{thm:gluing_linear}]
\textbf{Step 1}
We begin by recalling that our linearized constraint equations are
\begin{equation}
\begin{aligned}
\begin{cases}
\nabla^i_b \nabla^j_b \sigma_{ij} - \tr_b \sigma = M^{(0)}_{\sigma}(\pi,\pi) + M^{(1)}_{\sigma}(\nabla_b \sigma, \nabla_b \sigma) + M^{(2)}_{\sigma}(\sigma, \nabla_b^2 \sigma) \\
\nabla^i_b \pi_{ij} = N^{(1)j}_{\sigma}(\sigma, \nabla_b \pi) + N^{(0)j}_{\sigma}(\nabla_b \sigma, \pi)
\end{cases}
\end{aligned}
\end{equation}
where we have trace shifted to remove the Laplacian.

\textbf{Step 2}

The expression for the adjoint operators are
\begin{equation}
P^*_R f = \Hess_b(f) - bf; \qquad P^*_M Y= - \frac{1}{2} \Lcal_b Y
\end{equation}
given explicitly by 
\[(P^*_R f)^{ij} = \nabla^i_b \nabla^j_b(f) -fb^{ij}; \qquad (P^*_M Y)^{ij} = - \frac{1}{2} \left(\nabla_b^i Y^j + \nabla^j_b Y^i \right) = - \nabla_b^{(i} Y_{\vphantom{b}}^{j)} \]
This is a standard computation, for example see Threorem 1.14 of \cite{isett2025integral}.

Recall that the principal symbol of the linearized scalar curvature is $(p^*)_{ij}(\xi) = -\xi_i \xi_j$, which is complex elliptic and hence satisfies the finite cokernel condition.
Indeed, it is well-known that the cokernel has dimension at most $d+1$, for example see (\cite{corvino2000scalar}, \cite{isett2025integral}). 
For the momentum constraint, the principal symbol is $(p^*)_{ij}^k (\xi) = \frac{-i}{2}(\xi_i \delta^k_j + \xi_j \delta^k_i)$.
Notice that being $0$ is exactly the condition of being a Killing vector field.
This has kernel of dimension at most $\frac{d(d+1)}{2}$, and the authors of \cite{isett2025integral} deduce that in our case the augmented system is once again completely integrable.

The goal of this section is to prove the following structure of the cokernel elements:

\begin{lemma}\label{lem:cok_on_AC}
Let $A,C$ be as in Definition~\ref{defn:eucl_ann}.
Then the induced kernels of $P^*_R, P^*_M$ on $A$ and $C$ are given by
\begin{equation}
\begin{aligned}
\begin{cases}
\left\{\sqrt{1+||x||_2^2}, x^1, \dots, x^d\right\} \in \ker P^*_R \\
\left\{\sqrt{1+||x||_2^2}\delta_{ji} + x^l \delta_{j0}, x^k \delta_{jl}-x^l \delta_{jk}  \right\} \in \ker P^*_M.
\end{cases}
\end{aligned}
\end{equation}
and 
\begin{equation}
\begin{aligned}
\begin{cases}
\left\{\sqrt{1+z_d^2}, \left( \frac{2\left(z_d+\frac{1}{\delta}\right)}{1+||z'||_2^2}z^1\right), \dots, \left( \frac{2\left(z_d+\frac{1}{\delta}\right)}{1+||z'||_2^2}z^d\right)\right\} \in \ker P^*_R \\
\left\{\sqrt{1+z_d^2}\delta_{ji} + \left( \frac{2\left(z_d+\frac{1}{\delta}\right)}{1+||z'||_2^2}z^l\right) \delta_{j0}, \left( \frac{2\left(z_d+\frac{1}{\delta}\right)}{1+||z'||_2^2}z^k\right) \delta_{jl}-\left( \frac{2\left(z_d+\frac{1}{\delta}\right)}{1+||z'||_2^2}z^l\right) \delta_{jk}  \right\} \in \ker P^*_M.
\end{cases}
\end{aligned}
\end{equation}
respectively.
\end{lemma}

\begin{proof}[Proof of \ref{lem:cok_on_AC}]
Recall that in the ambient coordinates $\{y^0,\dots,y^d\}$ on the Hyperboloid and the annular subset $A_{2/\delta}$, the cokernel elements are explicitly given by 
\[\{y^0,\dots,y^d\}\qquad \text{for} \hspace{0.1cm} P^*_Rf\]
and
\[\{y^0\delta_{ji} + y^l\delta_{j0},y^k\delta_{jl}-y^l \delta_{jk}\}\qquad \text{for} \hspace{0.1cm} P^*_MY_{\alpha}.\]
Indeed, to see the former one verifies that
\begin{align*}(P^*_R y^{\mu})(X,Y) &= X(Yy^{\mu}) - (\overline{\nabla}_X Y) y^{\mu} - b(X,Y)y^{\mu} \\
&= \nabla_X(Y)y^{\mu} - (\overline{\nabla}_X Y) y^{\mu} - b(X,Y)y^{\mu} \\
&= (\overline{\nabla}_X Y) y^{\mu} + \Pi(X,Y) y^{\mu} - (\overline{\nabla}_X Y) y^{\mu} - b(X,Y)y^{\mu} \\
&= (\Pi(X,Y) - b(X,Y))y^{\mu} = 0.
\end{align*}
To see the latter, we note that both the rotation and boost vector fields are affine killing vector fields.

We wish to translate these cokernel elements into both $A$ and $C$.
Using polar coordinates on the Hyperboloid, write
\[y^0 = \cosh(r); \qquad y^i = \sinh(r)\theta^i\]
for $\theta \in \Sb^{d-1}$.
From this perspective, it is clear that by the Euclidean annular approximation, we can take
\[x^i = y^i,\]
and observe that $||x||_2 = \sinh(r)$, $x^0 = \cosh(r) = \sqrt{1+||x||_2^2}$.
Thus if we are interested in studying the Euclidean annular analog of Equation~\eqref{eqn:constraints} whose left-hand sides expand out in terms of the coordinates $x^i$ as
\begin{align*}
\begin{cases}
\pdv_{x^i} \pdv_{x^j} \sigma^{ij} - \Gamma^k_{ij}\pdv_{x^k} \sigma^{ij} - b_{ij} \sigma^{ij} \\
\pdv_{x^i} \pi^{ij} + \Gamma^i_{ik}\pi^{kj} + \Gamma^j_{ik} \pi^{ik} 
\end{cases}
\end{align*}
then the corresponding cokernel elements are given by 
\begin{equation}
\begin{aligned}
\begin{cases}
\left\{\sqrt{1+||x||_2^2}, x^1, \dots, x^d\right\} \in \ker P^*_R \\
\left\{\sqrt{1+||x||_2^2}\delta_{ji} + x^l \delta_{j0}, x^k \delta_{jl}-x^l \delta_{jk}  \right\} \in \ker P^*_M.
\end{cases}
\end{aligned}
\end{equation}

Likewise, we can take the pushforward of the expressions to the $z$ coordinates on $C$, which read
\begin{align*}
\begin{cases}
\pdv_{z^i} \pdv_{z^j} \sigma^{ij} - \Gamma^k_{ij}\pdv_{z^k} \sigma^{ij} - b_{ij} \sigma^{ij} \\
\pdv_{z^i} \pi^{ij} + \Gamma^i_{ik}\pi^{kj} + \Gamma^j_{ik} \pi^{ik} 
\end{cases}
\end{align*}
We note that in this step, the expressions look identical because this side of the equation involves only scalar quantities that are independent of choice of coordinates. 
Nevertheless, the tensors $b_{ij}$ and $\Gamma^k_{ij}$ contain factors of $\delta^{-1}$ for each raised angular component of $\sigma$ and $\pi$.
This will matter when we estimate linear terms of this form later.

From the inverse map 
\begin{align*}
    C &\to A \\
(z',z_d) &\mapsto \left( \frac{2\left(z_d+\frac{1}{\delta}\right)}{1+||z'||_2^2}z', \left(\frac{1}{\delta}+z_d\right)\frac{||z'||_2^2 - 1}{1+||z'||_2^2} \right)    
\end{align*}
we obtain the corresponding pushforward cokernel elements
\begin{equation}
\begin{aligned}
\begin{cases}
\left\{\sqrt{1+z_d^2}, \left( \frac{2\left(z_d+\frac{1}{\delta}\right)}{1+||z'||_2^2}z^1\right), \dots, \left( \frac{2\left(z_d+\frac{1}{\delta}\right)}{1+||z'||_2^2}z^d\right)\right\} \in \ker P^*_R \\
\left\{\sqrt{1+z_d^2}\delta_{ji} + \left( \frac{2\left(z_d+\frac{1}{\delta}\right)}{1+||z'||_2^2}z^l\right) \delta_{j0}, \left( \frac{2\left(z_d+\frac{1}{\delta}\right)}{1+||z'||_2^2}z^k\right) \delta_{jl}-\left( \frac{2\left(z_d+\frac{1}{\delta}\right)}{1+||z'||_2^2}z^l\right) \delta_{jk}  \right\} \in \ker P^*_M.
\end{cases}
\end{aligned}
\end{equation}
\end{proof}
Throughout the rest of the paper, we will simply denote these by their value in the $x^i$ coordinates, thinking of them as $x^i(z)$.

\textbf{Step 3}
We now turn our attention toward the solution operator $S$.

For our purposes, it will suffice to know that $S$ exists and has certain bounds depending on powers of $\delta$.
While we can use the theorems in Appendix \ref{sec:soln_op_thms}, a direct application will not be optimal, and therefore we must adapt a procedure for improvement.

In this step, we aim to prove the following proposition:

\begin{proposition}\label{prop:soln_op_bounds}
Consider the system
\[\begin{cases}
    P_R \sigma = \nabla_i \nabla_j \sigma^{ij}-b_{ij}\sigma^{ij} = F+M \\
    P_M \sigma = \nabla_i \pi^{ij} = G+N.
\end{cases}\]
Then there is a Bogovskii-type solution operator $S = (S_{f},S_Y)$ satisfying the following:
\begin{enumerate}
    \item \textbf{Support property}: We have that $\supp S \subseteq A_{2/\delta}$.
    \item \textbf{$S$ solves $P = (P_R, P_M)$ up to finite dimensional cokernel:}

    For initial data $(F+M, G+N) = h \in H^s(C)$, we have 
    \[PSh = h - w_A \langle Z^A, h \rangle,\]
    where $w_A \in C^{\infty}_c \left(B; \C^{2}\right)$ for $B$ a patch of geometric length $1/\delta$ and $Z^A$ are the globally defined cokernel elements in Lemma~\ref{lem:cok_on_AC}.
    \item \textbf{Solution Operator Estimates:}
    We have solution operator bounds
    \begin{equation}
    \begin{aligned}
        ||S_{f}||_{\wH^{s} \to \wH^{s+2}} \lesssim \frac{1}{\delta^2}; \qquad
        ||S_{Y}||_{\wH^{s} \to \wH^{s+1}} \lesssim \frac{1}{\delta^4}. 
    \end{aligned}
    \end{equation}
    where the constant is uniformly $O(1)$ depending on $s$ and $d$.
    \item \textbf{Commutator Estimates:}
    We have the commutator bounds
    \begin{equation}
    \begin{aligned}
        \|[S_{f},\pdv_{z^i}]\|_{\wH^s \to \wH^{s+1}} \lesssim \frac{1}{\delta^2}; \qquad
        ||[S_{Y},\pdv_{z^i}]||_{H^{s}(C) \to H^{s+1}(C)} \lesssim \frac{1}{\delta^4}. \\
    \end{aligned}
    \end{equation}
\end{enumerate}
\end{proposition}

The first step in this proof will be to verify the recovery on curves condition $(RC)$ and to derive solution operator bounds. 
For this, we turn to the concept of graded augmented systems described in \ref{subsec:grded_sys}.

\subsection{Preliminaries for Augmented System Computations}\hfill

We take $\tensor {\Rf}{_i_j^k_l}u^l = (\nabla_i \nabla_j - \nabla_j \nabla_i)u^k$ for the Riemann curvature tensor, and recall that the sectional curvature $\kappa = -1$.

Our metric is $b = dr^2 + \sinh^2(r)g_{\Sb^{d-1}}$ and for the purposes of this section we will denote $g_{\Sb^{d-1}}$ simply by $g$. 
As before, we denote the radial component by $r$, and angular components by $A,B$.
Our Christoffel symbols are 
\[\tensor \Gamma{^A_r_B} = \coth(r) \delta^A_B; \qquad \tensor \Gamma{^r_A_B} = -\sinh(r)\cosh(r)g_{AB};\ \qquad \tensor \Gamma{^A_B_C} = \tensor \Gamma{^A_B_C}(g_{\Sb^{d-1}}). \]
with all other components vanishing. 
Let us choose normal coordinates on $\Sb^{d-1}$ to arrange so that all off-diagonal metric components $g_{AB}, A \neq B$ are zero, and that each $g_{AA}$ is constant along a path varying $A$.

Normalized angular forms are given by $d\tilde{\omega}^A = \sinh(r)d\omega^A$ and $\pdv_{\tilde{\omega}}^A = \sinh(r)^{-1}\pdv_{\omega}^A$.
We remark that we will use the \emph{standard} (non-normalized) angular forms, which are of the scale $O(1)$ in our setting.

\textbf{The Graded Augmented Systems}

Recall that our adjoint equation reads
\[P^*(f,Y)_{ij} = ((P^*_R f)_{ij}, (P^*_M Y)_{ij}) = (\nabla_i \nabla_j f - f b_{ij}, -\nabla_{(i}Y_{j)}).\]

Let us write $\alpha_j = \nabla_j f$, and write $\eta_{jk} = \frac{1}{2}(\nabla_j Y_k - \nabla_k Y_j)$.

Then our graded augmented system becomes
\begin{equation}\label{eqn:gr_aug_sys}
\begin{cases}
\nabla_i f = \alpha_i \\
\nabla_i \alpha_j = (P^*_R f)_{ij} + b_{ij} f \\
\nabla_i Y_j = \eta_{ij} - (P^*_M Y)_{ij} \\
\nabla_i \eta_{jk} = - \tensor \Rf{_j_k_i^l} Y_l + \nabla_k (P^*_M Y)_{ij} - \nabla_j (P^*_M Y)_{ki}.
\end{cases}
\end{equation}

To solve this augmented system and derive an operator bound, there are two paths one needs to consider.

First, fix $y, y_1$ in the annulus and let $x(y,y_1,s)$ be a geodesic segment from $y$ to $y_1$ with $x(s) = (s, \omega_0)$ for some fixed $\omega_0 \in \Sb^{d-1}$.

In this setting, the augmented systems become:
\begin{equation}\label{eqn:rad_aug_sys_f}
\begin{cases}
\frac{d}{dt} f(x(t)) = \alpha_r \\
\frac{d}{dt} \alpha_r(x(t)) = (P^*_Rf)_{rr} + f \\
\frac{d}{dt} \alpha_A(x(t)) = (P^*_R f)_{rA} + \coth(r) \alpha_A.
\end{cases}
\end{equation}
and
\begin{equation}\label{eqn:rad_aug_sys_Y}
\begin{cases}
\frac{d}{dt} Y_r(x(t)) = -(P^*_M Y)_{rr} \\
\frac{d}{dt} Y_A(x(t)) = \coth(r) Y_A + \eta_{rA} - (P^*_M Y)_{rA} \\
\frac{d}{dt} \eta_{rA}(x(t)) = Y_A + \nabla_A(P^*_M Y)_{rr} - \nabla_r (P^*_M Y)_{rA} + \coth(r)\eta_{rA}.
\end{cases}
\end{equation}

These settings are discussed in Appendices A.2.2 and A.4.2 of \cite{isett2025integral}, where they prove, by construction, that $||S||_{H^k(C) \to H^{k+1}(C)} = O(1)$ for arbitrary $k$; in other words, the solution operator is uniformly bounded irrespective of $\delta$.
One can prove this directly using the Abstract ODE estimate (Proposition 6.3) or Theorem~\ref{thm:sol_op_5.15} in that same paper.

Now let us consider an angular path, varying some angular parameter $A$, i.e. $x(t) = (r,(1-s)A(y)+s(A(y_1)), B')$, where $B'$ denotes all other angular components, taken to be fixed.

Let us note that the angular coordinate chosen is the geometric scale $O(1/\delta)$ coordinate $\Delta t \approx \omega^A \approx O(1)$.
In addition, let us fix $\Omega \coloneqq \Delta_A = A(y)-A(y_1)$, which contributes the $O(1)$ factor $\dot{x}(y,y_1,s)$ from the chain rule.
Our augmented system then becomes

\begin{equation}\label{eqn:ang_aug_sys}
\begin{cases}
\frac{d}{dt} f = -\Omega \alpha_A \\
\frac{d}{dt} \alpha_r = -\Omega((P^*_R f)_{Ar} +  \coth(r) \alpha_A) \\
\frac{d}{dt} \alpha_A = -\Omega((P^*_R f)_{AA} + \sinh^2(r)g_{AA} f - \sinh(r)\cosh(r) g_{AA} \alpha_r) \\
\frac{d}{dt} Y_r = -\Omega(\eta_{Ar} - (P^*_M Y)_{Ar} + \coth(r)Y_A) \\
\frac{d}{dt} Y_A = -\Omega(-(P^*_M Y)_{AA} - \sinh(r)\cosh(r) g_{AA} Y_r) \\
\frac{d}{dt} \eta_{Ar} =-\Omega( \tensor \Rf{_r_A_A^l} Y_l -  \nabla_A (P^*_M Y)_{rA} + \nabla_r (P^*_M Y)_{AA}).
\end{cases}
\end{equation}

We also have separate equations for the $B$ coordinates with no interaction with the rest of the augmented system:

\begin{equation}\label{eqn:B_aug_sys}
\begin{cases}
\frac{d}{dt} \alpha_B = -\Omega (P^*_R f)_{AB} \\
\frac{d}{dt} Y_B = -\Omega( \eta_{AB} -  (P^*_M Y)_{AB}) \\
\frac{d}{dt} \eta_{AB} = -\Omega(- \tensor \Rf{_B_A_A^l} Y_l +  \nabla_B (P^*_M Y)_{AA} -   \nabla_A (P^*_M Y)_{BA}).
\end{cases}
\end{equation}
Thus we have constructed the graded augmented system for our equations, hence the recovery on curves condition as in Proposition~\ref{prop:aug_sys_RC} is satisfied.
Moreover, our augmented system is completely integrable, as we have a full formal cokernel.

Attempting to apply the ODE estimates in \cite{isett2025integral} would result in an exponential bound in $\delta$, but we can improve on this substantially by studying the system directly. 
Instead, our first solution operator will be designed to handle all desired regularity gain in the radial direction, obtaining optimal improvement in terms of factors of $\delta$.
Here already, we can use the existing theory and exploit angular symmetry.
However, to deal with the full anisotropic setting, we will also need to handle the angular components, which will be done by integrating the augmented system directly.

\begin{figure}[htbp]
\centering
\begin{tikzpicture}[
    scale=0.7,
    line cap=round,
    line join=round,
    bdry/.style={black,line width=0.9pt},
    traj/.style={green!60!black,line width=1.2pt},
    supp/.style={dashed,line width=.9pt},
    patchB/.style={supp,draw=blue,fill=blue!8},
    patchG/.style={supp,draw=green!60!black,fill=green!60!black!8},
    every node/.style={font=\small}
]
\def\R{2.4}
\def\r{1.6}
\def\Rm{2.0}          
\def\Rs{1.8}          
\def\pin{1.72}        
\def\pout{2.28}       
\newcommand{\annpatch}[3]{%
  \filldraw[#1] (#2:\pin) arc (#2:#3:\pin) -- (#3:\pout)
                arc (#3:#2:\pout) -- cycle;}
\begin{scope}[scale=1.15,yshift=-0.8cm]
\node at (0,3.2) {\Large First Operator};
\draw[bdry] (-2.25,.82) arc (160:20:\R);
\draw[bdry] (-1.50,.55) arc (160:20:\r);
\coordinate (c) at ({\Rm*cos(38)},{\Rm*sin(38)});
\node[above left=2pt] at (c) {$c(y)$};
\coordinate (y) at ({2.25*cos(40)},{2.25*sin(40)});
\fill (y) circle (1.6pt);
\node[right=2pt] at (y) {$y$};
\coordinate (yone) at ({\Rm*cos(30)},{\Rm*sin(30)});
\fill (yone) circle (1.6pt);
\node[below right] at (yone) {$y_1$};
\filldraw[patchB] (c) circle (.35);
\draw[traj] (y)--(c) arc (38:30:\Rm);
\end{scope}
\begin{scope}[xshift=7.5cm,scale=1.15,yshift=-0.8cm]
\node at (0,3.2) {\Large Second Operator};
\draw[bdry] (-\R,0) arc (180:0:\R);
\draw[bdry] (-\r,0) arc (180:0:\r);
\draw[bdry] (-\R,0)--(-\r,0);
\draw[bdry] (\R,0)--(\r,0);
\annpatch{patchB}{73}{107}
\draw[traj]
  ({\Rm*cos(165)},{\Rm*sin(165)})
  -- ({\Rs*cos(165)},{\Rs*sin(165)});
\draw[traj] ({\Rs*cos(165)},{\Rs*sin(165)}) arc (165:90:\Rs);
\draw[traj] (0,\Rm) arc (90:15:\Rm);
\coordinate (yl) at ({\Rm*cos(165)},{\Rm*sin(165)});
\fill (yl) circle (1.5pt);
\node[left=-1.5pt] at (yl) {$y$};
\coordinate (yr) at ({\Rm*cos(15)},{\Rm*sin(15)});
\fill (yr) circle (1.5pt);
\node[right=-1.5pt] at (yr) {$y$};
\coordinate (top) at (0,\Rm);
\fill (top) circle (1.5pt);
\coordinate (top2) at (0,\Rs);
\fill (top2) circle (1.5pt);
\node[above] at (0,{\pout+0.06}) {$y_1$};
\end{scope}
\begin{scope}[xshift=15cm]
\node at (0,3.2) {\Large Third Operator};
\draw[bdry] (0,\R) arc (90:-90:\R);
\draw[bdry] (0,\r) arc (90:-90:\r);
\draw[bdry] ({\R*cos(105)},{\R*sin(105)}) arc (105:-105:\R);
\draw[bdry] ({\r*cos(105)},{\r*sin(105)}) arc (105:-105:\r);
\draw[bdry] ({\R*cos(105)},{\R*sin(105)}) -- ({\r*cos(105)},{\r*sin(105)});
\draw[bdry] ({\R*cos(-105)},{\R*sin(-105)}) -- ({\r*cos(-105)},{\r*sin(-105)});
\annpatch{patchB}{76}{104}
\annpatch{patchB}{-104}{-76}
\annpatch{patchG}{-15}{15}
\node[blue,below] at (0,{\R-0.75}) {$\operatorname{supp}(\eta_2)$};
\node[blue,below] at (0,{-\R-0.05}) {$\operatorname{supp}(\eta_2)$};
\coordinate (yy) at (.02,-\Rm);
\draw[traj] (yy) arc (-90:0:\Rm);
\fill (yy) circle (1.5pt);
\node[left=3pt] at (yy) {$y$};
\coordinate (G) at (\Rm,0);
\fill (G) circle (1.5pt);
\node[right] at ({\R+0.05},0) {$y_1$};
\end{scope}
\end{tikzpicture}
\caption{The three-step Bogovskii-type operator construction. The blue region on the left is the support of $\eta_1$ for each $y$; the blue patches in the middle and on the right are the support of $\eta_2$ for all $y$. The green characteristic curves travel from $y$ to $y_1$.}
\label{fig:coarse-sets}
\end{figure}
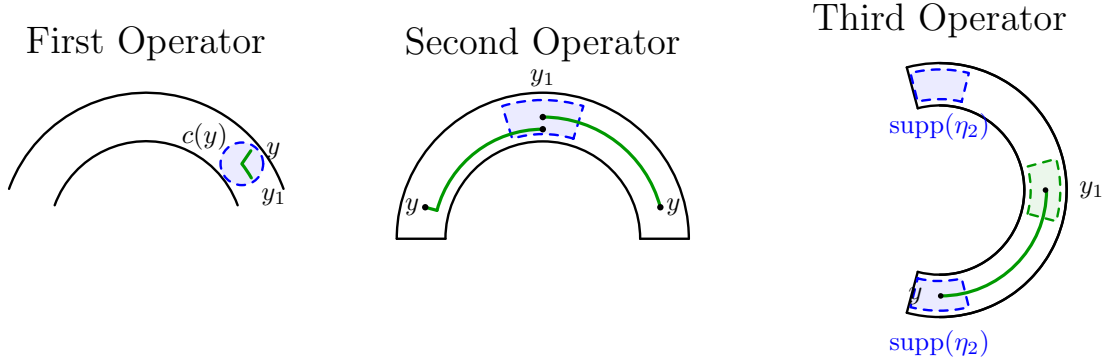

\textbf{The First Operator: Source Smoothing}

For the first operator, we will run estimates in the standard Sobolev space $H^k$, and then transport these results into the anisotropic setting with a concrete equivalence of norms.
We define the following construction on $A_{2/\delta}$.
For $y = (r,\omega)$, let $c(y) = \left(\frac{1}{2}(\log(2/\delta) + \log(8/\delta)), \omega \right)$ be the corresponding point with the same angular component but radial component purely in the center of the annulus. 
Let $\eta_1 = \eta_1(y,y_1)$ be a smooth function on $B_1(c(y))$ such that $\int \eta_1(y,y_1)\dd y_1 = 1$.
We then define the straight-line path 
\[x(y,y_1,s) = (1-s)y + s(y_1).\]
Thus, for each $y$ we define paths from $y$ to a small unit-sized ball localized around the center of the annulus. 

In light of the $y$-dependence of $\eta_1$, we wish to apply Lemma~\ref{lem:b_eta_5.17} and Theorem~\ref{thm:sol_op_existence}.
To do this, we observe that for $U = A_{2/\delta}$, we have the $x$-star-shaped condition:
\[\overline{\bigcup_{y \in \overline{U}, (y,y_1) \in \supp \eta_1} x(y,y_1,[0,1])} \subseteq \overline{U}. \]
and the formal cokernel condition $\dim \ker P^* = \# \Acal$.
Thus, we do indeed have a solution operator up to a smooth error term $B_{\eta_1}$.
To obtain quantitative estimates, we can verify the bounds for $A_{\eta},A_x$, and $A_S$.

Recall that $A_x$ is defined so that 
\[(1+A_x)^{-1}|y_1-y| \leq |\pdv_s x(y,y_1,s)| \leq (1+A_x)|y_1-y|,\]
and as $x$ is the straight-line path of unit-length here, $A_x = O(1)$.
Moreover, observe that as for each $y$, the function $\eta_1(y,y_1)$ is of unit scale, and $L_y \sim |y_1-y| \sim O(1)$ and $y,y_1$ derivatives correspond to $x',x_1'$ derivatives on the thin unit annulus, so a unit-scale change in $y,y_1$ corresponds to a size $\delta$ angular change. 
Hence 
\begin{equation*}
	|L_{y}^{|{\alpha}|} |y_{1}-y|^{|\beta|} (\pdv_{y} + \pdv_{y_{1}})^{\alpha} \pdv_{y_{1}}^{\beta}\eta(y, y_{1})| \leq A_{\eta} L_{\eta}^{-d} \hbox{ for } |\alpha| + |\beta| \leq M_{\eta}.
\end{equation*}
and
\begin{align*}
	|L_{y}^{|\alpha|} L_{b}^{|\beta|} (\pdv_{y} + \pdv_{y_{1}})^{\alpha} \pdv_{y_{1}}^{\beta} \eta(y, y_{1})| \leq A_{\eta} L_{\eta}^{-d} \quad \hbox{ for } |\alpha| + |\beta| \leq M_{\eta}'.
\end{align*}
With constants $M_{\eta},M_{\eta}' \in \N$ and with $A_{\eta} = O(1)$.

Finally we need to check the condition on $A_S$:
\[|L_y^{|\alpha|}|y_1-y|^{|\beta|}(\pdv_y + \pdv_{y_1})^{\alpha} \pdv_{y_1}^{\beta} S_J^{\gamma,K}(y,y_1,s) | \leq A_S |y_1-y|^{m_K + |\gamma|}s^{m_K + |\gamma|-1} \]
In the first solution operator regime, $A_S$ is indeed again uniformly bounded independent of $\delta$, as $L_y \sim |y_1-y| \sim O(1)$ and $y,y_1$ derivatives correspond to $x',x_1'$ derivatives on the thin unit annulus.
Moreover, the union of all line segments is strictly contained inside the annulus $A_{2/\delta}$.

Hence, the conditions of Theorem~\ref{thm:smth_ker_5.11} apply and we obtain a solution operator bound
\[||S^1_{f} h||_{H^{s+2}(A)} \lesssim ||h||_{H^s(A)}; \qquad \hbox{and} \qquad ||S^1_{Y} h||_{H^{s+1}(A)} \lesssim ||h||_{H^s(A)}\]
for any $s > 0$ and $h \in H^s(A)$ using the standard theory of pseudodifferential operators for the Hamiltonian constraint equation and momentum constraint equations respectively.
for the momentum constraint equation.
Here, we made the observation that these Sobolev space estimates actually hold on the Euclidean annular approximation.

Moreover, by the discussion preceding Lemma~\ref{lem:b_eta_5.17}, the $O(1)$ bounds for $A_g, A_Z$ hold by \cite[Proposition 6.6]{isett2025integral}.

Hence we can also apply Lemma~\ref{lem:b_eta_5.17} to obtain bounds 
\[||B_{\eta_1, f} h||_{H^{s+2}(A)} \lesssim ||h||_{H^s(A)} \quad \hbox{ and } \quad ||B_{\eta_1, f} h||_{H^{s+1}(A)} \lesssim ||h||_{H^s(A)}\]
respectively. 
Note that at this step everything is $O(1)$ as one would expect, given that our construction is completely unit-scale and are estimates are in the standard Sobolev space.

To finish the discussion of the first operator one needs to convert these estimates in $H^k(A)$ to ones on $\wH^k(A)$.
Observe that by definition,
\[\delta^k \|h\|_{\wH^k} \lesssim \|h\|_{H^k} \lesssim \|h\|_{\wH^k},\]
so in particular the above bounds would immediately give an $O(\delta^{-s-2})$ bound for $S^1$ and $B_{\eta_1}$, but these are far from optimal.
A sharper bound can be shown from the following lemma:

\begin{lemma}\label{lem:commuting_angular}
Let $S^1 = (S^1_{f},S^1_Y)$ and $B_{\eta_1} = (B_{\eta_1,f},B_{\eta_1,Y})$ as constructed above.
Then for any $\omega \in \Sb^{d-1}$, we have 
\[[S^1, \pdv_{\omega}] = [B_{\eta_1}, \pdv_{\omega}] = 0.\]

Consequently, w have the commutator bounds
    \begin{equation}
    \begin{aligned}
        ||[S^1_{f},\pdv_{x^i}]||_{\wH^{s} \to \wH^{s+1}} = O(1), \qquad ||[S^1_{Y},\pdv_{x^i}]||_{\wH^s \to \wH^s} = O(1). \\
    \end{aligned}
    \end{equation}
\end{lemma}

\begin{proof}[Proof of Lemma~\ref{lem:commuting_angular}]
Fix $R = e^{tA} \in \SO(d)$ a rotation matrix with $A \in \so(d)$, and for $f \in H^k$ define $Rf(x) = f(R^{-1}x)$.
Observe that our straight-line paths and $\eta$ satisfy
\[x(Ry,Ry_1,s) = Rx(y,y_1,s); \qquad \eta_1(Ry,Ry_1) = \eta(y,y_1).\]
Hence, by construction, if our integral kernel is of the form $S^1h = \int_{A_{2/\delta}} K(y,y_1)h(y_1)\dd y_1$, then we have $K(Ry,Ry_1) = K(y,y_1)$, so $K$ is $\SO(d)$ equivariant.
Using $\SO(d)$ equivariance and a change of variables, it is easy to see that $S^1Rh(y) = RS^1h(y)$.
Moreover, convergence in the strong operator topology yields
\[R'(0)h = \lim_{t \to 0} \left( \frac{R(t)-R(0)}{t}h \right) = \lim_{t \to 0} \frac{1}{t}\left(f(r,e^{-tA}\omega)-f(r,\omega)  \right) = (-\omega \cdot A) \nabla_{\omega} f.\]
and $S^1R'(0) = R'(0)S$.
Hence $S^1$ commutes with all angular vector fields as desired.
Exactly the same argument holds for $B_{\eta_1}$ as well.
An important consequence of this proof is that that the same bound holds for all components, because angular symmetry applies to each one individually.

Note that this construction is standard and can be found in, e.g. \cite{1972222}.
\end{proof}

As a result of Lemma~\ref{lem:commuting_angular}, we have
\begin{align*}
\|Sh\|_{\wH^{k+k'}} = \|\delta^{-k}\pdv^{\alpha}_x S h \|_{\wH^{k'}} &\lesssim \delta^{-k'} \| \delta^{-k}S \pdv^{\alpha}_x h \|_{H^{k'}} \\
&\lesssim \delta^{-k'} \|h\|_{\wH^k}.
\end{align*}

Hence we conclude the first solution operator bounds
\[||S^1_{f} h||_{\wH^{s+2}(A)} \lesssim \delta^{-2} ||h||_{\wH^s(A)}; \qquad \qquad ||S^1_{Y} h||_{\wH^{s+1}(A)} \lesssim \delta^{-1}||h||_{\wH^s(A)}\]
\[||B_{\eta_1, f} h||_{\wH^{s+2}(A)} \lesssim \delta^{-2}||h||_{H^s(A)}; \quad \quad ||B_{\eta_1, Y} h||_{\wH^{s+1}(A)} \lesssim \delta^{-1}||h||_{\wH^s(A)}.\]

Note that this bound of $\delta^{-2}$ for the Hamiltonian constraint is not uniform across all components of the operator; indeed, the conversion from $H^s$ to $\wH^s$ also contributes a factor of $\delta$ for every raised angular component in the solution operator.
It will be of particular importance in Step $4$ that $\|(S^1_f h)^{AA}\|_{\wH^{s+2}} \lesssim \|h\|_{\wH^s}$, while $\|(S^1_f h)^{rr}\|_{\wH^{s+2}} \lesssim 
\delta^{-2}\|h\|_{\wH^s}$.

\textbf{The Second Operator: Large-Scale Averaging}
We have succeeded in applying the Theorem to construct $PS_1h = h - B_{\eta_1}h$, where $S_1$ satisfies the bounds above and $B_{\eta_1}$ is smoothing. 
Here, ``$h$" refers to the nonlinearity on the right-hand side of Equation~\eqref{eqn:constraints}.

When constructing the bound on the first solution operator, all paths were unit length so we could get away with taking the stereographic projection of one hemisphere and obtaining solution operator estimates there. 
For the second operator, we wish to consider paths covering the full length of the annulus. 
For this purpose, we choose $\eta_2^1$ and $\eta_2^2$, purely functions of $y_1$, to live in a large $O(1/\delta)$-sized region at the north and south poles of the hyperbolic annulus. 

Then, depending on whether the point $y$ has positive or negative $y^d$ component, construct a path of fixed radius along the sphere to each point in $\supp \eta_{2}^1$ or $\supp \eta_2^2$, and average over them.
We will write $\eta_2 = \eta_2^1 + \eta_2^2$.

Note that $\eta_2$ is not dependent on $y$ in this construction, our augmented system is completely integrable, and all of the desired bounds for the smoothness condition $(C^{\infty})$ will still hold but with $\delta$-dependent bounds. 
Hence we can apply Theorem~\ref{thm:sol_op_5.15} to conclude the existence of a solution operator $S^2 = (S^2_f, S^2_Y)$ such that for any source $h \in \wH^s(A_{2/\delta})$ we have
\[PS^2B_{\eta_1}h = B_{\eta_1}h - B_{\eta_2}B_{\eta_1}h,\]
where $B_{\eta_2} = \sum_{\Af \in \Acal} (\zeta_{\eta})_{\Af} \langle \Zf^{\Af},\cdot \rangle$
with $(\zeta_{\eta})_{\Af}$ a smooth function such that $\supp(\zeta_{\eta})_{\Af} \subseteq \supp \eta_2$
and $\{\Zf^\Af\}_{\Af \in \Acal}$ a basis of $\ker P^*$ and $\zeta_{\Af}$.

However, we note that the bounds themselves arising from this theorem would be $\delta$-dependent and suboptimal, since our paths are no longer unit length.
Instead, we estimate the solution operator directly by estimating derivatives of the solution operator formula in $L^2$.

Note that the radial case was already considered, so it will suffice to look at purely angular paths to derive a uniform bound.
Our strategy will be to look at one fixed angular path for same-level estimates, but for improved regularity estimates we will need the fact that rotations are isometries which preserve the solution operator structure.

\subsection{Explicit $\delta$-Bound Computations}
For convenience, in this section we drop dependence of smoothing operators on $f$ and $Y$ until the estimates.
Let us first consider the \textbf{Hamiltonian Constraint}:
Reducing the equations for one in $\alpha_A$ yields
\[\frac{d^2}{dt^2}\alpha_A + \Omega^2 g_{AA} \alpha_A = - \Omega \frac{d}{dt} (P^*_R f)_{AA} - \Omega^2 \sinh(r) \cosh(r) g_{AA} (P^*_R f)_{Ar}. \]
The source terms on the right are of the order $O\left(\frac{1}{\delta^2}\right)$ due to the factor of $\sinh(r)\cosh(r)$.

On the other hand, we have an oscillatory solution with frequency $\omega_R = \Omega\sqrt{g_{AA}} \sim O(1)$. 
Our particular solution is given by 
\[\int_t^{1} \frac{\sin(\omega_R (t-s))}{\omega_R} \left(-\Omega \frac{d}{ds} (P^*_R f)_{AA} - \Omega^2 \sinh(r) \cosh(r) g_{AA} (P^*_R f)_{Ar}\right)(x(y,y_1,s)) \dd s,\]
whereas the homogeneous solution is given by 
\[\cos(\omega_R(t-1))\alpha_A(y_1) + \frac{\sin(\omega_R(t-1))}{\omega_R}\frac{d}{ds}\alpha_A \bigg\vert_{\alpha = y_1}.\]
Of course, this is just the equation for $\alpha_A$, and to recover an equation for the true adjoint variable $f$ we need to integrate again.
Denoting the source term by $\Ncal((P^*_Rf)_{AA}, (P^*_Rf)_{Ar})$, we have
\begin{align*}
f(y) &= f(y_1) - \Omega \frac{\sin(\omega_R)}{\omega_R}\alpha_A(y_1) + \Omega \frac{1-\cos(\omega_R)}{\omega_R^2}\alpha'_A(y_1) \\
&+ \int_0^1 \int_t^{1} \frac{\sin(\omega_R (t-s))}{\omega_R} \Ncal((P^*_Rf)_{AA}, (P^*_Rf)_{Ar})(x(y,y_1,s)) \dd s \dd t
\end{align*}
By interchanging the two integrals over the source in the final term, we obtain
\[\int_0^{1} \left( \frac{ \cos(\omega_R (s))-1}{\omega_R^2} \right)\Ncal((P^*_Rf)_{AA}, (P^*_Rf)_{Ar})(x(y,y_1,s)) \dd s.\]
Hence we have verified the recovery on curves for the Hamiltonian constraint equation. 
We will denote this second term by $S^*_f(P^*_Rf)$, and the expression $K(s,y,y_1) \coloneqq \left( \frac{ \cos(\omega_R (s))-1}{\omega_R^2} \right)$ for convenience. 
Note that the $y$ and $y_1$ dependence is implicit in the construction of $\omega_R$, but it is present. 

Recall that the path is given by $x(y,y_1,s) = (r,(1-s)A(y)+sA(y_1),B')$, where $A(y),A(y_1)$ denote the $A$ components of $y$ and $y_1$, respectively, and $B'$ denotes all other fixed angular components.
We will think about the path as taking the radial component of $y$ to that of $y_1$, and then taking a straight angular path to $y_1$; this need not stay in the direction $A$, but it will be straight angular path in a fixed coordinate up to some rotation.

Introducing a change of variables $w = (r,(1-s)A(y)+sA(y_1),B')$ and pairing with our smooth source $B_{\eta_1}h$ from the first solution operator, we obtain
\[\langle S^*_f(P^*_Rf), B_{\eta_1,f}h \rangle = \int_{A_{2/\delta}} \int_0^1 K(s,w,y_1) \Ncal(w) B_{\eta_1}h\left(r, \frac{(A(w) - sA(y_1))}{1-s},B'\right) \frac{1}{1-s}  \dd s \dd w. \]
Hence, our unsmoothed operator $S$ acts on the smooth source $B_{\eta_1}h$ with components
\begin{align*}
(S B_{\eta_1}h(w))^{Ar} &= \int_0^1 (\Omega^2 \sinh(r)\cosh(r)K(s,w,y_1)) B_{\eta_1}h \left(r, \frac{A(w) - sA(y_1)}{1-s}, B'\right) \frac{1}{1-s}  \dd s,  \\
(S B_{\eta_1}h(w))^{AA} &= \int_0^1 (\Omega \pdv_{s} K(s,w,y_1)) B_{\eta_1}h \left(r, \frac{A(w) - sA(y_1)}{1-s}, B'\right) \frac{1}{1-s}  \dd s  \\
&- \delta_{w} \int \Omega \left(\frac{\cos(\sqrt{g_{AA}}(A(y)-A(w)))-1}{g_{AA}(A(y)-A(w))^2}\right) B_{\eta_1}h(y) \dd y.
\end{align*}
where we fix a particular $y_1$.

Note that the $s=1$ boundary term is non-negligible, arising from integration by parts before we paired with the source.
Nevertheless, the component $\frac{\cos(\sqrt{g_{AA}}(A(y)-A(w)))-1}{g_{AA}(A(y)-A(w))^2}$ and its $\pdv_w$ derivatives are bounded uniformly $O(1)$, so it is infinitely smoothing even before adding the smooth $\eta_2$; hence we can safely ignore this term.

For the purposes of an $L^2$-estimate, the smoothing is not necessary, as it is $O(1)$.
Note also that we take an $s$ derivative of $K(s,w,y_1)$ here; we note that $\pdv_{A(w)} = -\frac{1}{\Omega} \pdv_s$ on $B_{\eta_1}h$ but this is in general not true on the kernel $K$.

Let us ignore the $AA$ component for now as $K'(s)$ is $O(1)$, and taking derivatives would introduce factors of $1/(1-s)$ which are no longer integrable.
Taking the $L^2$ norm over the entire annulus $A_{2/\delta}$ and observing that $\Omega^2 \sim O(1)$, we get
\begin{align*}
\|(SB_{\eta_1}h)^{Ar}\|_{L^2(y)} = \|(SB_{\eta_1}h)^{Ar}\|_{L^2(w)} &\lesssim \int_0^{1} (\sinh(r)\cosh(r)K(s,w,y_1)) \left\|B_{\eta_1}h(\cdot) \frac{1}{1-s}\right\|_{L^2(w)}  \dd s \\
&\lesssim \int_0^{1} ( \sinh(r)\cosh(r)K(s,w,y_1)) \left\|B_{\eta_1}h(\cdot) \frac{1}{\sqrt{1-s}}\right\|_{L^2(y)}  \dd s \\
&\lesssim \frac{1}{\delta^2}\|B_{\eta_1}h\|_{L^2(y)}, 
\end{align*}
where in the last step we used that $\sinh(r)\cosh(r) \sim O(1/\delta^2)$

Alternatively, peeling off one of the factors of $\Omega$ from $K(s,w,y_1)$ and using the substitution $-t = \frac{s}{1-s}((A(w)-A(y_1))$ removes the problematic factor of $\frac{1}{1-s}$; this, combined with the support property of $B_{\eta_1}h$, would conclude the $L^2$ bound.
This would also give an easy proof of a same-level, higher derivative estimate; derivatives fall either on $K$ or on $B_{\eta_1}$ by Leibniz rule, and with respect to this substitution, $w$ and $y$ derivatives are indistinguishable.

For completeness, for higher derivative estimate, we formulate the following lemma which uses a clever trick:

\begin{lemma}\label{lem:averaging_sobolev_estimate}
Let $\eta(y_1)$ be a smooth averaging function such that $\int \eta(y_1) \dd y_1 = 1$, $\supp(\eta) \subset U \subset A_{2/\delta}$ where $U$ is a set of geometric length $1/\delta$, and $\|\eta\|_{\wH^s} \sim O(1)$ for all $s > 0$.
Suppose we have a solution operator with components of the form
\[(SB_{\eta_1}h)^{ij} = \int \eta(y_1) \int_0^1 P(K^{(n)}(s,w,y_1))^{ij}B_{\eta_1}h\left(r, \frac{A(w)-sA(y_1)}{1-s},B'\right)\frac{1}{1-s} \dd s \dd y_1,\]
where $P(K^{(n)}(s,w,y_1))$ is a polynomial of $K(s,w,y_1)$ and its derivatives. 
Then we have an estimate of the form 
\[\|SB_{\eta_1}h\|_{\wH^s} \lesssim \delta^\alpha\|B_{\eta_1}h\|_{\wH^s},\]
with $\alpha$ given by the maximum over unit scale components $ij$ of the size of $P(K^{(n)}(s))$.
\end{lemma}

\begin{proof}[Proof of Lemma~\ref{lem:averaging_sobolev_estimate}]
    Note that our $\wH^s$ norms are measured in $w$, and $A(w)$ is of unit scale but geometric length $1/\delta$, so the $\wH^s$ norm can be measured in $\pdv_{A(w)}$ with no additional factors of $\delta$. 
    Fix a particular component $ij$ and drop the dependence of $P$ on this component for brevity. 
    Let us consider the $A(w)$ derivative first. 
    Observe that, if $D = \pdv_{A(w)} + \pdv_{A(y_1)}$, $D P(K^{(n)}) = 0$, since $PK^{(n)}(s,w,y_1)$ depends only on $w$ and $y_1$ through $A(w)-A(y_1)$.
    In addition, letting $\alpha = \frac{A(w)-sA(y_1)}{1-s}$, we have that $DB_{\eta_1}h = \frac{1-s}{1-s} \pdv_{\alpha} B_{\eta_1}h =  \pdv_{\alpha} B_{\eta_1}h$.

    Combining these facts together and using that $\pdv_{A(w)} = D - \pdv_{A(y_1)}$, we observe
    \begin{align*}
        \pdv_{A(w)}(SB_{\eta_1}h)^{ij} &= \int \eta(y_1)  \int_0^1 \pdv_{A(w)}( P(K^{(n)}(s,w,y_1))^{ij}B_{\eta_1}h\left(r, \frac{A(w)-sA(y_1)}{1-s},B'\right)\frac{1}{1-s} \dd s) \dd y_1 \\
        &= \int \eta(y_1)  \int_0^1 - \pdv_{A(y_1)} \left(P(K^{(n)})^{ij} B_{\eta_1}h\left(r, \frac{A(w)-sA(y_1)}{1-s},B'\right) \right) \frac{1}{1-s} \dd s \dd y_1\\
        &+ \int \eta(y_1)  \int_0^1 (P(K^{(n)}(s,w,y_1))^{ij}) \pdv_{\alpha} B_{\eta_1}h\left(r, \frac{A(w)-sA(y_1)}{1-s},B'\right)\frac{1}{1-s} \dd s \dd y_1.
    \end{align*}
The $A(y_1)$ derivative can be safely integrated by parts onto $\eta(y_1)$, whereas $\alpha$ derivatives in $B_{\eta_1}h$ are simply angular derivatives of the correct scale. 
Hence for any $k > 0$ we have
    \begin{align*}
    &\|\pdv_{A(w)}^k (S B_{\eta_1}h)^{ij}\|_{L^2} \\
    &\lesssim \|\pdv^k_{A(y_1)} \eta\|_{L^\infty} \int_0^{1/2} \left\|P(K^{(n)}) \pdv_{\alpha}^{k} B_{\eta_1}h\left(r, \frac{A(w)-sA(y_1)}{1-s},B'\right)\frac{1}{(1-s)}\right\|_{L^2(w)} \dd s \\
    &\lesssim \|\pdv^k_{A(y_1)} \eta\|_{L^\infty} \int_0^{1/2} \left\|P(K^{(n)}) \pdv_{\alpha}^{k} B_{\eta_1}h\left(r, \frac{A(w)-sA(y_1)}{1-s},B'\right)\frac{1}{(1-s)^{1/2}}\right\|_{L^2(r,\alpha,B')} \dd s \\
    &\lesssim \|\eta\|_{\wH^{d/2+k}} \|P(K^{(n)})\|_{L^{\infty}} \|B_{\eta_1}h\|_{\wH^k} \lesssim \|P(K^{(n)})\|_{L^{\infty}} \|B_{\eta_1}h\|_{\wH^k},
    \end{align*}
    where the last step follows by anisotropic Sobolev embedding and the fact that $\eta$ is uniformly bounded in $\wH^k$.
    In the penultimate step, we observe that $(1-s)\dd \alpha = \dd A(w)$, which allows us to integrate.

    The $\delta^\alpha$ factor then comes from the $L^{\infty}$-norm of $P(K^{(n)}(s))$ alongside our definition of anisotropic Sobolev norm for tensors.

    We now need to check derivatives in other angular components $B'$. For this, observe that $P(K^{(n)}(s,w,y_1))$ and $\eta(y_1)$ do not depend on the $B'$ coordinate of $w$, but only the corresponding coordinate of $y_1$. 
    Therefore, the $B'$ derivative simply passes through the other terms and we find
    \[\|\pdv_{B'}^k (S B_{\eta_1}h)^{ij}\|_{L^2} \lesssim \|\eta\|_{\wH^{d/2}}\|P(K^{(n)})\|_{L^{\infty}} \|B_{\eta_1}h\|_{\wH^k},\]
    as required.

    Finally, we consider the radial derivatives.
    By taking the radial path to $y_1$ first, $P(K^{(n)})$ only depends radially on $y_1$ components, not on $y$; hence, we can apply the same bound
    \[\|\pdv_{r}^k (S B_{\eta_1}h)^{ij}\|_{L^2} \lesssim \|\eta\|_{\wH^{d/2}}\|P(K^{(n)})\|_{L^{\infty}} \|B_{\eta_1}h\|_{\wH^k}.\]
\end{proof}

Note that the aforementioned substitution $-t = \frac{s}{1-s}((A(w)-A(y_1))$ would also give an easy proof of a same-level, higher derivative estimate; derivatives fall either on $K$ or on $B_{\eta_1}$ by Leibniz rule, and with respect to this substitution, $w$ and $y$ derivatives are indistinguishable, and importantly the smoothness of $\eta$ is not necessary at all at this step. 

By Lemma~\ref{lem:averaging_sobolev_estimate}, and the definition of the anisotropic Sobolev norm, we therefore get an estimate
\[ \|S_f^2 B_{\eta_1,f}h\|_{\wH^{s+2}} \lesssim \delta^{-2}\|B_{\eta_1,f}h\|_{\wH^{s+2}} \lesssim \delta^{-4} \|h\|_{\wH^{s}},\]
where the first factor of $\delta^{-2}$ comes from the $\sinh(r)\cosh(r) \sim \delta^{-2}$ term on an $Ar$ component, and the extra two factors of $\delta^{-1}$ come from the regularity gain obtained by Lemma~\ref{lem:commuting_angular}.

Unfortunately, this naive estimate is not sharp in this particular case, so we instead argue that we gain regularity directly from the integrations in the augmented system.

Up until now, we only used the path in one angular direction, but we require our regularity gain to be in all angular components.
To this end, we need to build a framework that will let us deal with multiple directions at once.

\begin{definition}
For a point $y$ with angular component $(A(y),B'(y))$, we define 
\[\Af(y) = \frac{(A(y),B'(y))}{(A(y),B'(y)) - p_0} - p_0,\]
where $p_0$ denotes either the point at the top of the hemisphere or the bottom.
This is the Gnomonic projection, and $\Af(y) \in \R^{d-1}$ are called the Gnomonic coordinates.
\end{definition}

Gnomonic coordinates have the nice feature that great circles are taken to straight lines; hence we still have a parametrization $\Af(w) = (1-s)\Af(y) + s \Af(y_1)$ as vectors in $\R^{d-1}$.
Moreover, observe that as rotations on $\Sb^{d-1}$ are isometries, for any point $y_1 \in \Sb^{d-1}$ there exists a rotation $R \in \SO (d-1)$ fixing $y$ that takes $A(Ry_1) = A(y)$.
As a result, up to some smooth and $O(1)$ changes of variables, the same kernel structure $K(s,y,y_1)$ will hold for each individual angular direction in a uniformly bounded way. 
For the purposes of the following estimate, we will denote this new, fully angular-dependent kernel by $L(s, \Af(w),\Af(y_1))$.

\begin{lemma}\label{lem:fio_estimate}
    Consider a solution operator $S$ given by a average over angular paths such that on each angular segment $\Af$, it acts on a source $B_{\eta_1}h$ by 
    \[SB_{\eta_1}h = \int \eta(r,\Af(y_1),B') \int_0^1 s^k L(s,\Af(w),\Af(y_1)) B_{\eta_1}h\left(r,\frac{\Af(w)-s\Af(y_1),B'}{1-s}\right) \frac{1}{1-s} \dd s \dd \Af(y_1),\]
    where $L(s,\Af(w),\Af(y_1))$ represents a particular component of the source, namely a function of $K(s,w,y_1)$, its derivatives, and the reparameterization into Gnomonic coordinates.
    Assume in addition that $L$ is smooth and oscillatory, with all derivatives uniformly bounded $O(1)$, hence with uniformly bounded frequencies, i.e., $\hat{L}(s,\lambda,\mu) \sim O(1)$.
    Then $S$ is smoothing of order $-(k+1)$ in all angular directions with phase $\Phi_{\Af(w)} = \frac{\Af(w)(\xi+\lambda)}{1-s}$ with respect to the anisotropic Sobolev space $\wH^s$.
    In particular, we have the estimate
    \[ \|SB_{\eta_1}h\|_{\wH^{s+k+1}_{an}} \lesssim \|B_{\eta_1}h\|_{\wH^{s}_{an}},\]
    where the subscript refers specifically to the angular derivatives, with no additional factors of $\delta$.
\end{lemma}

\begin{proof}[Proof of Lemma~\ref{lem:fio_estimate}]
As before, fix a component $ij$ of $S^2_fB_{\eta_1}h$.
Since $A(w),A(y_1)$ is $O(1)$ quantities that are geometrically length $1/\delta$, the corresponding norm $\wH^s$ carries no factors of $\delta$ per derivative and hence can be treated as the standard norm for the purposes of this argument.

Our operator is of the form
\[\int \eta(\Af(y_1)) \int_0^1 s^k L(s,\lambda,\mu) B_{\eta_1}h\left(\frac{\Af(w)-s\Af(y_1)}{1-s}\right) \frac{1}{(1-s)^{d-1}} \dd s \dd \Af(y_1),\]
where $L(s,\lambda,\mu) \sim O(1)$.
    
Locally, we can treat $\alpha \in \R$ and write $B_{\eta_1}h(\alpha) = \int \widehat{B_{\eta_1}h}(\xi) e^{i\xi \alpha}\dd \xi$.
Likewise, we expand out $L(s,w,y_1) = \int \hat{L}(s,\lambda) e^{i \alpha \lambda}e^{i \mu \Af(y_1)} \dd \lambda \dd \mu$. 
Here we similarly that $\alpha = \frac{\Af(w)-s\Af(y_1)}{(1-s)^{d-1}}$, so we can express our integral as 
\begin{align*} &\int_\R \int \int \int \eta(\Af(y_1)) \int_0^1 s^k\hat{L}(s,\lambda,\mu)e^{i (\alpha \lambda + \Af(y_1) \mu)} \widehat{B_{\eta_1}h}({\xi}) e^{i\xi \alpha} \frac{1}{(1-s)^{d-1}} \dd s  \dd \lambda \dd \Af(y_1) \dd \xi \dd \mu. \\
&= \int \eta(\Af(y_1)) s^k\hat{L}(s,\lambda,\mu) e^{-i\Af(y_1) \left( \frac{s(\lambda+\xi)}{1-s}\right)} e^{i\Af(y_1) \mu}\widehat{B_{\eta_1}h}(\xi) e^{iA(w) \left(\frac{\xi+\lambda}{1-s}\right)} \frac{1}{(1-s)^{d-1}} \dd s  \dd \lambda \dd \Af(y_1) \dd \xi \dd \mu.
\end{align*}

Performing the integration first over $\Af(y_1)$, we obtain
\[ \int_{\R^{d-1}} \int_{\R^{d-1}} \int_{\R^{d-1}}\int_0^1 \hat{\eta}\left(\frac{s( \xi + \lambda)}{1-s}-\mu\right) \widehat{B_{\eta_1}h}({\xi}) s^k \hat{L}(s,\lambda) e^{i \Af(w) (\xi + \lambda)/(1-s)} \frac{1}{(1-s)^{d-1}} \dd s \dd \xi \dd \lambda \dd \mu.\]

As a result, we have the symbol
\[a(s,\xi, \lambda,\mu) = \hat{\eta}\left(\frac{s(\xi + \lambda)}{1-s}-\mu\right) \frac{s^k \hat{L}(s,\lambda,\mu)}{(1-s)^{d-1}}.\]
Observe that since $\eta$ is smooth, its Fourier transform is rapidly decaying with $\left| \hat{\eta}\left(\frac{s(\xi+\lambda)}{1-s}-\mu\right) \right|\lesssim \langle \frac{s(\xi+\lambda)}{1-s} - \mu \rangle^{-N}$ for every $N \geq 0$.
Moreover, $\hat{L}(s,\lambda,\mu)$ is smooth and uniformly bounded in the variables $\Af(y_1)$ and $\alpha$, so $|\hat{L}(s,\lambda,\mu)| \lesssim \langle \lambda \rangle^{-N}\langle \mu \rangle^{-N}$ for any $N \geq 0$.
Therefore, $|a(s,\xi)| \lesssim  \langle \frac{s (\xi +\lambda)}{(1-s)}-\mu \rangle^{-N} s^k \langle \lambda \rangle^{-N}\langle \mu \rangle^{-N} \frac{1}{(1-s)^{d-1}}$.
The larger frequency contributions can therefore be neglected. 

Due to the phase $e^{i(\xi + \lambda) A(w)/(1-s)}$, taking derivatives in $A(w)$ gives extra factors of $(\xi+\lambda)/(1-s)$. 

To obtain the anisotropic Sobolev estimate, we first define auxilliary variables $\zeta = \frac{\xi + \lambda}{1-s}$, $u = (1-s)\zeta - \lambda$; the Jacobian factor of $(1-s)^{-(d-1)}$ cancels, and we obtain
\[\widehat{SB_{\eta_1}h}(\zeta) =  \int \int \left(  \int_{s=0}^1 \langle s^k \hat{\eta}(s\zeta-\mu) \hat{L}(s,\lambda,\mu) \widehat{B_{\eta_1}h}((1-s)\zeta - \lambda) \dd s \dd \lambda \right) \dd \mu. \]
Without loss of generality, we can assume that $|\zeta| \geq 1$. The region where $|\zeta| \leq 1$ is of unit scale, so all terms are uniformly $O(1)$ and we would obtain a direct bound of the source $B_{\eta_1}h$ in $L^2$ by Plancherel.
Let us first study a pointwise estimate in $\zeta$, and then integrate later.

Using Cauchy-Schwarz and the decay on $\hat{L}, \hat{\eta}$ we found earlier, we can write
\begin{align*}
|\widehat{SB_{\eta_1}h}(\zeta)|^2 &\lesssim \left( \int \int \int s^{2k} \langle \lambda \rangle^{-M} \langle \mu \rangle^{-M} \langle s\zeta - \mu \rangle^{-N} \dd s \dd \lambda \dd \mu \right)   \\
&\left(\int \int \int \langle \lambda \rangle^{-N} \langle \mu \rangle^{-N} \langle s\zeta - \mu \rangle^{-N} \widehat{B_{\eta_1}h}((1-s)\zeta - \lambda) \dd s \dd \lambda \dd \mu \right).
\end{align*}

For the first term, we use Peetre's inequality, namely that $\langle s \zeta - \mu \rangle^{-N} \lesssim \langle s \zeta \rangle^{-N} \langle \mu \rangle^{N}$; hence the $\mu$ factor gets absorbed into the decay term from $\hat{L}$ taken with a higher power, $\langle \lambda \rangle^{-N}$ can be integrated, and we find that the first term is bounded by 
\[\int_0^1 s^{2k} \langle s |\zeta| \rangle^{-N} ds = \frac{1}{|\zeta|} \int \left(\frac{t}{|\zeta|}\right)^{2k} \langle t \rangle^{-N} dt \lesssim |\zeta|^{-(2k+1)}.\]

Here we are using that $s \lesssim 1/|\zeta|$, which is true provided $s$ is sufficiently close to $0$; otherwise, $\zeta$ has arbtirarily many positive factors of $(1-s)$, so the $s=1$ endpoint is bounded.

As remarked earlier, each $A(w)$ derivative corresponds to a factor of the Japanese bracket $\langle \zeta \rangle$.
Hence, to show the desired regularity, let us multiply the second term by $\langle \zeta \rangle^{2(s+k+1)}$. 
As a result of the first bound, we only have $2(s+k+1)-2k-1 = 2s+1$ factors to deal with, so it suffices to bound 
\[\|B_{\eta_1}h\|_{\wH^{s+k+1}_{an}}^2 \lesssim \int \left(\int \int \int \langle \zeta \rangle^{2s+1} \langle \lambda \rangle^{-N} \langle \mu \rangle^{-N} \langle s\zeta - \mu \rangle^{-N} |\widehat{B_{\eta_1}h}((1-s)\zeta - \lambda)|^2 \dd s \dd \lambda \dd \mu \right) \dd \zeta.\]

For this bound, we split into two regimes: $s \leq 1/2$ and $s \geq 1/2$.

For $s \geq 1/2$, we have $|s \zeta| \geq |\zeta|/2$, hence $\langle s \zeta - \mu \rangle^{-N} \lesssim \langle \mu \rangle^{N} \langle \zeta \rangle^{-N}$.
In these coordinates, $B_{\eta_1}h$ is supported in a unit ball, so the $\|\widehat{B_{\eta_1}h}\|_{L^{\infty}} \sim \|B_{\eta_1}h\|_{L^1} \sim \|B_{\eta_1}h\|_{L^2}$ with no factors of $\delta$ in the constant. 
Thus our bound is
\begin{align*}
&\left| \int \left(\int \int \int \langle \zeta \rangle^{2s+1} \langle \lambda \rangle^{-N} \langle \mu \rangle^{-N} \langle s\zeta - \mu \rangle^{-N} |\widehat{B_{\eta_1}h}((1-s)\zeta - \lambda)|^2 \dd s \dd \lambda \dd \mu \right) \dd \zeta.\right| \\
&\lesssim \|B_{\eta_1}h\|_{L^2} \int \langle \zeta \rangle^{2s+1 - N} \langle \lambda \rangle^{-M} \langle \mu \rangle^{N-M} \dd s \dd \lambda \dd \mu \lesssim  \|B_{\eta_1}h\|_{L^2}. 
\end{align*}

For $0 \leq s \leq 1/2$, we have $\zeta - \xi = s \zeta + \lambda = (s \zeta - \mu) + (\lambda + \mu)$, hence by Peetre's inequality again we have
\[\langle \zeta \rangle^{2s} \lesssim \langle \xi \rangle^{2s} \langle \zeta - \xi \rangle^{2s} \lesssim \langle \xi \rangle^{2s} \langle s \zeta - \mu \rangle^{2s} \langle \lambda \rangle^{2s} \langle \mu \rangle^{2s}.\]

Using that $\langle \frac{\xi + \lambda}{1-s} \rangle \lesssim \langle \xi + \lambda \rangle$ and that $|s\zeta| \lesssim \langle s \zeta - \mu \rangle^{-N} \lesssim \langle \mu \rangle^N \langle s |\xi + \lambda| \rangle^{-N}$, we find that our bound becomes

\begin{align*}
&\left| \int \left(\int \int \int \langle \zeta \rangle^{2s+1} \langle \lambda \rangle^{-N} \langle \mu \rangle^{-N} \langle s\zeta - \mu \rangle^{-N} \widehat{B_{\eta_1}h}((1-s)\zeta - \lambda) \dd s \dd \lambda \dd \mu \right) \dd \zeta.\right| \\  
&\lesssim \int \int \langle \lambda \rangle^{-N'} \langle \mu \rangle^{-N'} \int \langle xi \rangle^{2s} |\widehat{B_{\eta_1}h}(\xi)|^2 \int \langle \xi + \lambda \rangle \langle s |\xi+ \lambda| \rangle^{-N'} \dd s \dd \lambda \dd \mu \dd \xi
&\lesssim \|B_{\eta_1}h\|_{\wH^s_{an}}^2,
\end{align*}
where in the last step we use the fact that the interval length in $s$ is like $\langle \xi + \lambda \rangle^{-1}$, canceling with the extra power, and the rest is absorbed by decaying, integrable terms.

Note that the Jacobian factor is not important due to the $s=0$ regime.

Thus, we find that $\|SB_{\eta_1}h\|_{\wH^{s+k+1}_{an}} \lesssim \|B_{\eta_1}h\|_{\wH^{s}_{an}}$, i.e. we gain $k+1$ derivatives in all angular directions, as desired.
\end{proof}

By the appendix computations A.2.2, A.4.2 in \cite{isett2025integral} and the first solution operator construction, the radial operator is already known to improve regularity by $2$ with no factors of $\delta$, so we get the desired maximal regularity in the radial components as well.

Note that now we can safely handle both the $Ar$ and $AA$ components.
Observe that the trick in the proof of Lemma~\ref{lem:averaging_sobolev_estimate} causes a loss of $2$ factors of $\delta$ from the sharp estimate we would expect.
However, by applying Lemma~\ref{lem:fio_estimate}, we recover regularity and the factors of $\delta$ back.
Explicitly, since the $Ar$ kernel is of the form $\Omega^2 s^2 \left(\frac{\cos(\omega_R s) - 1}{(\omega_R s)^2}\right)$ where the latter is uniformly bounded, Lemma~\ref{lem:fio_estimate} gives us a gain of $3$ derivatives.
On the other hand, the $AA$ component is of the form $\Omega s \left(\frac{\sin(\omega_R s)}{\omega_R s}\right)$, which gains two derivatives by the same reasoning.

Consequently, we find
\[\|(SB_{\eta_1}h)^{AA}\|_{\wH^s} \lesssim \|B_{\eta_1}h\|_{\wH^{s-2}} \lesssim \|h\|_{\wH^{s-2}} \]
and
\[\|(SB_{\eta_1}h)^{Ar}\|_{\wH^s} \lesssim \delta^{-2}\|B_{\eta_1}h\|_{\wH^{s-2}} \lesssim \delta^{-2}\|h\|_{\wH^{s-2}}. \]

Combining everything together, by first applying Lemma~\ref{lem:averaging_sobolev_estimate} together with Lemma~\ref{lem:fio_estimate} for the regularity improvement, we get a sharp estimate
\[ \|S_f^2 B_{\eta_1,f}h\|_{\wH^{s+2}} \lesssim \delta^{-2}\|B_{\eta_1,f}h\|_{\wH^{s}} \lesssim \delta^{-2} \|h\|_{\wH^{s}}.\]

An identical argument as in Lemma~\ref{lem:fio_estimate} gives the commutator estimate, except applying $\pdv_{\Af(w)}$ before $S$ yields a factor of $|\xi|$ in the bound, whereas applying it after yields a factor of $|\zeta|$. 
From our earlier computation, $\xi - \zeta = (s\zeta + \lambda) = (s\zeta - \mu) + (\lambda + \mu)$, each of which are bounded by the factors of $-N'$ in the earlier argument.

As a result, $[S^2_f,\pdv_{A(w)}]$ is once again smoothing of order $-(k+1)$, which gives a regularity gain of $k+1$ directly.

Hence, 
\[ \|[S^2_{f},\pdv_{\omega}]B_{\eta_1,f}h\|_{\wH^{s+1}} \lesssim \delta^{-2}\|B_{\eta_1,f}h\|_{\wH^{s}}\]
with no additional factors of $\delta$.

\begin{lemma}\label{lem:remainder_second_operator}
The remainder operator $B_{\eta_2}$ given by smoothing
\begin{equation}\label{eqn:b_eta_2}
\langle b_{\eta_2}(\cdot, y), f \rangle = f(y_1) - \frac{\sin(\sqrt{g_{AA}}(A(y)-A(y_1)))}{\sqrt{g_{AA}}}\alpha_A(y_1) + \frac{1-\cos(\sqrt{g_{AA}}(A(y)-A(y_1)))}{g_{AA}}\pdv_A \alpha_A(y_1)
\end{equation}
against $\eta_2(y_1)$ satisfies the bound
\[\|B_{\eta_2}f B_{\eta_1}h \|_{\wH^{s+2}} \lesssim \|B_{\eta_1}h\|_{\wH^s} \lesssim \|h\|_{\wH^s}\]
for any $s > 2$.
For any $s > 0$, we have
\[\|B_{\eta_2}f B_{\eta_1}h \|_{\wH^{s+2}} \lesssim \|B_{\eta_1}h\|_{\wH^2} \lesssim \max\{1,\delta^{s}\}\|h\|_{\wH^{s}}\]
\end{lemma}

\begin{proof}[Proof of Lemma~\ref{lem:remainder_second_operator}]
For a bound on $B_{\eta_2}$, using the substitution $t = A(y), 1 = A(y_1)$ to represent everything in the angular coordinates and Equation~\ref{eqn:b_eta_2}, observe that $b_{\eta_2}(\cdot, y) = \delta_{y_1} - c_1 \pdv_{A(y_1)} \delta_{y_1} + c_2 \pdv_{A(y_1)}^2 \delta_{y_1}$; where the second and third coefficients are labeled $c_1$ and $c_2$, respectively. 
Explicitly, then, the scalar function $(B_{\eta_2}B_{\eta_1}h)(y)$ is given by
\[\int \eta_2(y_1) [(B_{\eta_1}h)(y_1) -  c_1 \pdv_{A(y_1)} (B_{\eta_1}h)(y_1) + c_2\pdv_{A(y_1)}^2 (B_{\eta_1}h)(y_1) ] \dd y_1. \]
Since $|\pdv_{A(y)}^k(B_{\eta_2}B_{\eta_1})h(y)| \lesssim \left| \int \eta_2(y_1) (g_{AA}^{(k-1)/2} + g_{AA}^{(k-2)/2})   \dd y_1 \right|$, applying Sobolev embedding to $\|\eta\|_{L^\infty}$ again and using that $\sqrt{g_{AA}} \sim O(1)$ gives us the estimate
\[\|\pdv_{A(y)}^k B_{\eta_2,f}^{(A)}B_{\eta_1,f}h\|_{L^2} \lesssim \|B_{\eta_1,f}h\|_{\wH^{2}},\]
for any $k > 2$, with no fators of $\delta$ whatsoever.

By the fact that the connection defined by the homogeneous augmented system is flat, the path for the operator $B^*$ is irrelevant and hence for simplicity we can take a composition of several legs, each traversing a single angular component.
Combining all legs of the path together, starting first with the radial and passing over all angular components, we obtain a structure of remainder terms of the form
\[B_{\eta_2,f}B_{\eta_1}h \coloneqq B_{\eta_2,f}^{(r)}B_{\eta_1,f}h + B_{\eta_2,f}^{(A)}B_{\eta_2,f}^{(r)}B_{\eta_1,f}h + \dots + B_{\eta_2,f}^{(B'_{d-2})}\cdots B_{\eta_2,f}^{(A)}B_{\eta_2,f}^{(r)}B_{\eta_1,f}h.\]
The radial path and its derivatives are $O(1)$ by construction, whereas the $A$ and $B'$ derivatives all satisfy this precise structure, falling only one one particular term as the rest are absorbed in the $y_1$-dependent source.

Hence, for any $s > 2$ we get an estimate 
\[\|B_{\eta_2,f}B_{\eta_1}h\|_{\wH^{s+2}} \lesssim \|B_{\eta_1}h\|_{\wH^{s}} \lesssim \|h\|_{\wH^{s}};\]
in other words, we gained all of the regularity improvement from the our smooth functions $c_1$ and $c_2$, and there is no factor of $\delta$.

For the second statement, we combine the above with the fact that $\|B_{\eta_1}h\|_{\wH^2} \lesssim \delta^{s}\|h\|_{\wH^{s}}$ whenever $0 < s < 2$ by Lemma~\ref{lem:commuting_angular}.
\end{proof}

Similarly, the commutator estimate also produces an $O(1)$ bound on independent terms as $\sqrt{g_{AA}} \sim O(1)$:
\[\|[B_{\eta_2,f}, \pdv_{A(y)}]B_{\eta_1}h\|_{\wH^{s+2}} \lesssim \|B_{\eta_1}h\|_{\wH^{s}} \lesssim \|h\|_{\wH^{s}};\]
provided that $s > 3$; for lower regularity, we still need to use the first solution operator to obtain factors of $\delta$.



Next, let us consider the \textbf{Momentum constraint} equations. 

Reducing the system for $Y_r$ yields
\begin{align*}
\frac{d^2}{dt^2} Y_r &= \Omega^2( - \nabla_A (P^*_M Y)_{Ar} + \nabla_R (P^*_M Y)_{AA} + \sinh^2(r)g_{AA} Y_r \\
&- \cosh^2(r)g_{AA}Y_r -\frac{1}{\Omega} \frac{d}{dt} (P^*_M Y)_{Ar} - \coth(r) (P^*_M Y)_{AA}).\\
&=\Omega^2(-g_{AA} Y_r - \nabla_A (P^*_M Y)_{Ar} + \nabla_r (P^*_M Y)_{AA} + \frac{1}{\Omega} \frac{d}{dt} (P^*_M Y)_{Ar} - \coth(r) (P^*_M Y)_{AA}).
\end{align*}

As before, setting our frequency to be $\omega_M = \Omega \sqrt{g_{AA}}$, our particular solution is given by 
\[\int_t^{d(y_1,y)} \frac{\sin(\omega_M (t-s))}{\omega_M} \Omega^2 \left( \left(\frac{1}{\Omega} \frac{d}{ds}-\nabla_A\right) (P^*_M Y)_{Ar} - (\coth(r)-\nabla_r)(P^*_M Y)_{AA}\right)(s) \dd s. \]
with homogeneous solution
\[\cos(\omega_M(t-1))Y_r(y_1) + \frac{\sin(\omega_M(t-1))}{\omega_M}Y_r'(y_1).\]

Using the definitions $\nabla_A (P^*_M Y)_{Ar} = \pdv_A(P^*_M Y)_{Ar} - \Gamma^r_{AA} (P^*_M Y)_{rr} - \Gamma^A_{Ar}(P^*_M Y)_{AA}$ and  \[\nabla_r (P^*_M Y)_{AA} = \pdv_r(P^*_M Y)_{AA} - 2\Gamma^A_{rA} (P^*_M Y)_{AA},\] 

we find that for our unsmoothed operator $S$ and tensor $h^r \in \wH^s$, we have components
\begin{align*}
(S B_{\eta_1}h^r)^{rr}(w) &= \int_t^1 -\Omega^2\sinh(r)\cosh(r)g_{AA} \frac{K(s,w,y_1)}{1-s} B_{\eta_1}h^r\left(r,\frac{A(w)-sA(y_1)}{1-s},B'\right)ds ,  \\
(S B_{\eta_1}h^r)^{Ar}(w) &= \int_t^1 -\Omega^2\left(\frac{2}{\Omega}\pdv_{s}\right)K(s,w,y_1)B_{\eta_1}h^A\left(r,\frac{A(w)-sA(y_1)}{1-s},B'\right)\frac{1}{1-s}ds  \\
&+ \delta_{w} \int \Omega \left(\frac{\sin(\sqrt{g_{AA}}(A(y)-A(w))(t-1))}{\omega_M}\right) B_{\eta_1}h(y) \dd y. \\
(S B_{\eta_1}h^r)^{AA}(w) &= \int_t^{1} -\Omega^2 ( \pdv_r + 2\coth(r))K(s,w,y_1) B_{\eta_1}h^A \left(r,\frac{A(w)-sA(y_1)}{1-s},B'\right) \frac{1}{1-s}  \dd s,  \\
\end{align*}
where $K(s,w,y_1) \coloneqq \frac{\sin(\omega_M (t-s))}{\omega_M}$.

Again, the boundary term is infinitely smoothing and uniformly bounded as $O(1)$ and can hence be ignored.

Using the equation for $Y_A$, we can close the system exclusively in $Y_A$, which becomes the third order equation
\begin{align*}
\pdv_t^3 Y_A + \Omega^2 \pdv_t Y_A &= \sinh(r)\cosh(r)g_{AA}(\pdv_t^2(P^*_M Y)_{AA} + \Omega^2 ((P^*_M Y)_{AA} + \pdv_r (P^*_M Y)_{AA} \\
&- 2\pdv_A (P^*_M Y)_{Ar} - 2\coth(r) (P^*_M Y)_{AA} -\sinh(r)\cosh(r)g_{AA}(P^*_M Y)_{rr})).
\end{align*}
Solving for $\pdv_t Y_A$ and integrating in t yields particular solution
\[\int_0^{1} \left( \frac{-\Omega( \cos(\omega_R (s))-1)}{\omega_R^2} \right)\Ncal((P^*_Rf)_{AA}, (P^*_Rf)_{Ar})(x(y,y_1,s)) \dd s.\]

So for tensor $h^A \in \wH^s$, the unsmoothed operator $S$ has components

\begin{align*}
(S B_{\eta_1}h^A)^{rr}(w) &= \int_t^1 -\Omega^2(\sinh(r)\cosh(r)g_{AA})^2 \frac{K(s,w,y_1)}{1-s} B_{\eta_1}h^r\left(r,\frac{A(w)-sA(y_1)}{1-s},B'\right)ds ,  \\
(S B_{\eta_1}h^A)^{Ar}(w) &= \int_t^1 2 \sinh(r)\cosh(r)g_{AA}\frac{\pdv_{s}^2 K(s,w,y_1)}{1-s}B_{\eta_1}h^A\left(r,\frac{A(w)-sA(y_1)}{1-s},B'\right)ds  \\
&+ \text{bdry terms}. \\
(S B_{\eta_1}h^A)^{AA}(w) &= \int_t^{1} -\sinh(r)\cosh(r) \Omega^2 \frac{( \pdv_r + 2\coth(r) + \frac{\pdv_s^2}{\Omega^2}) K(s,w,y_1)}{1-s} \\
&\cdot B_{\eta_1}h^A \left(r,\frac{A(w)-sA(y_1)}{1-s},B'\right)  \dd s + \text{bdry terms},  \\
\end{align*}
where $K(s,w,y_1) = \left( \frac{-\Omega( \cos(\omega_R (s))-1)}{\omega_R^2} \right)$.

Note that the boundary terms are treated identically to the earlier cases, with one notable exception: integrating by parts creates terms of the form
\[\int \eta(w) \pdv_s (P^*_M Y)_{Ar}(w) \int \Omega \pdv_s K(1,y,w) \dd y \dd w  \]
and likewise with $(P^*_M Y)_{AA}$.
Such terms can be handled by replacing $\pdv_s$ with $-\Omega \pdv_{A(y)} = \Omega \pdv_{A(w)}$, which can be safely integrated by parts onto $\eta$ with no boundary terms, and we obtain the same infinitely smoothing estimate as before.

Thus we can once again ignore the boundary terms for the remainder of this discussion.

We now present the relevant bounds for the solution operator and remainder.

\begin{proposition}\label{prop:momentum_second_op_est}
The smoothed operator $S^2_Y$ as above defined by integrating $\eta_2(y_1)$ against $S$ 
and the remainder $B_{\eta_2,Y}$ given by
\begin{align*}
B_{\eta_2}B_{\eta_1}h^r(y) &= \int \eta_2(y_1) [\cos(\omega_M(t-1)) B_{\eta_1}h^r(y_1) - \frac{\sin(\omega_M(t-1))}{\omega_M} \pdv_s B_{\eta_1}h^r(y_1)] \dd y_1 \\
B_{\eta_2}B_{\eta_1}h^A(y) &= \int \eta_2(y_1) -\Omega \sinh(r)\cosh(r) \left(\frac{\sin(\omega_M)}{\omega_M}B_{\eta_1}h^A(y_1) + \frac{\cos(\omega_M)-1}{\omega_M^2}\pdv_s B_{\eta_1}h^A(y_1) \right) \dd y_1
\end{align*}

satisfy the bounds
\[\|S^2_YB_{\eta_1}h\|_{\wH^{s+1}} \lesssim \delta^{-4} \|B_{\eta_1}h\|_{\wH^{s}}\]
for every $s > 0$ and 
\[\|B_{\eta_2}B_{\eta_1}h\|_{\wH^{s+1}} \lesssim \delta^{-2}\|B_{\eta_1}h\|_{\wH^{s}} \lesssim \delta^{-2}\|h\|_{\wH^s}\]
for every $s \geq 1$.
\end{proposition}

\begin{proof}
Observe that $K(s) \sim O(1)$ and $\pdv_s K(s) \sim O(1)$, just as in the Hamiltonian constraint equation case.
In addition, up to leading order in $\delta$, the term $\pdv_r K(s) = \frac{(t-s)\sinh(2r)\cos(\omega_M(t-s))}{\omega_M} \sim O(1)$.

The $K(s)$ term are $O(1)$ and they are associated to at least one angular component of $SB_{\eta_1}h$.
The most problematic term is the $rr$ component of the operator on $h^A$, which contributes a factor of $\delta^{-4}$ from the $(\sinh(r)\cosh(r))^2$ factor.
By, Lemma~\ref{lem:averaging_sobolev_estimate} we therefore have
\[\|S^2_YB_{\eta_1}h\|_{\wH^{s+1}} \lesssim  \delta^{-4} \|B_{\eta_1}h\|_{\wH^{s+1}} \lesssim \delta^{-5} \|B_{\eta_1}h\|_{\wH^{s+1}}\]
for every $s > 0$.
Notably, the only non-negligible contribution is from the $rr$-component.

Given that the frequency is also $O(1)$, we can apply Lemma~\ref{lem:fio_estimate} and the radial regularity improvement of $1$ in the momentum constraint case with no factors of $\delta$ to upgrade this to an estimate 
\[\|S^2_YB_{\eta_1}h\|_{\wH^{s+1}} \lesssim  \delta^{-4} \|B_{\eta_1}h\|_{\wH^{s}}\]
where we gain only one regularity due to the $Ar$ component appearing with no factors of $s^k$.
This saves one factor of $\delta$, and produces the desired estimate.

Moreover, the trick from Lemma~\ref{lem:averaging_sobolev_estimate} tells us that $([S^2_Y, \pdv_{A(w)}]B_{\eta_1}h^r)^{rr}$ is given by
\[\int \pdv_{A(y_1)}\eta_2(y_1) \int_t^{1} \Omega^2 (-\sinh(r)\cosh(r))K(s,w,y_1) B_{\eta_1}h^A \left(r, \frac{(A(w) - sA(y_1))}{1-s},B'\right) \frac{1}{1-s}  \dd s \dd y_1,\]
namely the same solution operator expression with a derivative harmlessly falling on $\eta_2$.
Likewise, the other operator components satisfy precisely the same structure; moreover, $B'$ and $r$ derivatives pass through the solution operator and are killed by the commutator.
The $h^A$ solution operator has the exact same structure, but with a factor of $(\sinh(r)\cosh(r))^2$.

We therefore obtain the same $O(1)$ estimate
\[\|[S^2_Y,\pdv_{\omega}]B_{\eta_1}h\|_{\wH^{s}} \lesssim \delta^{-4}\|B_{\eta_1}h\|_{\wH^{s}}.\]

To obtain the bound on $B_{\eta_2}$, we follow the same strategy as in the proof of Lemma~\ref{lem:remainder_second_operator}, and hence because each derivative only falls on a single remainder term, it suffices to consider one particular direction $A$. 
We have \[\langle b_{\eta_2,Y},Y_r \rangle = \cos(\omega_M(t-1))Y_r(y_1) - \frac{\sin(\omega_M(t-1))}{\omega_M}Y_r'(y_1).\]
Denoting the coefficients by $c_1$ and $c_2$ as before, we find
\[B_{\eta_2}B_{\eta_1}h^r(y) = \int \eta_2(y_1) [c_1 B_{\eta_1}h^r(y_1) - c_2 \pdv_s B_{\eta_1}h^r(y_1)] \dd y_1.\]
Then as $\pdv_{A(y)}B_{\eta_2}B_{\eta_1}h^r(y_1) = \int \eta_2(y_1) \pdv_{A(y)} c_1 (B_{\eta_1}h^r(y_1)) - \pdv_{A(y)} c_2 (\pdv_s B_{\eta_1}h^r(y)) \dd y_1$ and observing that as in the Hamiltonian constraint case, both $c_1$ and $c_2$ and their derivatives are uniformly bounded and $O(1)$, applying Sobolev embedding to $\eta_2$ and using $\sqrt{g_{AA}} \sim O(1)$ gives

\[|\pdv_{A(y)}^k B_{\eta_2}B_{\eta_1}h| \lesssim \|\eta_2\|_{\wH^{d/2}} \|\pdv_{A(y)}^k(c_1 + c_2)\|_{L^{\infty}} \|B_{\eta_1}h\|_{\wH^1}.\]
Here, the derivatives of $c_1 + c_2$ are bounded of size $O(1)$ in $L^\infty$; the extra derivative on $c_2$ gives an estimate of the source in $\wH^1$.

Consequently, we find that 
\[\|B_{\eta_2}B_{\eta_1}h^r\|_{\wH^{s+1}} \lesssim \|B_{\eta_1}h^r\|_{\wH^{1}} \lesssim \|h^r\|_{\wH^1}.\]
with no factors of $\delta$ from our estimate, and the same level estimate for the first solution operator.

For the angular component $A$, we use $Y_A' = -\Omega( -(P^*_M Y)_{AA} - \sinh(r)\cosh(r)g_{AA}Y_r)$ and integrate. 
Excluding terms which are absorbed in the operator $S$ we have already bounded, we find 
\[\langle b_{\eta_2,Y},Y_A \rangle = -\Omega \sinh(r)\cosh(r)g_{AA} \left(\frac{\sin(\omega_M)}{\omega_M}Y_r(y_1) + \frac{\cos(\omega_M)-1}{\omega_M^2}Y_r'(y_1) \right).  \]
Both terms are of $O(1/\delta^2)$.

Hence for each leg, $B_{\eta_2,Y}^{(A)}$ satisfies 
\[\|B_{\eta_2,Y}^{(A)}B_{\eta_1,Y}h\|_{\wH^{s+1}} \lesssim \delta^{-2} \|B_{\eta_1,Y}h\|_{\wH^{1}} \lesssim \delta^{-2} \|h\|_{\wH^1}. \]
But attempting to use the chain of estimates on legs produces compounding factors of $\delta$, as 
\[B_{\eta_2,Y}B_{\eta_1}h \coloneqq B_{\eta_2,Y}^{(r)}B_{\eta_1,Y}h + B_{\eta_2,Y}^{(A)}B_{\eta_2,Y}^{(r)}B_{\eta_1,Y}h + \dots + B_{\eta_2,Y}^{(B'_{d-2})}\cdots B_{\eta_2,Y}^{(A)}B_{\eta_2,Y}^{(r)}B_{\eta_1,Y}h\]
would cause each term to contain this problematic $\sinh(r)\cosh(r)$ at every step.
However, recall that due to constant curvature and the structure of the equations, the transport and consequently the bounds themselves are path-independent, hence we can take a straight path directly from each intermediate point to $y_1$.

Consequently, the bound on the entire operator is the same as the bound on one leg, and hence
\[ \|B_{\eta_2,Y}B_{\eta_1,Y}h^A\|_{\wH^1} \lesssim \delta^{-2} \|B_{\eta_1,Y}h^r\|_{\wH^1} \]
and
\[ \|B_{\eta_2,Y}B_{\eta_1,Y}h\|_{\wH^{s+1}} \lesssim \delta^{-2} \|B_{\eta_1,Y}h\|_{\wH^{s}} \lesssim \delta^{-2} \|h\|_{\wH^s}. \]

For the commutator estimate, we can apply the same trick as for the operator $B_{\eta_2,Y}$ itself to find that the first term in $B_{\eta_2,Y}\pdv_{A(y)} B_{\eta_1}h - \pdv_{A(y)}B_{\eta_2,Y}B_{\eta_1}h$ costs one derivative on the source, while the latter remains $O(1)$.
Hence for $s > 2$,
\[  \|[B_{\eta_2,Y},\pdv_{\omega}]B_{\eta_1}h\|_{\wH^{s}} \lesssim  \delta^{-2} \|B_{\eta_1}h\|_{\wH^{{s+1}}} \lesssim \delta^{-3} \|h\|_{\wH^s}.  \]
Note that this is okay, because we do not expect to gain regularity for the commutator.


\end{proof}

Thus, we have solved $PS^2 h = B_{\eta_1}h - B_{\eta_2}B_{\eta_1}h$, as a second correction to the first $PS_1h = h - B_{\eta_1}h$ with operator bounds
\[\|S^2_f\|_{\wH^{s}\to \wH^{s+2}} \lesssim O(\delta^{-2}); \qquad \|B_{\eta_2,f}\|_{\wH^{2}\to \wH^{s+2}} \lesssim O(1)\] 
and 
\[\|S^2_Y\|_{\wH^{s}\to \wH^{s+1}} \lesssim O(\delta^{-4}); \qquad \|B_{\eta_2,Y}\|_{\wH^{1}\to \wH^{s+1}} \lesssim O(\delta^{-2}).\] 
with commutator bounds
\[\|[S^2_f, \pdv_{z^i}]\|_{\wH^{s}\to \wH^{s+1}} \lesssim O(\delta^{-2}); \qquad \|[B_{\eta_2,f},\pdv_{z^i}]\|_{\wH^{3}\to \wH^{s}} \lesssim O(1)\] 
and 
\[\|[S^2_Y, \pdv_{z^i}]\|_{\wH^{s}\to \wH^{s}} \lesssim O(\delta^{-4}); \qquad \|[B_{\eta_2,Y},\pdv_{z^i}]\|_{\wH^{2}\to \wH^{s}} \lesssim O(\delta^{-3}).\] 

It is worth noting that already, Theorem~\ref{thm:sol_op_5.15} and our explicit computation tell us $B_{\eta_2}B_{\eta_1}h$ is actually of the form 
\[\sum_{A \in \{1 \dots, \dim \ker P^*\}} w_A \langle Z^A, B_{\eta_1}h \rangle,\] 
where $w_A \in C_c^{\infty}(\supp \eta_2; \C^{2})$ are smooth functions and $Z^A$ are the cokernel elements.

But observe that $B_{\eta_1} = I-PS$, hence we have
\[w_A \langle Z^A, B_{\eta_1}h \rangle = w_A \langle Z^A, h \rangle.\]

There is a technical point, however, that because we split $A_{2/\delta}$ into two distinct regions to run the analysis, the cokernel elements on the top and bottom need not be compatible.
Thus we need to apply the theorem one more time.

\textbf{The Third Operator: Cokernel}
We now seek a solution to $PSh = B_{\eta_2}B_{\eta_1}h$, where we recall $\eta_2 = \eta_2^1 + \eta_2^2$ is supported in two patches of order $1/\delta$ on the top and bottom of $A_{2/\delta}$. 
We finish with a third application of the theorem as follows:

Define the coordinate transformation
\begin{equation}\label{eqn:stereo_proj_right}
\begin{aligned}
\Phi &\colon A_{\mathrm{right}} \to \hat{C} \\
\Phi(x',x_d) &= \left(\frac{x'}{||x||_2 + x_1}, ||x||_2 - \frac{1}{\delta} \right) \coloneqq (z',z_d).
\end{aligned}
\end{equation}
Where $A_{\mathrm{right}}$ is defined to be $\{x^1,\dots, x^d \in A: x^1 > -c\}$, $c >1$ is some positive constant of order $O(1/\delta)$, and $\hat{C}$ is another approximately unit cylinder rescaling of $A$.

Certainly, $\supp \eta_2 \subseteq A_{\mathrm{right}}$, hence one can again take a patch $B$ of size $1/\delta$ in a neighborhood of 
\[\left(\frac{1}{2}(\log(2/\delta) + \log(8/\delta)), 0,\dots, 0\right)\] 
and take paths $x(y,y_1,s)$ from any point $y \in \supp \eta_2$ to $y_1 \in B$.
These paths will again be a combination of angular and radial paths, with angular paths of geometric length $O(1/\delta)$.

Thus, we can then apply Theorem~\ref{thm:sol_op_5.15} and Lemmas~\ref{lem:averaging_sobolev_estimate} and \ref{lem:fio_estimate} once more, observe that $B_{\eta_2}$ and $B_{\eta_1}$ do not affect the inner product, and conclude that we have finitely many cokernel elements
\[w_A \langle Z^A, B_{\eta_1}h \rangle = w_A \langle Z^A, h \rangle,\]
where now $Z^A$ are defined globally on $A$ without compatibility concerns.
There is no issue using the gnomonic coordinates in Lemma~\ref{lem:fio_estimate} because they must hold for a fixed $y$; even though we are taking a region larger than a hemisphere, the paths do not traverse the entire length.

Hence in our third operator step we obtain precisely the same bounds as in the case of the second solution operator:

\[\|S^3_f\|_{\wH^{s}\to \wH^{s+2}} \lesssim O(\delta^{-2}); \qquad \|S^3_Y\|_{\wH^{s}\to \wH^{s+1}} \lesssim O(\delta^{-4})\] 
with commutator bounds
\[\|[S^3_f, \pdv_{z^i}]\|_{\wH^{s}\to \wH^{s+1}} \lesssim O(\delta^{-2}); \qquad \|[S^3_Y, \pdv_{z^i}]\|_{\wH^{s+1}\to \wH^{s}} \lesssim O(\delta^{-4}).\] 

With our bounds for each individual operator complete, we can finally finish the proof of the solution operator bound.

\begin{proof}[Proof of Proposition~\ref{prop:soln_op_bounds}]
With our three solution operators, we have 
\begin{align*}
P_R(S^1_f+S^2_f+S^3_f)h &= h - w_A \langle Z^A,h \rangle \\
P_M(S^1_Y+S^2_Y+S^3_Y)h^i &= h^i - w_A \langle Z^A,h^i \rangle \\
\end{align*}
and the bounds on the entire solution operator $S = (S_f, S_Y) = (S^1_f+S^2_f+S^3_f,S^1_Y+S^2_Y+S^3_Y)$ are simply the maximum of each individual operator.
Combining the estimates obtained from Lemma~\ref{lem:commuting_angular}, Lemma~\ref{lem:fio_estimate} and Proposition~\ref{prop:momentum_second_op_est}, we find 
\[\|S_f\|_{\wH^{s}\to \wH^{s+2}} \lesssim O(1/\delta^2) \qquad \|S_Y\|_{\wH^{s}\to \wH^{s+1}} \lesssim O(1/\delta^4),\] 
completing the proof.
\end{proof} 

In addition, the remainders $w_A \langle Z^A, \cdot \rangle$ satisfy 
\[\|w_A \langle Z^A, g \rangle \|_{\wH^s} \lesssim \|w_A \|_{\wH^s} |\langle Z^A, g \rangle| \lesssim \|w_A \|_{\wH^s} | \langle Z^A, g \rangle|\]
where the latter is given by integration of $\eta_{3}$ against uniform $O(1)$ coefficients applied to derivatives of $g$, which can be integrated by parts onto the smoothing at no cost.
By construction, the functions $w_A$ are of sinusoidal form of $\sin(\omega), \cos(\omega)$ with rescaling by metric components; they are uniformly $O(1)$ in the Hamiltonian constraint case, while they are uniformly $O(\delta^{-2})$ in the momentum constraint case.

\subsection{Fixed Point Iteration}
\textbf{Step 4}

With solution operator bounds obtained, we are ready to begin the first fixed point argument. 
For convenience, we will first write the equations on $A$ equipped with the anisotropic Sobolev norms $\wH^s$, then translate to the unit scale picture $C$ with standard $H^s$ norm and run the fixed-point argument there where estimates are easier. 
Recall that $A$ is the euclidean annular approximation of $A_{2/\delta}$ and $\widetilde{A}$ is the approximation of $\widetilde{A}_{2/\delta}$.

In this case, Equation~\eqref{eqn:constraints} becomes

\begin{equation}\label{eqn:constraints_thinx}
\begin{aligned}
\begin{cases}
\pdv_{x^i} \pdv_{x^j} \sigma^{ij} - \Gamma^k_{ij}\pdv_k \sigma^{ij} - b_{ij} \sigma^{ij} = M^{(0)}_{\sigma}(\pi,\pi) + M^{(1)}_{\sigma}(\pdv_x \sigma, \pdv_x \sigma) + M^{(2)}_{\sigma}(\sigma, \pdv_x^2 \sigma) \\
\pdv_{x^i} \pi^{ij} + \Gamma^i_{ik}\pi^{kj} + \Gamma^j_{ik} \pi^{ik} = N^{(1)j}_{\sigma}(\sigma, \pdv_x \pi) + N^{(0)j}_{\sigma}(\pdv_x \sigma, \pi)
\end{cases}
\end{aligned}   
\end{equation}
where the Christoffel symbols and metric are of size $O(1)$ in the x coordinates.

Let us begin with a naive first trial for our initial data set.
Fix $\chi(r)$ to be a smooth radial function which is $0$ for $r < \frac{1}{\delta}$ and $1$ for $r > \frac{1}{\delta} + 1$.

As in \cite{mao2023}, we will let $(g_{in; Q}, k_{in;Q}, g_{out;Q}, k_{out;Q})$ either refer to $(\cg,\ck)$ or $(g_Q, k_Q)$ for interior and exterior gluing, respectively. 
We will do the same for $(\sigma_{in;Q}, \pi_{in;Q}, \sigma_{out;Q}, \pi_{out;Q})$.

Our first naive guess will be 
\begin{equation}\label{eqn:naive_guess}
    (\bsigma, \bpi) = (1-\chi)((\sigma_{in;Q}, \pi_{in;Q}) + \chi (\sigma_{out;Q}, \pi_{out;Q}).
\end{equation}
Naturally, this will not be a solution to Equation~\eqref{eqn:constraints_thinx}.
Instead, we wish to find a corrected solution $(\tsigma_Q,\tpi_Q)$ supported in $A$ and a choice of $Q \in \Qcal$ so that $(\sigma, \pi) = (\bsigma_Q+\tsigma_Q, \bpi_Q+\tpi_Q)$ is actually a solution.

So we see that, analogously to \cite{mao2023}, $(\tsigma_Q, \tpi_Q)$ must solve 
\begin{equation}\label{eqn:fixed_point_correction_ham}
\begin{aligned}
\pdv_{x^i} \pdv_{x^j} \tsigma^{ij}_Q - \Gamma^k_{ij}\pdv_{x^k} \tsigma^{ij}_Q - b_{ij} \tsigma^{ij}_Q &= F_Q + M^{(0)}_{\bsigma_Q + \tsigma_Q}(\bpi_Q + \tpi_Q,\bpi_Q+\tpi_Q) - M^{(0)}_{\bsigma_Q}(\bpi_Q,\bpi_Q) \\
&+ M^{(1)}_{\bsigma_Q + \tsigma_Q}(\pdv_x (\bsigma_Q + \tsigma_Q), \pdv_x (\bsigma_Q + \tsigma_Q)) - M^{(1)}_{\bsigma_Q}(\pdv_x (\bsigma_Q), \pdv_x (\bsigma_Q)) \\
&+ M^{(2)}_{\bsigma_Q + \tsigma_Q}(\bsigma_Q + \tsigma_Q, \pdv_x^2 (\bsigma_Q + \tsigma_Q)) - M^{(2)}_{\bsigma_Q}(\bsigma_Q, \pdv_x^2 (\bsigma_Q)),
\end{aligned}
\end{equation}

and 

\begin{equation}\label{eqn_fixed_point_correction_mom}
\begin{aligned}
\pdv_{x^i} \tpi^{ij}_Q + \Gamma^i_{ik}\tpi^{kj}_Q + \Gamma^j_{ik} \tpi^{ik}_Q &= G_Q + N^{(1)j}_{\bsigma_Q + \tsigma_Q}(\bsigma_Q + \tsigma_Q, \pdv_x (\bpi_Q + \tpi_Q)) - N^{(1)j}_{\bsigma_Q}(\bsigma_Q, \pdv_x (\bpi_Q)) \\
&+ N^{(0)j}_{\bsigma_Q + \tsigma_Q}(\pdv_x (\bsigma_Q + \tsigma_Q), \bpi_Q + \tpi_Q) - N^{(0)j}_{\bsigma_Q}(\pdv_x (\bsigma_Q), \bpi_Q).
\end{aligned}
\end{equation}

Here, $F_Q$ and $G_Q$ are the errors from plugging the naive solution into Equation~\eqref{eqn:constraints_thinx}, given explicitly by
\begin{align*}
F_Q &= -\pdv_{x^i} \pdv_{x^j} \bsigma^{ij}_Q + \Gamma^k_{ij}\pdv_{x^k} \bsigma^{ij}_Q + b_{ij} \bsigma^{ij}_Q + M^{(0)}_{\bsigma_Q}(\bpi_Q,\bpi_Q) \\
&+ M^{(1)}_{\bsigma_Q}(\pdv_x (\bsigma_Q), \pdv_x (\bsigma_Q)) + M^{(2)}_{\bsigma_Q}(\bsigma_Q, \pdv_x^2 (\bsigma_Q))\, \\
G_Q &= -\pdv_{x^i} \bpi^{ij}_Q - \Gamma^i_{ik}\bpi^{kj}_Q - \Gamma^j_{ik} \bpi^{ik}_Q + N^{(1)j}_{\bsigma_Q}(\bsigma_Q, \pdv_x (\bpi_Q)) + N^{(0)j}_{\bsigma_Q}(\pdv_x (\bsigma_Q, \bpi_Q)).
\end{align*}

Then denoting the nonlinearities by $M_Q$, $N_Q$, we see that $(\tsigma_Q, \tpi_Q)$ must solve the fixed point problem
\begin{equation}\label{eqn:fixed_point_sigmapi}
(\tsigma_Q, \tpi_Q) = (S_f(F_Q + M_Q), S_{Y}(G_Q + N_Q) ).
\end{equation}

Now let us translate to the unit-scale picture.
Recall that for the angular components, $z' \approx \delta x'$. 
Hence, $dz' = \delta dx'$, and $\pdv_{z'} = \frac{1}{\delta} \pdv_{x'}$.
Consequently, $\sigma^{z_i'z_j'} = \delta^2 \sigma^{x_i'x_j'}$. 
Thus, while the equations in the $z$ coordinates appear the same:
\begin{equation}
\begin{aligned}
\pdv_{z^i} \pdv_{z^j} \tsigma^{ij}_Q - \Gamma^k_{ij}\pdv_{z^k} \tsigma^{ij}_Q - b_{ij} \tsigma^{ij}_Q &= F_Q + M^{(0)}_{\bsigma_Q + \tsigma_Q}(\bpi_Q + \tpi_Q,\bpi_Q+\tpi_Q) - M^{(0)}_{\bsigma_Q}(\bpi_Q,\bpi_Q) \\
&+ M^{(1)}_{\bsigma_Q + \tsigma_Q}(\pdv_z (\bsigma_Q + \tsigma_Q), \pdv_z (\bsigma_Q + \tsigma_Q)) - M^{(1)}_{\bsigma_Q}(\pdv_z (\bsigma_Q), \pdv_z (\bsigma_Q)) \\
&+ M^{(2)}_{\bsigma_Q + \tsigma_Q}(\bsigma_Q + \tsigma_Q, \pdv_z^2 (\bsigma_Q + \tsigma_Q)) - M^{(2)}_{\bsigma_Q}(\bsigma_Q, \pdv_z^2 (\bsigma_Q)),
\end{aligned}
\end{equation}
\begin{equation}
\begin{aligned}
\pdv_{z^i} \tpi^{ij}_Q + \Gamma^i_{ik}\tpi^{kj}_Q + \Gamma^j_{ik} \tpi^{ik}_Q &= G_Q + N^{(1)j}_{\bsigma_Q + \tsigma_Q}(\bsigma_Q + \tsigma_Q, \pdv_z (\bpi_Q + \tpi_Q)) - N^{(1)j}_{\bsigma_Q}(\bsigma_Q, \pdv_z (\bpi_Q)) \\
&+ N^{(0)j}_{\bsigma_Q + \tsigma_Q}(\pdv_z (\bsigma_Q + \tsigma_Q), \bpi_Q + \tpi_Q) - N^{(0)j}_{\bsigma_Q}(\pdv_z (\bsigma_Q), \bpi_Q).
\end{aligned}
\end{equation}
there are extra factors of $\frac{1}{\delta}$ in the nonlinear terms that need to be taken into account.

Notice that on the left-hand side, the only problematic Christoffel term the one of the form $\Gamma^r_{AA} = \sinh(r)\cosh(r)g_{AA}$, which only appear when contracting against two angular components and we lose two factors of $\delta$.
However, this is to be expected; the $\delta^{-2}$ loss appears in both the $x$ and $z$ coordinates, and is the reason why we need a hierarchy on our estimates in moving from $\bsigma,\bpi$ to $F_Q, G_Q$, specifically to handle the $AA$ components of the tensors.

To obtain the sharpest conditions possible, we need to keep track of components.
Observe that at the level of the solution operator directly, we have proved the following bounds:

\[\begin{cases}
||(S_f)^{rr}||_{\wH^s \to \wH^{s+2}} \lesssim \delta^{-2} \\
||(S_f)^{Ar}||_{\wH^s \to \wH^{s+2}} \lesssim \delta^{-2} \\
||(S_f)^{AA}||_{\wH^s \to \wH^{s+2}} \lesssim O(1) \\
\end{cases}
\qquad
\begin{cases}
||(S_Y)^{rr}||_{\wH^s \to \wH^{s+1}} \lesssim \delta^{-4} \\
||(S_Y)^{Ar}||_{\wH^s \to \wH^{s+1}} \lesssim \delta^{-2} \\
||(S_Y)^{AA}||_{\wH^s \to \wH^{s+1}} \lesssim \delta^{-2} \\
\end{cases} \]

When we convert from the $x$ coordinates to the $z$ coordinates on the right hand side, we lose the number of factors of $\delta$ corresponding to the number of raised angular components in the product of the nonlinear terms; for example, $T_{ABCD}^{rr} \pdv_r \pdv_r \sigma^{AB}\sigma^{CD}$, where $r,A$ are understood in $z$ coordinates $\pdv_{z_d} \pdv_{z_d} \sigma^{z_i'z_j'} \sigma^{z_k'z_l'}$ would cost $T_{ABCD}^{rr} \sim \delta^{-4}$ due to four raised angular components. 

Likewise, for the momentum constraint, the worst conversion would cost $\delta^{-3}$ from a term $\pdv_{z_d}\sigma^{z_i'z_j'} \pi^{z_k'z_l'}$.
However, the componentwise solution operator bounds are very useful here; by setting separate bounds for what components should be, we can improve our fixed point smallness estimates.

Indeed, the estimates above suggest we take 
\[ \|\tsigma^{AA}\|_{\wH^s} = \delta^2 \|\tsigma^{Ar,rr}\|_{\wH^s}; \qquad \|\tpi^{rr}\|_{\wH^s} = \delta^{-2}\|\tpi^{Ar,AA}\|_{\wH^s}\]
for all $s > 0$.
 
Thus, we set 
\[\begin{cases}
||\tsigma^{rr}||_{H^s(C)} \lesssim \delta^4 \ep \\
||\tsigma^{Ar}||_{H^s(C)} \lesssim \delta^4 \ep \\
||\tsigma^{AA}||_{H^s(C)} \lesssim \delta^6 \ep \\
\end{cases}
\qquad
\begin{cases}
||\tpi^{rr}||_{H^{s-1}(C)} \lesssim \delta^{3} \ep \\
||\tpi^{Ar}||_{H^{s-1}(C)} \lesssim \delta^{5} \ep \\
||\tpi^{AA}||_{H^{s-1}(C)} \lesssim \delta^{5} \ep \\
\end{cases} \]
Note that, after passing to unit scale coordinates $||\sigma||_{H^s}$ does not have the hierarchy for angular components in Definition~\ref{def:sobolev}, so we do not need to worry about derivatives or factors of $\delta$ per raised angular component.

We have two goals:
\begin{enumerate}
\item \label{goal_1} First, we wish to show that the map 
\[T \colon (\tsigma_Q^{(n)}, \tpi_Q^{(n)}) \mapsto (S_f(F_Q + M_Q^{(n)}), S_{Y}(G_Q + N_Q^{(n)}) )\]
maps a ball of suitably small size to itself. \item \label{goal_2} Second, we wish to show that the above map is a contraction.
\end{enumerate}

Our choices of the size of $\tsigma, \tpi$ will recover the desired factors of $\delta$ without requiring $\ep$ contains any dependence on $\delta$.
From the iteration scheme, if $||\tsigma^{AA}||_{H^s(C)} \lesssim \delta^6 \epsilon$, then a product of the form $\tsigma \cdot \tsigma$, losing $4$ factors of $\delta$ from the conversion, and $0$ for the $AA$ solution operator returns $\delta^6 \ep^2$, while losing $2$ for the $Ar$ and $rr$ components returns $\delta^4 \ep^2$.
Moreover, $\delta^4 \ep$ for the $rr,Ar$ component is optimal, as losing $2$ factors from the $Ar$ conversion and $2$ more from the solution operator indeed returns $\delta^4 \ep$.
By setting $||\tpi^{AA}||_{H^{s-1}(C)} \lesssim \delta^{5} \epsilon$, the product $\tsigma^{AB}\tpi^{CD}$ and losing $3+4$ factors to obtain $rr$ component returns $\delta^{4} \epsilon^2$.
The $\delta^3 \ep$ bound is also sharp, as a product $\tpi^{rr} \tpi^{rr}$ with no conversion nor solution operator loss needs to obtain $\delta^6 \ep$ for the $\tsigma^{AA}$ component.

\begin{remark}
Suppose we are in the time-symmetric case $k=0$, similar to that of \cite{chrusciel2007gluing}.
In this situation, there is no iteration scheme for $\pi$, and the smallness condition on $\sigma$ is $\|\bsigma\|_{H^s(C)} \lesssim \delta^6 \ep$, with this required hierarchy for components.
This estimate is to be directly compared to \cite[Corollary 3.5]{chrusciel2007gluing}, in which they obtain a solution operator bound of $\delta^{-4}$.
In this convention of norm, this would cost two extra factors of $\delta$, requiring an initial assumption of $\|\bsigma\|_{H^s(C)} \lesssim \delta^{8}\ep$, rather than $\delta^6\ep$.
\end{remark}

To accomplish goals \ref{goal_1} and \ref{goal_2}, we use the standard estimates as in \cite{mao2023}:

\begin{lemma}[Lemma 3.1 in \cite{mao2023}]\label{lem:moser_est}
Let $u,v$ be Schwarz functions on $\R^d$, $F \in C^{\infty}$ with bounded derivatives. 
Then for $s > 0$ we have 
\begin{enumerate}
\item \textbf{Moser Estimates.} $||F(u)-F(0)||_{H^s} \lesssim_{s,F,||u||_{L^{\infty}}} ||u||_{H^s}.$
\item \textbf{Moser Difference Estimates.}
\[||F(u)-F(v)||_{H^s} \lesssim_{s, F, ||u||_{L^{\infty}},||v||_{L^{\infty}}} ||u-v||_{H^s} + ||u-v||_{L^{\infty}}(||u||_{H^s} + ||v||_{H^s}).\]
\item \textbf{Product Estimates}. Fix $s_0, s_1, s_2$ such that $s_0 + s_1 + s_2 \geq d/2$, $s_0+s_1+s_2 \geq \max\{s_0,s_1,s_2\}$ with at least one of the two inequalities strict, then 
\[||u \cdot v||_{H^{-s_0}} \lesssim_{s_0,s_1,s_2} ||u||_{H^{s_1}} ||v||_{H^{s_2}}. \]
\end{enumerate}
\end{lemma}

Applying the estimates with $(s_0,s_1,s_2) = (2-s,s,s-2)$ and $(2-s,s-1,s-1)$ for $\sigma$ and $\pi$ nonlinear terms, and using the assumptions in Theorem \ref{thm:gluing_linear}, we find that 
\[\supp(F_Q,G_Q) \subseteq C; \quad
||F_Q||_{H^{s-2}(C)} \lesssim \delta^6 \ep; \quad
||G_Q||_{H^{s-2}(C)} \lesssim \delta^{7} \ep,\] 
since the quadratic terms of $F_Q,G_Q$ are all strictly smaller.

On the other hand, the nonlinear terms are all designed so that after applying the conversion they all satisfy all satisfy
\[||M_Q||_{H^{s-2}(C)} \lesssim \delta^{6} \ep^2; \qquad ||N_Q||_{H^{s-2}(C)} \lesssim_{M_c} \delta^{7} \ep^2. \]
Moreover, they also satisfy the support property $\supp(M_Q,N_Q) \subseteq C$.

Choosing $\ep = \ep_0$ with no dependence on $\delta$ therefore ensures that for $||\tsigma^{(n)}||_{H^s(C)} \lesssim \delta^4 \ep$ and , $||\tpi^{(n)}||_{H^{s-1}(C)} \lesssim \delta^{3} \ep$ with the above hierarchy, we have
\[ ||\tsigma^{(n+1)}||_{H^s(C)} = ||S_f(F_Q + M_Q^{(n)})||_{H^s(C)} \lesssim \delta^4 \ep\]
and
\[\|\tpi^{(n+1)}\|_{H^{s-1}(C)} = \|S_{Y}(G_Q + N_Q^{(n)})\|_{H^{s-1}(C)} \lesssim \delta^{3} \ep,\]
with precisely the same component structure after using the bound on the solution operator $S$.

Note that at each step we ensure that we have a uniform bound of $\delta^7 \ep, \delta^{13/2} \ep$ for the Hamiltonian and Momentum constraint equations respectively across all tensor components to recover the bound; it is not possible to improve this because our estimates involve the scalar quantities $F_Q,G_Q,M_Q,N_Q$.

Finally, we observe that our choice of $\ep = \ep_0$ means that if we fix $(\tsigma_Q', \tpi_Q')$ such that $\supp(\tsigma_Q',\tpi_Q') \subseteq C$ and 

\[\begin{cases}
||\tsigma_Q'^{rr}||_{H^s(C)} \lesssim \delta^4 \ep \\
||\tsigma_Q'^{Ar}||_{H^s(C)} \lesssim \delta^4 \ep \\
||\tsigma_Q'^{AA}||_{H^s(C)} \lesssim \delta^6 \ep \\
\end{cases}
\qquad
\begin{cases}
||\tpi_Q'^{rr}||_{H^{s-1}(C)} \lesssim \delta^{3} \ep \\
||\tpi_Q'^{Ar}||_{H^{s-1}(C)} \lesssim \delta^{5} \ep \\
||\tpi_Q'^{AA}||_{H^{s-1}(C)} \lesssim \delta^{5} \ep \\
\end{cases} \]
then
\begin{align*}
&||(F_Q + M_Q)[(\tsigma_Q, \tpi_Q)] - (F_Q + M_Q)[(\tsigma_Q', \tpi_Q')]||_{H^{s-2}} \lesssim  \delta^6\ep ||(\tsigma_Q - \tsigma_Q', \tpi_Q - \tpi_Q')||_{H^s \times H^{s-1}(C)}. \\
&||(G_Q + N_Q)[(\tsigma_Q, \tpi_Q)] - (G_Q + N_Q)[(\tsigma_Q', \tpi_Q')]||_{H^{s-2}} \lesssim \delta^7 \ep ||(\tsigma_Q - \tsigma_Q', \tpi_Q - \tpi_Q')||_{H^s \times H^{s-1}(C)}.
\end{align*}
which holds componentwise for all $rr,Ar,AB$.

Here we can avoid problematic factors of $\delta$ by estimating in the ball with the natural distance
\[d((\sigma,\pi),(\sigma',\pi')) \coloneqq
\max\left(
\begin{aligned}
&\delta^{-4}\|\sigma^{rr} - \sigma'^{rr}\|_{H^s},\\
&\delta^{-4}\|\sigma^{Ar} - \sigma'^{Ar}\|_{H^s},\\
&\delta^{-6}\|\sigma^{AA} - \sigma'^{AA}\|_{H^s},\\
&\delta^{-3}\|\pi^{rr} - \pi'^{rr}\|_{H^{s-1}},\\
&\delta^{-5}\|\pi^{Ar} - \pi'^{Ar}\|_{H^{s-1}},\\
&\delta^{-5}\|\pi^{AA} - \pi'^{AA}\|_{H^{s-1}}
\end{aligned} \right).\]

Thus as long as $\ep$ is sufficiently small, contraction mapping principle shows that for every $Q \in \Qcal$ there is a unique $(\tsigma_Q, \tpi_Q)$ solving the projected Equation~\eqref{eqn:constraints_thinx} away from the cokernel that vanishes outside of $A$.
Moreover, the map $Q \mapsto (\tsigma_Q, \tpi_Q)$ is locally Lipschitz.

\textbf{Step 5} 

We now wish to prescribe $Q \in \Qcal$ so that the cokernel vanishes:
\begin{equation}\label{eqn:nonlinear_cokenel}
\begin{aligned}
&\int (F_Q + M_Q) (\sqrt{1+||x||_2}, x^1 \dots, x^d)^T dx = 0, \\
&\int (G_Q + N_Q) (\dots,     \sqrt{1+||x||_2} \delta_{jl} + x^l \delta_{j0}, \dots,  x^i \delta_{jk} - x^k \delta_{ji})^T = 0.
\end{aligned}
\end{equation}
where we recall that the cokernel elements are functions of $z$ as in Lemma~\ref{lem:cok_on_AC}.
Note that these cokernel elements and their $z$-derivatives are all $O(1)$ for the Hamiltonian constraint, $O(1/\delta^2)$ in the momentum constraint case, and uniformly bounded. 

Let us introduce $\chi^C_{-1,1}$ as in Equation~\eqref{eqn:avg_C}, and observe that $\supp(\chi^C_{-1,1}(r))\subseteq \widetilde{C}$, with $\chi^C_{-1,1} \equiv 1$ on $C$.

Let us also isolate the linear and quadratic parts of $F_Q + M_Q$ and $G_Q+N_Q$, which we will write as $\widetilde{M_Q^2} - \pdv_{z^i} \pdv_{z^j} \bsigma_Q^{ij} + b_{ij} \bsigma_Q^{ij}$ and $\widetilde{N_Q^2} - \pdv_{z^i} \bpi_Q^{ij}$, respectively.
Hence after averaging, by applying Lemma \ref{lem:avg_charges} applied on $C$, the first line of Equation~\eqref{eqn:nonlinear_cokenel} becomes

\begin{align*}
    &\int -\chi^C_{-1,1}(r) (- \pdv_{z^i} \pdv_{z^j} \bsigma_Q^{ij} + b_{ij} \bsigma_Q^{ij}) \begin{pmatrix}
        \sqrt{1+||x||_2^2}(z) \\
        \vdots \\
        x^d(z) \\
    \end{pmatrix}
    dx + 
    \int_{\widetilde{C}}
    \chi^C_{-1,1}(r) \widetilde{M_Q^2} 
    \begin{pmatrix}
        \sqrt{1+||x||_2^2}(z) \\
        \vdots \\
        x^d(z) \\
    \end{pmatrix}
    dz  \\
    &=
    \begin{pmatrix}
    \Ef \\
    \Cf^1 \\
    \vdots \\
    \Cf^d \\
    \end{pmatrix}
    [(\sigma_{in; Q}, \pi_{in; Q}); C_{-1}] - 
    \begin{pmatrix}
    \Ef \\
    \Cf^1 \\
    \vdots \\
    \Cf^d \\
    \end{pmatrix}
    [(\sigma_{in; Q}, \pi_{in; Q}); C_{1}] + \int_{\widetilde{C}}
    \chi^C_{-1,1}(r) \widetilde{M_Q^2} 
    \begin{pmatrix}
        \sqrt{1+||x||_2^2}(z) \\
        \vdots \\
        x^d(z) \\
    \end{pmatrix}
    dz 
\end{align*}
Here, we use $C_{-1}$ and $C_1$ to denote the translations of $C$ in the $z_d$ coordinate by $\pm 1$. 
This is in analogy with the rescaling of the annulus $A_{1/\delta}$ and $A_{8/\delta}$.

Taking Equation~\eqref{eqn:constraints_thinx} for $(\sigma_{in}, \pi_{in})$, translating into $z$ coordinates, we can rewrite these charge terms with respect to $C$ (without any translations) by integration by parts:
\[\begin{pmatrix}
    \Ef \\
    \Cf^1 \\
    \vdots \\
    \Cf^d \\
    \end{pmatrix}
    [(\sigma_{in; Q}, \pi_{in; Q}); C_{-1}] = \begin{pmatrix}
    \Ef \\
    \Cf^1 \\
    \vdots \\
    \Cf^d \\
    \end{pmatrix}
    [(\sigma_{in; Q}, \pi_{in; Q}); C] + \int_{\widetilde{C}}
    \chi^C_{-1,0}(r) \widetilde{M_Q^2} 
    \begin{pmatrix}
        \sqrt{1+||x||_2^2}(z) \\
        \vdots \\
        x^d(z) \\
    \end{pmatrix}
    dz.   \]
Repeating a similar computation for $(\sigma_{out; Q}, \pi_{out;Q})$ lets us rewrite the first line of Equation~\eqref{eqn:nonlinear_cokenel} as 
\[\begin{pmatrix}
    \Ef \\
    \Cf^1 \\
    \vdots \\
    \Cf^d \\
    \end{pmatrix}
    [(\sigma_{in; Q}, \pi_{in; Q}); C] 
    - 
    \begin{pmatrix}
    \Ef \\
    \Cf^1 \\
    \vdots \\
    \Cf^d \\
    \end{pmatrix}
    [(\sigma_{out; Q}, \pi_{out; Q}); C] + 
    \begin{pmatrix}
    n[Q]_{\Ef} \\
    \vdots \\
    n[Q]_{\Cf^d} \\
    \end{pmatrix}
    \]
For 
\[\begin{pmatrix}
    n[Q]_{\Ef} \\
    \vdots \\
    n[Q]_{\Cf^d} \\
    \end{pmatrix} = 
    \int_{\widetilde{C}} (\chi^C_{-1,1}(r) \widetilde{M_Q^2} + \chi^C_{-1,0}(r) \widetilde{M_{in;Q}^2} + \chi^C_{0,1}(r) \widetilde{M_{out;Q}^2}) \begin{pmatrix}
        \sqrt{1+||x||_2^2}(z) \\
        \vdots \\
        x^d(z) \\
    \end{pmatrix}dz.\]

We can perform the identical computations for the Momentum equation to obtain that the second line of Equation~\eqref{eqn:nonlinear_cokenel} becomes
\[\begin{pmatrix}
    \Pf^1 \\
    \vdots \\
    \Jf^{d(d-1)/2} \\
    \end{pmatrix}
    [(\sigma_{in; Q}, \pi_{in; Q}); C] 
    - 
    \begin{pmatrix}
    \Pf^1 \\
    \vdots \\
    \Jf^{d(d-1)/2} \\
    \end{pmatrix}
    [(\sigma_{out; Q}, \pi_{out; Q}); C] + 
    \begin{pmatrix}
    n[Q]_{\Pf^1} \\
    \vdots \\
    n[Q]_{\Jf^{d(d-1)/2}} \\
    \end{pmatrix}
    \]
For
\[\begin{pmatrix}
    n[Q]_{\Pf^1} \\
    \vdots \\
    n[Q]_{\Jf^{d(d-1)/2}} \\
    \end{pmatrix} = 
    \int_{\widetilde{C}} (\chi^C_{-1,1}(r) \widetilde{N_Q^2} + \chi^C_{-1,0}(r) \widetilde{N_{in;Q}^2} + \chi^C_{0,1}(r) \widetilde{N_{out;Q}^2}) \begin{pmatrix}
        \sqrt{1+||x||_2^2}\delta^1_j + x^1 \delta_{j0} \\
        \vdots \\
        x^d \delta_{j(d-1)} - x^{(d-1)} \delta_{jd} \\
    \end{pmatrix}dz.\]

We claim that $|n[Q]|_w \lesssim \delta^6 \ep^2$ and $||n[\cdot]||_{Lip,w} \lesssim K \ep$ independent of $\delta$.

Recall we have the componentwise bounds $\| \tsigma^{rr}\|_{\wH^s}, \| \tsigma^{Ar}\|_{\wH^s} \lesssim \delta^4 \ep$, $\| \tsigma^{AA}\|_{\wH^s} \lesssim \delta^6 \ep$, $\| \tpi^{rr}\|_{\wH^s} \lesssim \delta^3 \ep$ and $\| \tpi^{Ar}\|_{\wH^s}, \| \tpi^{AA}\|_{\wH^s} \lesssim \delta^5 \ep$.
Our cokernel is also size $O(1)$ in the Hamiltonian constraint case, whereas it is $O(\delta^{-2})$ for the momentum constraint. 
Letting $K_F, K_G$ denote the Lipschitz constants of $Q \mapsto F_Q$ and $Q \mapsto G_Q$ respectively in $H^{s-2}(C)$.
Since the charges are pairings of $F_Q$ and $G_Q$ with the cokernel elements, the natural sizes of $K_F$ and $K_G$ are $O(1)$ and $O(\delta^2)$ respectively. 

Consequently, keeping into account the relative sizes of components and of the cokernel, we find that we have bounds 
\[ \|\Delta \tsigma^{rr,Ar}\| \lesssim \delta^{-2}|\Delta Q_H|, \quad \|\Delta \tsigma^{AA}\| \lesssim |\Delta Q_H|, \quad \|\Delta \tpi^{rr}\| \lesssim \delta^{-2}|\Delta Q_M|, \quad \|\Delta \tpi^{Ar,AA}\| \lesssim |\Delta Q_M|, \]

For the Hamiltonian constraint case, the worst term is $\tsigma^{Ar}\tsigma^{Ar}$, where we lose $\delta^{-2}$ from the conversion and $\delta^{-2}$ from the solution operator, with cokernel of size $O(1)$ for $\tsigma$. For $M(\pi,\pi)$, all terms are treated the same. For example, $\tpi^{AA}\tpi^{AA}$ loses $\delta^{-4}$ in the conversion, but the $AA$ component solution operator is $O(\delta^{-2})$ and cokernel is size $O(1)$; hence in converting back from $\tsigma$ to $F_Q$ we find that we have a bound
\[\delta^{-4}(\delta^4 \ep \|\Delta \tsigma\|) + \delta^{-4}(\delta^{5} \ep \|\Delta \tpi\|) \lesssim (K_F + \delta^{-1}K_G) \ep |\Delta Q|,\]
where $K_F$ and $K_G$ denote the Lipschitz constants of $F_Q$ and $G_Q$, respectively.

Similarly, for the momentum constraint the worst terms are $\tsigma^{Ar}\tpi^{rr}$ and $\tsigma^{Ar}\tpi^{AA}$ paired with the angular cokernel which has size $\delta^{-2}$.
For the first term, we start with $\delta^4 \|\Delta \tpi^{rr}\|$, lose $0$ from the conversion, $2$ from the cokernel, and $4$ from the solution operator.
For the latter, we start again with $\delta^4 \|\Delta \tpi^{AA}\|$, lose $2$ from the conversion, $2$ from the cokernel, and $2$ from the solution operator.

For the other term, the problematic term is $\tpi^{rr}\tsigma^{Ar}$ we start with $\delta^3 \|\Delta \tsigma^{Ar}\|$, lose nothing from the conversion, but $2$ from the cokernel, and $2$ more from the solution operator.
Hence we have a bound
\[\delta^{-4}(\delta^4 \ep \|\Delta \tpi\|) + \delta^{-2}(\delta^{3} \ep \|\Delta \tsigma\|) \lesssim (\delta^{-2} K_G + \delta^{-1}K_F) \ep |\Delta Q|. \]

Since the natural sizes of $K_F$ and $K_G$ are $O(1)$ and $O(\delta^2)$ respectively, we can arrange everything into a Lipschitz matrix of the form
$\ep \begin{pmatrix}
 1 & \delta \\
 \delta^{-1} & 1
\end{pmatrix}$, which suggests the weighted charge norm $|Q|_w = |Q|_H + \delta |Q|_M$, in which the Lipschitz constant $K$ becomes independent of $\delta$.

For the former, remember that the Hamiltonian nonlinear terms were all bounded by $\delta^6 \ep^2$ in $H^k(C)$ norm with cokernel $O(1)$, while for the momentum constraint the nonlinear terms were bounded by $\delta^{7}\ep^2$ and the components of the cokernel and their derivatives are all $O(\delta^{-2})$ and uniformly bounded on $\widetilde{C}$; in both cases, we have a uniform $\delta^6 \ep^2$ bound in the weighted charge norm in the momentum case.
Importantly, $\tsigma$ dominates all contributions, since $\bsigma$ is substantially smaller. 
Hence the $L^{\infty}$ and $H^1 \times L^2(C)$ of $(\sigma_Q, \pi_Q)$ suffice.
As remarked in \cite{mao2023}, the quadratic terms where one of the factors is $\pdv_{z^i} \pdv_{z^j} \sigma_Q^{ij}$ can be handled simply by integration by parts.


Let us now for simplicity focus on exterior gluing, where interior gluing simply flips the in and out parameters. 
Then $\Qf[(\sigma_{in;Q}, \pi_{in;Q});C] = \cQ$, while $\Qf[(\sigma_{out;Q}, \pi_{out;Q});C] = Q$.

Hence we need to solve the fixed point problem
\[Q = \cQ + n[Q].\]
Note that in this problem we see no application of the solution operator $S$, and hence no loss of factors of $\delta$.
Using that $|n[Q]|_w \lesssim \delta^6 \ep^2$, $||n[\cdot]||_{Lip,w} \lesssim K\ep$, and the fact that $\text{dist}_w{(\cQ, \pdv \Qcal)} > \delta^6 M_c \ep^2$, the map is a contraction inside of $\Qcal$ and thus there is a unique $Q \in \Qcal$ satisfying the fixed point argument. 
This concludes the proof of the theorem.

\end{proof}

Let us now prove Proposition \ref{prop:properties_linear}.

\begin{proof}[Proof of \ref{prop:properties_linear}]
First, we note that the Lipschitz dependence is a consequence of the fixed point iterations we used in the above proof.

For persistence of regularity, it suffices to bound the derivatives $(\pdv^{\alpha} \sigma_Q, \pdv^{\alpha}_z \pi_Q)$ in $H^s \times H^{s-1}(C)$, where $|\alpha| \leq m$.

To this end, let us observe that, as in \cite{mao2023} differentiating the fixed point Equation~\eqref{eqn:fixed_point_sigmapi} yields 
\begin{equation}\label{eqn:persist_reg}
\begin{aligned}
&(\pdv^{\alpha}_z \tsigma, \pdv^{\alpha}_z \tpi) - S(D_{(\tsigma,\tpi)} (F + M, G + N)[(\pdv^{\alpha}_z \tsigma, \pdv^{\alpha}_z \tpi)]) \\
&= S[\pdv^{\alpha}_z (F,G)] + [\pdv^{\alpha},S](F+M,G+N)(\tsigma,\tpi) + (E_{\alpha,F+M}, E_{\alpha,G+N})(\tsigma,t\pi)
\end{aligned}
\end{equation}
where $D_{(\tsigma,\tpi)} (F + M, G + N)[(\pdv^{\alpha}_z \tsigma, \pdv^{\alpha}_z \tpi)]$ is the linearization of $F+M$ around $(\tsigma,\tpi)$ applied to $(\pdv^{\alpha}_z \tsigma, \pdv^{\alpha}_z \tpi)$, i.e. the term where all of the $\alpha$ derivatives fall on $(\tsigma,\tpi)$.
In addition, $(E_{\alpha,F+M}, E_{\alpha,G+N})$ refers to lower order error terms with mixed derivative components. 
These will be negligible due to fewer derivatives (and hence fewer factors of $\delta$) when acting upon $S$.

Applying the commutator estimates at the end of Step 3, we find that the right hand side of Equation~\eqref{eqn:persist_reg} yields
\begin{align*}
||\text{RHS of} \hspace{0.1cm} (\ref{eqn:persist_reg})||_{H^s \times H^{s-1}(C)} \lesssim \frac{1}{\delta^{6}} ||(\bsigma, \bpi)||_{H^{s+m} \times H^{s+m-1}} + (\text{quadratic terms}).
\end{align*}
Here, the factor of $1/\delta^6$ comes from the momentum case only; for the Hamiltonian case, we have the loss of $1/\delta^2$ from the solution operator, and a separate loss of $1/\delta^2$ from the hierarchy and returning from $F_Q$ to $\bsigma_Q$.
From the momentum case: we get $O(1/\delta^4)$ from the commutator estimates, and then another factor of $1/\delta^2$ from converting $G_Q$ back into $\bpi_Q$.

The quadratic terms will be lower order in the factors of $\delta$, due to the presence of the term of the form $||(\tsigma, \tpi)||_{H^{s+m-1} \times H^{s+m-2}}$ which appears when applying Gagliardo-Nirenberg.

To obtain the other bound, observe that 
\begin{align*}
||(\pdv^{\alpha}_z \tsigma, \pdv^{\alpha}_z \tpi)||_{H^s \times H^{s-1}(C)} &\lesssim  ||(\pdv^{\alpha}_z \tsigma, \pdv^{\alpha}_z \tpi)|| + ||S(D_{(\tsigma,\tpi)} (F + M, G + N)[(\pdv^{\alpha}_z \tsigma, \pdv^{\alpha}_z \tpi)])||_{H^s \times H^{s-1}(C) }  \\
&\lesssim ||(\pdv^{\alpha}_z \tsigma, \pdv^{\alpha}_z \tpi)|| + \widetilde{B}||(D_{(\tsigma,\tpi)} (F + M, G + N)[(\pdv^{\alpha}_z \tsigma, \pdv^{\alpha}_z \tpi)])||_{H^s \times H^{s-1} } \\
&\lesssim ||(\pdv^{\alpha}_z \tsigma, \pdv^{\alpha}_z \tpi)|| + \delta^{-2}\ep||(\pdv^{\alpha}_z \tsigma, \pdv^{\alpha}_z \tpi))||_{H^s \times H^{s-1} }
\end{align*}
where we observe that the worst componentwise factor takes $\tsigma^{Ar}\tpi^{AA}$ to $\tpi^{rr}$, giving a total of $\delta^{4-2-4}=\delta^{-2}$.
Hence using $\ep=\delta^2\ep_c$, we get the desired coercivity which completes the proof.
\end{proof}

\appendix
\section{Deriving Trace-Shifted Equations}\label{sec:derive_sigma_pi}

Here, we detail the derivation of Equation~\eqref{eqn:constraints}.

To begin with, let us apply a similar procedure to take a trace-shifted difference of our prescribed metrics.
Note that this will not be trace-free, as we are trying to eliminate the Laplacian term. 

For $(g,k)$ sufficiently close to $(b,0)$, let us introduce the following substitutions:
\begin{equation}
\sigma_{ij} = g_{ij} - b_{ij} - b_{ij}\tr_b(g-b); \qquad \pi_{ij} = k_{ij} - b_{ij} \tr_b k
\end{equation}

Studying the linearized scalar curvature independently, let $h = g-b$. 
Then we have
\begin{align*}
DR\vert_b [h] = -\Delta_b(\tr_b h) + \dv_b(\dv_b h) + (d-1)\tr_b h \\
\end{align*}
Note that
\begin{align*}
\dv_b(\dv_b h) = \nabla^i_b \nabla^j_b h_{ij} =  \nabla^i_b \nabla^j_b \sigma_{ij} +   \nabla^i_b \nabla^j_b b_{ij} \tr_b h = \nabla^i_b \nabla^j_b \sigma_{ij}+ b^{ik}b^{jl} \nabla_{k,b} \nabla_{l,b} b_{ij} \tr_b h
\end{align*}
and
\[b^{ik}b^{jl} \nabla_{k,b} \nabla_{l,b} b_{ij} \tr_b h = b^{ik}b^{jl} b_{ij} \nabla_{k,b} \nabla_{l,b} \tr_b h = b^{ik}\delta^l_i \nabla_{k,b} \nabla_{l,b} \tr_b h = b^{lk} \nabla_{k,b} \nabla_{l,b} \tr_b h \]
But this is just $\nabla_{k,b} \nabla^{k,b} \tr_b h = \Delta_b (\tr_b h)$.

Thus we have shown that
\[DR\vert_b [h] = \nabla^i_b \nabla^j_b \sigma_{ij} + (d-1) \tr_b h \]

When linearizing, we have
\[R(g) = R(b) + DR \vert_b [h] + Q_R(h)\]
where the last piece contains nonlinear terms in $h$.
But $R(b) = -n(n-1)$, so this becomes
\begin{align*}
DR \vert_b [h] &= R(g) + d(d-1) - Q_R(h) \\
&= |k|^2_g - (\tr_g k)^2
\end{align*}

Now if $\pi_{ij} = k_{ij} - b_{ij}\tr_b k$, we can write
\[\tr_b \pi = b^{ij}\pi_{ij} = b^{ij}k_{ij} - b^{ij}b_{ij}\tr_b k = -(d-1) \tr_b k \]

and
\[ |\pi|^2_b = \pi^{ij}\pi_{ij} = |k|^2_b + (d-2)(\tr_b k)^2 \]

So that 
\[|k|^2_b - (\tr_b k)^2 = |\pi|^2_b - \frac{(d-2)}{(d-1)^2}(\tr_b \pi)^2 - \frac{1}{(d-1)^2} (\tr_b \pi)^2 = |\pi|^2_b - \frac{1}{(d-1)} (\tr_b \pi)^2 \]

Now we need to estimate the difference
\[|k|^2_g - (\tr_g k)^2 - \left(|k|^2_b - (\tr_b k)^2\right) \]

To do this, we use the Schematic notation as in Definition \ref{defn:schematic}.

With this notation, we notice that the difference described includes terms of the form $A(h)^{ij}k_ik_j$ where $A$ is some algebraic function of $h$ with no derivatives, so we can write 
\[|k|^2_g - (\tr_g k)^2 =  \left(|k|^2_b - (\tr_b k)^2\right) + M^{(0)}_h(k,k)\]

And as in \cite{mao2023}, expanding our nonlinear piece $Q_R(h)$ gives expressions in terms of the difference of Christoffel symbols $\Gamma(g) - \Gamma(b)$, arising from the expansion of the scalar curvature as 
\[R(g) = g^{ij}(\partial_k \Gamma^k_{ij} - \partial_j \Gamma^k_{ik} + \Gamma^k_{ij}\Gamma^l_{kl} - \Gamma^l_{ik}\Gamma^k_{jl})\]

Assuming that the differences are sufficiently close, these terms will therefore look like 
\[M^{(1)}_h(\nabla_b h, \nabla_b h), \qquad M^{(2)}_h(h, \nabla_b^2 h)\] 
from the last two terms, and the first two terms, respectively.

Moreover, we can write $k$ in terms of $\pi$ as 
\[k_{ij} = \pi_{ij} - \frac{1}{d-1} b_{ij} \tr_b \pi\]
we can then rewrite this as 
\[|k|^2_g - (\tr_g k)^2 =  \left(|k|^2_b - (\tr_b k)^2\right) + M^{(0)}_h(\pi,\pi) = M^{(0)}_h(\pi,\pi)\]
since all of the relevant terms are of this form. 

Putting everything together, we see that our equation in reformulated coordinates reads
\begin{equation}\label{eqn:ham}
\nabla^i_b \nabla^j_b \sigma_{ij} + (d-1) \tr_b h = M^{(0)}_h(\pi,\pi) + M^{(1)}_h(\nabla_b h, \nabla_b h) + M^{(2)}_h(h, \nabla_b^2 h)
\end{equation}

Similarly for the Momentum constraint, notice that if we define
\[\widehat{\pi}_{ij} = k_{ij} - g_{ij} \tr_g k\]
we see that the momentum constraint becomes $\nabla^i_g \widehat{\pi}_{ij}=0$.

Plugging in $\pi$, we get
\[\widehat{\pi}_{ij} = \pi_{ij} + (b_{ij}-g_{ij})\tr_b k - g_{ij}(\tr_g k - \tr_b k)\]
and consequently 
\[\nabla^i_g \widehat{\pi}_{ij} = \nabla^i_g[\pi_{ij}] + \nabla^i_g[(b_{ij}-g_{ij})\tr_b k - g_{ij}(\tr_g k - \tr_b k)]\]
There are two types of error terms that show up, one involving changing from $\nabla^i_g$ to $\nabla^i_b$, and one is from the terms specified here.
Thus we can use the same schematic notation to write
\begin{equation}\label{eqn:mom}
\nabla^i_b \pi_{ij} = N^{(1)j}_h(h, \nabla_b \pi) + N^{(0)j}_h(\nabla_b h, \pi)
\end{equation}

Equations \eqref{eqn:ham} and \eqref{eqn:mom} are defined in terms of a ``trace-shifted" metric $\sigma$ and the metric $h=g-b$.
This is actually redundant; instead, we modify all of our terms to simply depend on $\sigma$ and the quantity $\tau = \tr_b h$. 
Notice however that $\tau = \tr_b h  = b^{ij}\sigma_{ij} + b^{ij}b_{ij} \tr_b h= \tr_b \sigma + d\tau$, so $\tr_b \sigma = (1-d)\tau$.
Moreover, the shift will not change the behavior of the schematic notation, as it will just readjust the expression and constants but not change the boundedness assumptions.
Thus, all quantities can be rewritten in terms of sigma as:
\begin{align*}
\begin{cases}
\nabla^i_b \nabla^j_b \sigma_{ij} - \tr_b \sigma = M^{(0)}_{\sigma}(\pi,\pi) + M^{(1)}_{\sigma}(\nabla_b \sigma, \nabla_b \sigma) + M^{(2)}_{\sigma}(\sigma, \nabla_b^2 \sigma) \\
\nabla^i_b \pi_{ij} = N^{(1)j}_{\sigma}(\sigma, \nabla_b \pi) + N^{(0)j}_{\sigma}(\nabla_b \sigma, \pi)
\end{cases}
\end{align*}
The reason we still need to carry the trace around is because the limiting curvature is not zero, and eliminating the Laplacian term comes at the price of also holding on to $\tr_b \sigma$.

\section{Solution Operator Theorems}\label{sec:soln_op_thms}
In this section, we state the theorems that we will need to use in Step 3 of the proof of \ref{thm:gluing_linear}.

First, we need to define some constants. 

We require $A_S$ so that 
\[|L_y^{|\alpha|}|y_1-y|^{|\beta|}(\pdv_y + \pdv_{y_1})^{\alpha} \pdv_{y_1}^{\beta} S_J^{\gamma,K}(y,y_1,s) | \leq A_S |y_1-y|^{m_K + |\gamma|}s^{m_K + |\gamma|-1} \]
where $L_y^{|\alpha|}$ is the characteristic length scale of $y$-derivatives and $|\alpha|+|\beta| \leq M_S$ for some $M_S \in \N$. 
Then $A_x$ is defined so that 
\[(1+A_x)^{-1}|y_1-y| \leq |\pdv_s x(y,y_1,s)| \leq (1+A_x)|y_1-y|,\]
and $A_{\eta}$ is defined so that
\begin{equation*}
	|L_{y}^{|{\alpha}|} |y_{1}-y|^{|\beta|} (\pdv_{y} + \pdv_{y_{1}})^{\alpha} \pdv_{y_{1}}^{\beta}\eta(y, y_{1})| \leq A_{\eta} L_{\eta}^{-d} \hbox{ for } |\alpha| + |\beta| \leq M_{\eta}.
\end{equation*}
and
\begin{align*}
	|L_{y}^{|\alpha|} L_{b}^{|\beta|} (\pdv_{y} + \pdv_{y_{1}})^{\alpha} \pdv_{y_{1}}^{\beta} \eta(y, y_{1})| \leq A_{\eta} L_{\eta}^{-d} \quad \hbox{ for } |\alpha| + |\beta| \leq M_{\eta}'.
\end{align*}
With constants $M_{\eta},M_{\eta}' \in \N$.

We will not always be so fortunate to eliminate the $y$-dependence of our smooth averaging function $\eta = \eta(y_1)$.
As a result, we do not get a nice formula for the correction in terms of the cokernel, but rather a smoothing operator $B_{\eta}$.

The following theorem is useful to derive a solution operator bound in the setting where $\eta$ has $y$-dependence.

\begin{theorem}[Theorem 5.11 in \cite{isett2025integral}] \label{thm:smth_ker_5.11}
Let $x : W \times [0, 1] \to \R^{d}$ and $\eta : W \to \R$, $\Pcal$ and $\tensor{S}{_{J}^{(\gamma, K)}}(y, y_{1}, s)$ satisfy the above bounds for $A_x,A_{\eta}$, and $A_S$. Assume also that
\[\overline{U} \cap \overline{\cup_{y \in U} \supp_{y_{1}} \eta(y, y_{1})} = \emptyset.\]
Then the integral kernel $K_{\eta} = \tensor{(K_{\eta})}{^{K}_{J}}$ defined by 
\begin{equation} \label{eq:smth-avg-ker}
	\int \tensor{(K_{\eta})}{^{K}_{J}}(x, y) \psi_{K}(x) \, \dd x
	= \int \int_{0}^{1} \sum_{\gamma} \tensor{S}{_{J}^{(\gamma, K)}}(y, y_{1}, s) \eta(y, y_{1}) \pdv_{x}^{\gamma} \psi_{K}(x(y, y_{1}, s)) \, \dd s \, \dd y_{1},
\end{equation}
satisfies
\begin{equation*}
	\Pcal K_{\eta}(\cdot, y) = \delta_{0}(\cdot-y) - b_{\eta}(\cdot, y) \quad \hbox{ on } U,
\end{equation*}
where $\tensor{(b_{\eta})}{^{J'}_{J}}(x, y) = 0$ in Case~(1), and in Case~(2), $\tensor{(b_{\eta})}{^{J'}_{J}}(x, y)$ is a distribution on $\R^{d}_{x} \times \R^{d}_{y}$ defined by
\begin{align*}
\tensor{(b_{\eta})}{^{J'}_{J}}(x, y) &\coloneqq \sum_{\mathbf{A} \in \Acal} \int (Z^{\mathbf{A}})_{J}(y, y_{1}) \eta(y, y_{1}) (g_{\mathbf{A}})^{J'}(x, y_{1}) \, \dd y_{1} \\
&= \sum_{\mathbf{A} \in \Acal} \sum_{\alpha} (-1)^{|\alpha|} \pdv_{x}^{\alpha} \left( c[g_{\mathbf{A}}]^{(\alpha, J')}(x) (Z^{\mathbf{A}})_{J}(y, x) \eta(y, x) \right),
\end{align*}
which moreover satisfies
\begin{align*}
	|L_{y}^{|\alpha|} |x-y|^{|\beta|} (\pdv_{y} + \pdv_{x})^{\alpha} \pdv_{x}^{\beta} \tensor{(b_{\eta})}{^{J'}_{J}}(x, y)| & \leq C_{\# \Acal, m_{K}+m'_{K}}  A_{g} A_{Z} A_{\eta} \quad \hbox{ for all } x \neq y, \\
	\supp_{x} \tensor{(b_{\eta})}{^{J'}_{J}}(x, y) &\subseteq \supp_{y_{1}} \eta(y, y_{1}), \quad \hbox{ for all } y \in \R^{d},
\end{align*}
where $|\alpha| + |\beta| \leq \min\{M_{x}, M_{Z}, M_{g}, M_{\eta}\} - C_{m_{K}, m_{K}'}$.

In both cases, the integral kernel $K_{\eta}(x, y)$ is a locally integrable function on $\R^{d} \times \R^{d}$ satisfying
\begin{align*}
	|\tensor{(K_{\eta})}{^{K}_{J}}(x, y)| &\leq C_{m_{K}, m_{K}'} A_{S} A_{\eta} (1+A_{x})^{2(d+m_{K'})} |x-y|^{-d+m_{K}} \quad \hbox{ for all } x \neq y, \\
	\supp \tensor{(K_{\eta})}{^{K}_{J}}(\cdot, y) &\subseteq \overline{\bigcup_{y_{1} : (y, y_{1}) \in \supp \eta} x(y, y_{1}, [0, 1])}, \\
	\supp \tensor{(K_{\eta})}{^{K}_{J}}(x, \cdot) &\subseteq \overline{\{y \in \R^{d} : \exists y_{1} \in \R^{d} \hbox{ s.t. } (y, y_{1}) \in \supp \eta, \, x \in x(y, y_{1}, [0, 1])}\},
\end{align*}
and the symbol $\tensor{(a_{\eta})}{^{K}_{J}}(\xi, y) = \int \tensor{(K_{\eta})}{^{K}_{J}}(y+z, y) e^{- i \xi \cdot z} \, \dd z$ obeys
\begin{align*}
	|\pdv_{y}^{\alpha} \pdv_{\xi}^{\beta} \tensor{(a_{\eta})}{^{K}_{J}}(\xi, y)| \leq C_{m_{K}, m_{K}', \alpha, \beta} A_{S} A_{\eta} (1 + A_{x})^{c (d+|\alpha|+|\beta|+m_{K}+m_{K}')} L_{y}^{-|\alpha|} (((1+A_{x}) L_{\eta})^{-1} |\xi|)^{-m-|\beta|} 
\end{align*}
where $y \in \R^{d}$, $|\alpha| + |\beta| \leq \min \{M_{x}, M_{S}, M_{\eta}\} - C_{m_{K}, m_{K}'}$.
\end{theorem}

We note that as a result of Theorem~\ref{thm:smth_ker_5.11}, we have the solution operator bound
\begin{equation*}
\|\tensor{(S_{\eta})}{^{K}_{J}}\|_{W^{s, p} \to W^{s+m_{K}, p}} \leq C_{d, m_{K}, m_{K}', s, p} A_{S} A_{\eta} (1 + A_{x})^{C (d + m_{K} + m_{K'} + s)}.
\end{equation*}

In addition, in this general setting there is also a corresponding bound on the smoothing operator $\Bcal_{\eta}$, provided we have additional assumptions:

\begin{enumerate}[label=($b_{y_{1}}$-\arabic*), series=by1condition]
\item {\bf Structure of $b_{y_{1}}$, quantitative.} \label{hyp:b-y1-q}
For every $(y, y_{1}) \in W$, there exists a distribution $b_{y_{1}}(\cdot, y)$ of the form
\begin{equation}
\tensor{(b_{y_{1}})}{^{J'}_{J}}(x, y) = \sum_{\mathbf{A} \in \Acal} (Z^{\mathbf{A}})_{J}(y, y_{1}) (g_{\mathbf{A}})^{J'}(x, y_{1})
\end{equation}
for some finite index set $\Acal$ (independent of $y, y_{1}$), which satisfies the relation 
\begin{equation} \label{eq:crc-x-trunc}
	\varphi_{J}(y) = \int_{0}^{1} \sum_{\gamma} \tensor{S}{_{J}^{(\gamma, K)}}(y, y_{1}, s) (\pdv_{x}^{\gamma} \Pcal^{\ast} \varphi)(x(y, y_{1}, s)) \, \dd s + \langle\tensor{(b_{y_{1}})}{^{J'}_{J}}(\cdot, y), \varphi_{J'}\rangle,
\end{equation}
for every $(y, y_{1}) \in W$ with $y \in U$. Moreover, each $Z^{\mathbf{A}}$ obeys
\begin{align*}
	|L_{y}^{|\alpha|} |y- y_{1}|^{|\beta|} (\pdv_{y} + \pdv_{y_{1}})^{\alpha} \pdv_{y_{1}}^{\beta} Z^{\mathbf{A}}(y, y_{1})| \leq A_{Z} \quad \hbox{ for } |\alpha| + |\beta| \leq M_{Z},
\end{align*}
and each $g_{\mathbf{A}}$ takes the form (for $J'=1, \ldots, r_{0}$)
\begin{align*}
(g_{\mathbf{A}})^{J'}(x, y_{1}) = \sum_{\alpha} c[g_{\mathbf{A}}]^{(\alpha, J')}(y_{1}) \pdv^{\alpha} \delta_{0}(x-y_{1}),
\end{align*}
where the coefficients $c[g_{\mathbf{A}}]^{(\alpha, J')}$ are zero except if $|\alpha| \leq \max_{K} (m_{K} + m'_{K})-1$ and obey
\begin{equation*}
	|\pdv_{y_{1}}^{\beta} c[g_{\mathbf{A}}]^{(\alpha, J')}| \leq A_{g} \quad \hbox{ for } |\beta| \leq M_{g}.
\end{equation*}

\item {\bf Local regularity of $Z^{\mathbf{A}}$.} \label{hyp:b-y1-smth}
\ref{hyp:b-y1-q} holds with $Z^{\mathbf{A}}$ obeying the following additional bounds:
\begin{align*}
	|L_{y}^{|\alpha|} L_{b}^{|\beta|} (\pdv_{y} + \pdv_{y_{1}})^{\alpha} \pdv_{y_{1}}^{\beta} Z^{\mathbf{A}}(y, y_{1})| \leq A_{Z}  \quad \hbox{ for } |\alpha| + |\beta| \leq M_{Z}'
\end{align*}
\end{enumerate}

We note as in \cite[Section 6]{isett2025integral} that our previous assumptions on $A_x$ actually suffice for both of these conditions to hold, so there is nothing in addition to prove.

\begin{lemma}[Lemma 5.17 in \cite{isett2025integral}]\label{lem:b_eta_5.17}
Assume that $U,W,x,\eta,P,S,$ and $b_{y_1}$ satisfy smoothness assumptions $(C^{\infty})$.
Moreover, assume that $\overline{U}$ is $x$-star-shaped with respect to $\supp \eta$ in the sense that 
\[\overline{\bigcup_{(y,y_1) \in \supp \eta} x(y,y_1,[0,1])} \subseteq \overline{U}.\]
Then the Bogovskii-type solution operator of Theorem~\ref{thm:sol_op_5.15} $S_{\eta}$ satisfies $PS_{\eta}f=f-{\Bcal}_{\eta}f$, where
\par
{\bf $\Bcal_{\eta}$ is smoothing.} The integral kernel $\tensor{(b_{\eta})}{^{J'}_{J}}(x, y)$ for $\Bcal_{\eta}$ is smooth and satisfies
\[|L_y^{|\alpha|}L_b^{|\beta|}(\pdv_y + \pdv_x)^{\alpha}\pdv_x^{\beta}\tensor{(b_{\eta})}{^{J'}_{J}}(x, y)| \leq C_{\# \Acal, m_K + m_K'}A_g A_Z A_{\eta},  \]
\[\supp_x \tensor{(b_{\eta})}{^{J'}_{J}}(x, y) \subseteq \supp{y_1} \eta(y,y_1), \qquad \text{for all}\hspace{0.1cm} y \in U, \]
where $A_g, A_Z, A_{\eta}$ correspond to $M_g, M_Z, M_{\eta}, M_Z', M_{\eta}' \geq |\alpha|+|\beta| + C_{d, m+m_1}$.
\end{lemma}

Note that $(C^{\infty})$ is a large collection of assumptions, but these will typically be easy to satisfy in our case with straight line paths $x$ and smoothing operators $\eta$ localized to a unit-sized ball. 
For straight-line paths, for example, $A_x = 1$ which translates to $O(1)$ bounds on $A_g$ and $A_Z$ as well.

We also note that $\Bcal_{\eta}$ can be approximated by finite-rank operators, from which the authors of \cite{isett2025integral} obtain the following, more general version of Theorem~\ref{thm:sol_op_5.15}: 

\begin{theorem}[Theorem 1.2 in \cite{isett2025integral}]\label{thm:sol_op_existence}
Let $U \subset \R^d$ be a connected, bounded open subset. Let $P$ satisfy the recovery on curves condition for an admissible family of curves $x(y,y_1,s)$, meaning that one can linearly  recover the value of $\phi$ at $y$ in terms of its jet at $y_1$ and the jet of $P^*\phi$ along $x(y,y_1, \cdot)$.
Suppose that $\overline{U}$ is $x$-star shaped with respect to $\overline{U}_1$, where $\overline{U}_1 \subseteq U$ and $U_1$ is an open subset of $U$ such that
\[\overline{\bigcup_{y \in \overline{U}, y_1 \in \overline{U_1}} x(y,y_1, [0,1])} \subseteq U.\]
Then up to a finite-dimensional cokernel, there exists an integral kernel $\tilde{K}(x,y)$ such that $\tilde{S}f(x) = \int_U \tilde{K}(x,y) f(y)dy$ for $f \in C^{\infty}_c(U; \C^{s_0})$ satisfies
\[P\tilde{S}f = f - \sum_{A \in \{1, \dots, \dim \ker P^* \}} w_A \langle Z^A, f \rangle \qquad \text{for all} \hspace{0.2cm} f \in C_c^{\infty}(U; \C^{s_0}).\]
for $w_A \in C_{c}^{\infty}(U_1; \C^{r_0})$.
\end{theorem}

We see that this theorem holds for connected, bounded open subsets of $\R^d$ such that $\overline{U}$ is $x$-star shaped with respect to some $\overline{U}_1$.
One complication is that the thin hyperbolic annulus by itself is not even convex, so one must be clever in how to construct the Solution operators.
One approach is given in the proof of Lemma 2.2 in \cite{mao2023} by constructing a partition of unity over finitely many convex ball subsets, constructing a family of solution operators $S_i$ on the $i$th ball, and declaring that $Sf = \sum_i S_i f_i$ is a valid solution operator up to finite dimensional cokernel, where requiring them to agree on the intersections reduces to a linear algebra problem. 
It is worth mentioning that this only works due to the flexibility of having a full cokernel, which makes the more general problem of metric gluing substantially more difficult.
 
We finally consider the specialized case when $\eta = \eta(y_1)$ is independent of $y$, in which case there is a stronger result

\begin{theorem}[Theorem 5.15 in \cite{isett2025integral}, Full solution operator and Poincar\'e-type inequality, completely integrable case]\label{thm:sol_op_5.15}
Assume that $U,W,x, \eta,P,S, b_{y_1}$ satisfy the smoothness conditions $(C^{\infty})$.
Assume furthermore that $\eta = \eta(y_1)$ and that $\overline{U}$ is $x$-star-shaped with respect to $\supp_{\eta}$, i.e. 
\[\bigcup_{y \in \overline{U}, y_1 \in \supp \eta} x(y,y_1,[0,1]) \subseteq \overline{U}. \]
If the formal cokernel of $P$ satisfies $\dim \ker P^* = \#\Acal$, then there is a bounded operator $S_{\eta}: H^s(U) \to H^{s+m_K}(U)$ for every $1 < p < \infty$, and moreover we have
\[PS_{\eta}f = f - \sum_{\Af \in \Acal} (\zeta_{\eta})_{\Af} \langle \Zf^{\Af},f\rangle. \]
for all $f \in H^s(U)$, where $(\zeta_{\eta})_{\Af}$ is a smooth function with $\supp(\zeta_{\eta})_{\Af} \subseteq \supp \eta$
with $\{\Zf^\Af\}_{\Af \in \Acal}$ a basis of $\ker P^*$ and $\zeta_{\Af}$.
In addition, 
\[\supp S_{\eta}f \subseteq \overline{\bigcup_{y \in \supp f, y_1 \in \supp \eta}x(y,y_1,[0,1])} \subseteq \overline{U},\]
and $||\tensor {(S_{\eta})}{^K_J}||_{H^s(U) \to H^{s+m_K}(U)}$ is bounded by $C_{d,m_K,m_K',s,p}A_S A_{\eta}(1+A_x)^{C(d+m_K+m_K'+s)},$ with $A_S, A_{\eta}$, and $A_x$ corresponding to $M_x$, $M_{\eta}$, $M_S \geq s + C_{d,m_K,m_K'}$.
\end{theorem}


\nocite{*} 

\printbibliography 

\end{document}